\documentclass[12pt,a4paper]{amsart}
\usepackage{amsmath,amssymb,amsthm,bm}
\usepackage{latexsym}
\usepackage{amsfonts}
\usepackage{bbm,dsfont}
\usepackage{eucal}
\usepackage{eufrak}
\newcommand{\la}{\langle}
\newcommand{\ra}{\rangle}
\newcommand{\om}{\omega}
\newcommand{\mc}[1]{\mathcal{#1}}
\newcommand{\mf}[1]{\mathfrak{#1}}
\newcommand{\mb}[1]{\mathbb{#1}}
\newcommand{\mr}[1]{\mathrm{#1}}
\newcommand{\ms}[1]{\mathsf{#1}}
\newcommand{\vek}[1]{\mathbf{#1}}
\newcommand{\bs}[1]{\boldsymbol{#1}}
\newcommand{\ovl}[1]{\overline{#1}}
\newcommand{\tr}[1]{\mathrm{tr}[#1]}
\newcommand{\sis}[2]{\la #1|#2\ra}
\newcommand{\Sis}[2]{\big\la #1\big| #2\big\ra}

\newcommand{\hil}{\mc H}
\newcommand{\kil}{\mc K}
\newcommand{\lh}{\mc{L(H)}}
\newcommand{\lk}{\mc{L(K)}}
\newcommand{\f}{\varphi}

\newcommand{\tj}{\vartheta}
\newcommand{\Om}{\Omega}
\renewcommand{\a}{\alpha} %PROJEKTIIVINEN KERNELI
\newcommand{\m}{\sigma} %2-KOSYKLI
\newtheorem{ma}{Definition}
\newtheorem{lemma}{Lemma}
\newtheorem{prop}{Proposition}
\newtheorem{theor}{Theorem}
\newtheorem{cor}{Corollary}
\newtheorem{rem}{Remark}
\newtheorem{ex}{Example}

\begin{document}

\title{Covariant KSGNS construction\\ and quantum instruments}

\author{Erkka Haapasalo}
\address{Department of Physics and Astronomy, University of Turku, FI-20014 Turku, Finland}  \email{ethaap@utu.fi}

\author{Juha-Pekka Pellonp\"a\"a}
\address{Department of Physics and Astronomy, University of Turku, FI-20014 Turku, Finland}  \email{juhpello@utu.fi}

\begin{abstract}
We study positive kernels on $X\times X$, where $X$ is a set equipped with an action of a group, and taking values in the set of $\mc A$-sesquilinear forms on a (not necessarily Hilbert) module over a $C^*$-algebra $\mc A$. These maps are assumed to be covariant with respect to the group action on $X$ and a representation of the group in the set of invertible ($\mc A$-linear) module maps. We find necessary and sufficient conditions for extremality of such kernels in certain convex subsets of positive covariant kernels. Our focus is mainly on a particular example of these kernels: a completely positive (CP) covariant map for which we obtain a covariant minimal dilation (or KSGNS construction). We determine the extreme points of the set of normalized covariant CP maps and, as a special case, study covariant quantum observables and instruments whose value space is a transitive space of a unimodular type-I group. As an example, we discuss the case of instruments that are covariant with respect to a square-integrable representation.
\end{abstract}

\maketitle

\section{introduction}
\label{sec:intro}

In standard quantum mechanics, instruments are the appropriate mathematical representatives of quantum measurement processes \cite{davies, kraus, Pe13a, Pe13b}. An instrument is described as a map $\Gamma:\mc L(\mc K)\times\Sigma\to\mc L(\hil)$ that takes pairs of outcome sets of the measurement, with the outcome space modelled by a measurable space $(\Om,\Sigma)$, $\Sigma\subset\mc P(\Om)$, and bounded operators on the output Hilbert space $\mc K$ of the measurement into bounded operators on the input Hilbert space $\hil$ such that as a set function the instrument is weakly $\sigma$-additive and $\Gamma(\cdot,X)$ is completely positive (CP) and normal ($\sigma$-weakly continuous) for any $X\in\Sigma$. In the pre-dual picture, the instrument maps the input quantum state $\rho$ (trace-1 positive operator on $\hil$) into a sub-normalized output state $\tilde\Gamma(\rho,X)=[\Gamma(\cdot,X)]_*(\rho)$ on the output quantum system modelled by $\mc K$. This output state is interpreted as the conditional transformed state after the measurement conditioned by registering an outcome in $X$. Moreover, the probability of detecting a measurement outcome in $X$ when the pre-measurement state of the system was $\rho$ is $\tr{\tilde\Gamma(\rho,X)}$. The instrument $\Gamma$ contains a quantum observable $\ms M_\Gamma$, i.e.,\ a positive operator measure and a channel $\mc E_\Gamma$ as its marginals,
$$
\ms M_\Gamma=\Gamma(I_{\mc K},\cdot),\qquad \mc E_\Gamma=\Gamma(\cdot,\Om).
$$
Here, $\ms M_\Gamma$ is the actual physical quantity to be measured on the system and $\mc E_\Gamma$ corresponds to the unconditioned over-all state transformation induced by the measurement. Thus quantum observables, channels, and instruments have a central role in the quantum theory of measurement \cite{BuLaMi} and CP maps in general are involved in many fields of mathematics and physics \cite{bhatskeide2000, heo, kraus}.

In extensions of quantum theory, particularly in quantum field theories, generalizations of these CP maps are needed. This is why, in this paper, we study positive definite kernels and, as a special example of them, CP sesquilinear-form-valued maps on $C^*$-algebras operating in modules over $C^*$-algebras. The study of kernels provides a platform for a unified treatment for generalized quantum measurement devices and positive kernels studied earlier, e.g.,\ in \cite{partha}. This approach is based on our earlier work on positive kernels \cite{KSGNS,pe13} largely inspired by Murphy's paper \cite{Murphy} and also by \cite{partha} which contains many applications (e.g.\ a construction of Bose-Einstein fields by using positive kernels).

The structures of measurement devices such as instruments and observables are convex in any probabilistic physical theory, particularly in quantum theory. Extremality in a convex set can often be seen as an optimality property; extreme observables and instruments contain no noise and ambiguity from combining different measurement schemes. Another important theme in physics is symmetry \cite{davies, varadarajan}: quantum measurement devices usually exhibit natural symmetries of the measurement outcome space such as phase space translation symmetries in position and momentum measurements. In this paper, we concentrate on convex sets of positive definite kernels and CP maps which are covariant under actions of a symmetry group on the inputs and outputs of the maps. Particularly, we study the dilations (Kolmogorov decompositions) and structure of these covariant maps and use these results to characterize the extreme points of convex sets of covariant kernels and CP maps. Results dealing with covariant CP maps have been published earlier (see, e.g.,\ \cite{BaBhLiSk2004, costache, costache2013, davies, heo, joita2011, varadarajan}) but our research on extremality, generalizing the characterization of extreme CP maps (between $C^*$-algebras) presented in \cite{arveson}, seems to be novel.

The study of covariant quantum instruments has arisen from Holevo's article \cite{holevo98} where some questions were left open. As we will see, the characterization of covariant instruments reduces to (generalized) imprimitivity systems which have been studied extensively and used in quantum physics as a fundamental tool to describe symmetric quantum observables \cite{mackey, varadarajan}. In physics, most symmetry groups are unimodular and of type I. Such groups have enough structure (e.g.\ unitary Fourier-Plancherel operators) to completely characterize covariant observables \cite{holevotypeI}. We will generalize results of papers \cite{carmeli, holevotypeI, holevopell} for covariant observables and also extend this study for quantum instruments. Particularly, the results obtained in \cite{cattaneo} linking covariant positive operator measures to canonical systems of imprimitivity are special cases of  theorems in Sections \ref{CPmaps} and \ref{quantumCP}.

The paper is arranged as follows: In Section \ref{def}, we give general definitions and basic results concerning modules over $C^*$-algebras and Hilbert modules. In Section \ref{kernels}, the covariant positive kernels are defined and their minimal covariant Kolmogorov decompositions are studied (Theorem \ref{Kdilat} and Corollary \ref{WKdilat}). The extremality conditions for normalized covariant positive kernels are given in Theorem \ref{Kcovarext}. We introduce covariant CP maps as an example of covariant positive kernels in Section \ref{CPmaps} and give results on their dilation theory and extremality properties in Theorems \ref{CPdilat} and \ref{Scovarext}. In all these findings, we show that, assuming the multiplying algebra of the module to be a $W^*$-algebra instead of being merely a $C^*$-algebra, especially the extremality conditions become simpler. Moreover, in the $W^*$-case, the dilation modules of covariant kernels and CP maps can be assumed to be self-dual. In Section \ref{marginal}, we define the concept of a subminimal dilation which is useful in the study of covariant quantum instruments. The relevant generalized covariant quantum CP maps are presented in Section \ref{quantumCP}, and their structure is characterized in Propositions \ref{lemmmm}, \ref{propo2}, \ref{paapropo}, \ref{prop:Ckiet}, and \ref{prop:lcsc}. In the case of a locally compact second countable symmetry group and a transitive value space, this result can be expressed by means of canonical systems of imprimitivity, giving a version of the imprimitivity theorem for generalized quantum instruments. A similar result for ordinary quantum instruments has been proven in \cite{carmeli2}. After this, we go on to examine quantum observables and instruments that are covariant with respect to a unimodular type-I group (Section \ref{covariantobs}) and apply our findings to the case of a square-integrable representation. Especially, a representation of Kraus type for a covariant instrument, conjectured in \cite{holevo98}, is shown to exist in Theorem \ref{covinstr}.

\section{Notations and definitions}\label{def}

Throughout this paper, we follow the convention $\mb N=\{1,\,2,\,3,\ldots\}$ and assume that the scalar field of any vector space is $\mb C$, the complex numbers. Let $\mb T=\{c\in\mb C\,|\,|c|=1\}$ be the circle group. Moreover, we let $\|\,\cdot\,\|_{N}$, or briefly $\|\,\cdot\,\|$, denote the norm of any normed (vector) space $N$ and let $I_{V}$ be the identity operator of a vector space $V$ 

Let $\mc A$ be a $C^*$-algebra. We say that an element $a\in\mc A$ is {positive}, $a\geq0$, if $a=b^*b$ for some $b\in\mc A$. If $\mc A$ is not unital we let $\tilde{\mc A}=\mc A\times\mb C\cong\mc A\oplus\mb C$ denote the unitalization of $\mc A$ so that $1_{\tilde{\mc A}}=(0,1)$ is the unit element of $\tilde{\mc A}$. If $\mc A$ is unital we denote its unit element by $1_{\mc A}$ and set $\tilde{\mc A}=\mc A$ in all equations where the symbol $\tilde{\mc A}$ occurs. Let $\mc U_{\mc A}=\big\{a\in\tilde{\mc A}\,\big|\,a^*a=aa^*=1_{\tilde{\mc A}}\big\}$ be the group of unitary elements, $\mr{Inv}_{\mc A}$ the group of invertible elements, and $\mc Z_{\mc A}$ the center of $\tilde{\mc A}$.
Hence, $\mc U_{\mc A}=\mc U_{\tilde{\mc A}}\subset\mr{Inv}_{\mc A}$ and $\mc Z_{\mc A}=\mc Z_{\tilde{\mc A}}$.

Let $\mb V$ be a (right) module\footnote{We do not assume that $\mb V$ is an inner product module or a Hilbert $C^*$-module over $\mc A$.} over $\mc A$, or briefly an {\it $\mc A$-module}. We say that a map $s:\mb V\times\mb V\to\mc A$ is an {\it $\mc A$-sesquilinear form} if
\begin{itemize}
\item[(S1)] $s(u,v+ c  w)=s(u,v)+ c  s(u,w)$,
\item[(S2)] $s(v+ c  w,u)=s(v,u)+\ovl c  s(w,u)$,
\item[(S3)] $s(u,va)=s(u,v)a$ and
\item[(S4)] $s(ua,v)=a^*s(u,v)$
\end{itemize}
for all $u,\,v,\,w\in\mb V$, $ c \in\mb C$ and $a\in\mc A$. We denote the set of $\mc A$-sesquilinear forms $s:\mb V\times\mb V\to\mc A$ by $S_{\mc A}(\mb V)$. If, additionally,
\begin{itemize}
\item[(S5)] $s(v,v)\geq0$
\end{itemize}
for all $v\in\mb V$, we say that the $\mc A$-sesquilinear form $s$ is {positive}.
Note that (S1), (S2), and (S5) imply $s(w,v)=s(v,w)^*$ for all $v,\,w\in\mb V$.
If $s:\mb V\times\mb V\to\mc A$ satisfies conditions (S1)--(S6), where
\begin{itemize}
\item[(S6)] $s(v,v)=0$ (if and) only if $v=0$,
\end{itemize}
we say that $s$ is an {inner product}. For an inner product $s:\mb V\times\mb V\to\mc A$ we often denote $s(v,w)=\sis{v}{w}_{\mb V}=\sis{v}{w}$. When an $\mc A$-module $\mb V$ is equipped with an inner product $\sis{\,\cdot\,}{\,\cdot\,}:\mb V\times\mb V\to\mc A$, we may define a norm $\|\cdot\|:\mb V\to[0,\infty)$ through
$$
\|v\|=\sqrt{\|\sis{v}{v}\|_{\mc A}},\qquad v\in\mb V,
$$
and, if the normed space $(\mb V,\|\cdot\|)$ is complete, $\mb V$ is called a {\it Hilbert $\mc A$-module} (or a Hilbert $C^*$-module over $\mc A$).

When $\mb V$ and $\mb W$ are $\mc A$-modules we say that a map $L:\mb V\to\mb W$ is {\it $\mc A$-linear} if $L(v+ c  w)=L(v)+ c  L(w)$ and $L(va)=L(v)a$ for all $v,\,w\in\mb V$, $ c \in\mb C$ and $a\in\mc A$. We denote the set of $\mc A$-linear maps $L:\mb V\to\mb W$ by $\mr{Lin}_{\mc A}(\mb V;\mb W)$. We also denote the set of $\mc A$-linear maps $L:\mb V\to\mb V$ by $\mr{Lin}_{\mc A}(\mb V)$, and define the {commutator} $[L_1,L_2]=L_1L_2-L_2L_1\in\mr{Lin}_{\mc A}(\mb V)$ for all $L_1,\,L_2\in\mr{Lin}_{\mc A}(\mb V)$. Moreover, let $\mr{GL}_{\mc A}(\mb V)\subset \mr{Lin}_{\mc A}(\mb V)$ be the group of $\mc A$-linear bijections on $\mb V$. If $a$ belongs to the center of $\mc A$ we define the right multiplication map $a.\in\mr{Lin}_{\mc A}(\mb V)$ by $a.v:=va$ for all $v\in\mb V$.

Suppose next that $(\mb M,\sis{\,\cdot\,}{\,\cdot\,})$ is a Hilbert $\mc A$-module and denote the set of bounded $\mc A$-linear maps $B:\mb M\to\mb M$ by $\mc B_{\mc A}(\mb M)$. Let $\mb M^*$ denote the set of continuous $\mc A$-linear maps $f:\mb M\to\mc A$. We say that $\mb M$ is {\it self-dual} if for any $f\in\mb M^*$ there exists an $m\in\mb M$ such that $f(m')=\sis{m}{m'}$ for all $m'\in\mb M$. For example, any Hilbert $\mc A$-module $\mb M$ is self-dual if $\mc A$ is finite-dimensional \cite{manuilov}. Let $B\in\mc B_{\mc A}(\mb M)$. If there is $B^*\in\mc B_{\mc A}(\mb M)$ such that $\sis{m}{Bm'}=\sis{B^* m}{m'}$ for all $m,\,m'\in\mb M$ we say that $B$ is {adjointable}. Denote the $C^*$-algebra of adjointable maps $B\in\mc B_{\mc A}(\mb M)$ by $\mc L_{\mc A} (\mb M)$. It is easy to see that if $\mb M$ is self-dual then $\mc L_{\mc A}(\mb M)=\mc B_{\mc A}(\mb M)$. If $B\in\mc L_{\mc A}(\mb M)$ and $B^*=B$ we say that $B$ is {self-adjoint} and if $B$ is invertible and $B^*=B^{-1}$ then $B$ is {unitary.} A map $B\in\mc L_{\mc A}(\mb M)$ is unitary if and only if it preserves the inner product and is surjective \cite{lance}. We denote the group of unitary $\mc A$-linear maps on $\mb M$ by $\mc U_{\mc A}(\mb M)$. Especially, $\mc U_{\tilde{\mc A}}(\tilde{\mc A})\cong\mc U_{\mc A}$. In the case $\mc A=\mb C$ (so that $\mb M$ is a Hilbert space), we denote simply $\mc L_{\mb C}(\mb M)=\mc L(\mb M)$ and $\mc U_{\mb C}(\mb M)=\mc U(\mb M)$. 

\begin{rem}\rm\label{rem1}
Suppose that $\mc A$ is a $C^*$-algebra which is not unital, and $\mb V$ is an $\mc A$-module. Let $\tilde{\mc A}=\mc A\oplus\mb C$ be the unitalization of $\mc A$. Immediately one sees that $\mb V$ can be viewed as an $\tilde{\mc A}$-module and, if $\mb M$ is a Hilbert $\mc A$-module, then $\mb M$ is also a Hilbert $\tilde{\mc A}$-module. Any $\mc A$-sesquilinear form $s:\,\mb V\times\mb V\to\mc A$ can be seen as an $\tilde{\mc A}$-sesquilinear form $\mb V\times\mb V\to\tilde{\mc A}$ which we denote by the same symbol $s$. A similar result holds for $\mc A$-linear maps. For example, if $B\in\mc L_{\mc A}(\mb M)$ then $B\in\mc L_{\tilde{\mc A}}(\mb M)$ and $a.\in\mc L_{\tilde{\mc A}}(\mb M)$ for all $a\in\mc Z_A$.
\end{rem}

\begin{ma}\rm
Let $G$ be a group, $e$ the neutral element of $G$, and $\mb V$ a module over a $C^*$-algebra $\mc A$. Let $\m:\,G\times G\to\mc U_{\mc A}\cap\mc Z_{\mc A}$ be a 2-cocycle, that is, $\m(e,g)=1_{\tilde{\mc A}}=\m(g,e)$ and $\m(g,hk)\m(h,k)=\m(gh,k)\m(g,h)$ for all $g,\,h,\,k\in G$. We say that a map $U:\,G\to\mr{GL}_{\mc A}(\mb V)$ is a {\it ($\sigma$-)multiplier representation} of $G$ on $\mb V$ if $U(e)=I_{\mb V}$ and $U(gh)=\m(g,h).U(g)U(h)$ for all $g,\,h\in G$. If the 2-cocycle is trivial, $\m(g,h)\equiv 1_{\tilde{\mc A}}$, then the group homomorphism $U:\,G\to\mr{GL}_{\mc A}(\mb V)$ is called 
a {\it representation} of $G$ on $\mb V$. Finally, we say that a (multiplier) representation $\tilde U:\,G\to\mc U_{\mc A}(\mb M)$ is a {\it unitary (multiplier) representation} of $G$ on a Hilbert $\mc A$-module $\mb M$.
\end{ma}

\begin{rem}\label{vonNeumann}\rm 
Recall that any $C^*$-algebra $\mc A$ can be viewed as a norm-closed ${}^*$-subalgebra of $\lh$, where $\hil$ is a Hilbert space, and then each Hilbert $\mc A$-module $\mb M$ can be identified with a norm-closed $\mc A$-submodule of $\mc L(\hil;\hil')$, the set of bounded linear operators from $\hil$ to a Hilbert space $\hil'$ \cite[Corollary 3.14]{skeide2000}. Now $\mb M\subseteq\mc L(\hil;\hil')$ is equipped with the inner product $\langle m|m'\rangle:=m^*m'$, $m,\,m'\in\mb M$, and the norm $m\mapsto \|m\|=\sqrt{\|m^*m\|_{\lh}}$ which is the usual operator norm of $\mc L(\hil;\hil')$. Moreover, the $C^*$-algebra $\mc L_{\mc A}(\mb M)$ of $\mc A$-linear bounded adjointable operators on $\mb M$ can be identified with a norm closed ${}^*$-subalgebra $\mc B\subseteq\mc L(\hil')$ through
$$
\mc B\ni b\mapsto B\in\mc L_{\mc A}(\mb M),\qquad B(m)=bm,\qquad m\in\mb M\subseteq\mc L(\hil;\hil').
$$

If $\mc A$ is a $W^*$-algebra then it can be seen as a von Neumann algebra, that is, a ${}^*$-subalgebra of $\lh$ which is strongly closed and contains the identity $I_\hil$. Moreover, any self-dual Hilbert $\mc A$-module $\mb M$ becomes then a strongly closed $\mc A$-submodule of $\mc L(\hil;\hil')$ and $\mc L_{\mc A}(\mb M)$ is a $W^*$-algebra \cite[Proposition 3.10]{paschke73}. For example, if $\mc A=\lh$ then all self-dual Hilbert $\lh$-modules are of the form $\mc L(\hil;\hil')$ \cite[Proposition 13.9]{bhatskeide2000}. One has $\mc L_{\mc L(\hil)}\big(\mc L(\hil;\hil')\big)\cong\mc L(\hil')$, so that, especially, the unitary group $\mc U_{\lh}\big(\mc L(\hil;\hil')\big)$ is just the group $\mc U(\hil')$ of unitary operators on $\hil'$.
\end{rem}

\begin{rem}\label{vonNeumann1}\rm
Whenever $\mb L$ is a self-dual Hilbert $\mc A$-module over a $W^*$-algebra $\mc A$ we may define for all $\f\in\mc A_*^+$ ($\mc A_*^+$ being the positive cone within the predual $\mc A_*$ of $\mc A$) and $\ell'\in\mb L$ the seminorms $p_\f$ and $p_{\f,\ell'}$ on $\mb L$ through
$$
p_\f(\ell)=\sqrt{\f(\sis{\ell}{\ell})},\quad p_{\f,\ell'}(\ell)=|\f(\sis{\ell'}{\ell})|,\qquad\ell\in\mb L.
$$
We denote by $\tau_1$ the topology of $\mb L$ induced by all the seminorms $p_\f$ and by $\tau_2$ the topology induced by the seminorms $p_{\f,\ell'}$. It follows that $\mb L$ is complete with respect to these topologies \cite[Theorem 3.5.1]{manuilov}.

Suppose that $\mc A$ is a $W^*$-algebra and $\mb M$ is a Hilbert $\mc A$-module. The module $\mb M$ can be seen as a subset of its dual $\mb M^*$ through the map $\mb M\ni m\mapsto\hat m\in\mb M^*$, $\hat m(m')=\sis{m}{m'}$. There is an inner product $\sis{\,\cdot\,}{\,\cdot\,}:\,\mb M^*\times\mb M^*\to\mc A$ that makes $\mb M^*$ into a self-dual Hilbert $\mc A$-module and $\sis{\hat m}{\hat m'}=\sis{m}{m'}$, $\sis{\hat m}{f}=f(m)$ for all $m,\,m'\in\mb M$ and $f\in\mb M^*$ \cite[Theorem 3.2]{paschke73}. Moreover, each $\mc A$-linear adjointable operator $B\in\mc L_{\mc A}(\mb M)$ extends to a unique $\overline B\in\mc L_{\mc A}(\mb M^*)=\mc B_{\mc A}(\mb M^*)$ \cite[Corollary 3.7]{paschke73}. We may endow $\mb M^*$ with the topologies $\tau_1$ and $\tau_2$ introduced above. As a subset of $\mb M^*$ the original module $\mb M$ is $\tau_1$-dense; this follows from an analysis parallelling that of the proof of the indication (ii)$\Rightarrow$(i) in the proof of \cite[Theorem 3.5.1]{manuilov}. Similarly, applying the same theorem, $\mb M$ is $\tau_2$-dense in $\mb M^*$.

For example, assume that $\mc A=\mc L(\hil)$ and consider a Hilbert $\mc L(\hil)$-module $\mb M$ so that (as a self-dual module) $\mb M^*\cong\mc L(\hil;\hil')$ (here $\hil$ and $\hil'$ are Hilbert spaces). We may identify $\mc L(\hil)_*^+$ with the set $\mc T(\hil)^+$ of positive trace-class operators on $\hil$, and for the seminorms $p_t$, $t\in\mc T(\hil)^+$, generating the topology $\tau_1$ this implies
$$
p_t(m)=\sqrt{\tr{tm^*m}}=\sqrt{\sum_{j=1}^\infty\|m\xi_j\|^2}=:p_{\{\xi_j\}_j}(m)
$$
for all $m\in\mc L(\hil;\hil')$, where the vectors $\xi_j\in\hil$ constitute a decomposition $t=\sum_j|\xi_j\ra\la\xi_j|$. It follows that $\tau_1$ is generated by the seminorms $p_{\{\xi_j\}_j}$ for sequences of vectors $\xi_j$ such that $\sum_j\|\xi_j\|^2<\infty$. We call this topology as the $\sigma$-strong topology of $\mc L(\hil;\hil')$. Similarly, the topology $\tau_2$ is induced by the seminorms $m'\mapsto|\tr{tm^*m'}|$ for $t\in\mc T(\hil)^+$ and $m\in\mc L(\hil,\hil')$ so that we call it as the $\sigma$-weak topology of $\mc L(\hil,\hil')$. Hence, any Hilbert $\mc L(\hil)$-module can be viewed as a $\sigma$-strongly/weakly dense $\mc L(\hil)$-submodule of $\mc L(\hil;\hil')$ with some Hilbert space $\hil'$.
\end{rem}

\begin{rem}\label{vonNeumann2}\rm
Let $\Omega$ be a set equipped with a $\sigma$-algebra $\Sigma$, and let $\nu:\,\Sigma\to[0,\infty]$ be a $\sigma$-finite measure. We let $L^\infty(\nu)$ denote the von Neumann algebra of (equivalence classes of) essentially bounded $\nu$-measurable functions $f:\,\Omega\to\mb C$. Moreover, we let $L^2(\nu)$ be the Hilbert space of (equivalence classes of) $\nu$-square integrable (measurable) functions $\psi:\,\Omega\to\mb C$ and let
\begin{equation}\label{valuedirect}
\hil_\oplus=\int_\Omega^\oplus\hil_{n(\omega)}\,d\nu(\omega)\cong \left[\bigoplus_{n\in\mb N}L^2(\nu_n)\otimes\hil_n\right]\oplus\left[L^2(\nu_\infty)\otimes\hil_\infty\right]
\end{equation}
be the direct integral Hilbert space \cite{Dix2} where $\Om\ni\omega\mapsto n(\omega)\in\{0\}\cup\mb N\cup\{\infty\}$ is $\nu$-measurable and, for each $n\in\{0\}\cup\mb N\cup\{\infty\}$, $\hil_n$ is an $n$-dimensional separable Hilbert space and $\nu_n$ is the restriction of $\nu$ to the $\sigma$-algebra $\Sigma_n:=\Sigma\cap\Omega_n$ where $\Omega_n:=\{\omega\in\Omega\,|\,n(\omega)=n\}$. Note that, in our definition, the direct integral $\hil_\oplus$ depends essentially only on the multiplicity function $\om\mapsto n(\om)$ (and the measure $\nu$). As above, without any further mention, we identify the Hilbert space $L^2(\nu;\hil)$ of (equivalence classes of) $\nu$-square integrable functions $\psi:\,\Omega\to\hil$ with $L^2(\nu)\otimes\hil$ (where $\hil$ is a Hilbert space).

Let $\nu'$ be the restriction of $\nu$ to $\Sigma':=\Sigma\cap\Omega'$, $\Omega':=\Omega\setminus\Omega_0$. Note that $\hil_\oplus$ is separable if and only if $L^2(\nu')$ is separable which is the case exactly when the measure space $(\Omega',\Sigma',\nu')$ has a countable basis \cite[Theorem 8.1]{Taylor}.

Since any $f\in L^\infty(\nu)$ defines a multiplicative (or diagonalizable) operator $\hat f\in\mc L(\hil_\oplus)$ through $(\hat f\psi)(\om)=f(\om)\psi(\om)$ for all $\psi\in\hil_\oplus$ and $\nu$-a.a.\ $\om\in\Om$, one can view $L^\infty(\nu)$ as a von Neumann subalgebra of $\mc L(\hil_\oplus)$ and the commutant of $L^\infty(\nu)$ is the von Neumann subalgebra consisting of the decomposable operators of  $\mc L(\hil_\oplus)$.\footnote{Recall that a decomposable bounded operator $D=\int_\Om^\oplus D(\om)\,d\nu(\om)$ is defined by a weakly measurable field $\Om\ni\om\mapsto D(\om)\in\mc L(\hil_{n(\om)})$ such that $\|D\|=\nu\mr{-ess}\,\sup_{\om\in\Om}\|D(\om)\|<\infty$ via $(D\psi)(\om)=D(\om)\psi(\om)$ for all $\psi\in\hil_\oplus$ and $\nu$-a.a.\ $\om\in\Om$.} Moreover, $\hil_\oplus$ is a self-dual Hilbert $L^\infty(\nu)$-module.

It is well known that any abelian von Neumann algebra of bounded operators on a separable Hilbert space $\hil'$ is isomorphic to some $L^\infty(\mu)$ where the measure $\mu:\,\Sigma\to[0,\infty)$ is $\sigma$-finite (with a countable basis). In addition, $\hil'$ is unitary equivalent with some direct integral Hilbert space where $L^\infty(\mu)$ acts multiplicatively as above \cite{Dix2}.
\end{rem}

\section{Positive covariant kernels}\label{kernels}

Suppose that $G$ is a group and $X$ is a $G$--space, i.e.\ the (nonempty) set $X$ is equipped with a $G$--action $G\times X\ni(g,x)\mapsto gx\in X$, that is, $(gh)x=g(hx)$ for all $g,\,h\in G$ and $x\in X$ and $ex=x$, $x\in X$, where $e$ is the unit of $G$. Assume that $\mc A$ is a (possibly nonunital) $C^*$-algebra, $\mb V$ an $\mc A$-module, and
$U:\,G\to\mr{GL}_{\mc A}(\mb V)$ a representation of $G$ on $\mb V$. Moreover, let us fix a map $\a:\,G\times X\to\mr{Inv}_{\mc A}\cap\mc Z_{\mc A}$ such that $\a(e,x)=1_{\tilde{\mc A}}$, $x\in X$, and  
\begin{equation}\label{multi}
\a(gh,x)=\m(g,h)\a(h,x)\a(g,hx),\qquad g,\,h\in G,\quad x\in X.
\end{equation}
where $\m:\,G\times G\to\mc U_{\mc A}\cap\mc Z_{\mc A}$ is a 2-cocycle called as {\it the 2-cocycle associated to $\alpha$.} 

\begin{ma}\label{kovarkernel}
{\rm
We say that a (non-zero) map $K=\big[X\times X\ni(x,y)\mapsto K_{x,y}\in S_{\mc A}(\mb V)\big]$ is a {\it positive $(X,\a,U)$-covariant kernel} or simply {\it positive covariant kernel} if it is positive, i.e.,\ for all $n\in\mb N$, $x_1,\ldots,\,x_n\in X$ and $v_1,\ldots,\,v_n\in\mb V$
$$
\sum_{j,k=1}^n K_{x_j,x_k}(v_j,v_k)\geq0
$$
and covariant, i.e.,\ for all $g\in G$, $x,\,y\in X$ and $v,\,w\in\mb V$
\begin{equation}\label{covacond}
K_{gx,gy}(v,w)=\a(g,x)^*K_{x,y}\big(U(g^{-1})v,U(g^{-1})w\big)\a(g,y).
\end{equation}
We denote the convex set of $(X,\a,U)$-covariant positive kernels by ${\bf Ker}_\a^U(X)$.
Note that, by Remark \ref{rem1}, $K_{x,y}\in S_{\tilde{\mc A}}(\mb V)$ for all $x,\,y\in X$ but still the right hand side of equation \eqref{covacond} belongs to $\mc A$ since $\mc A\cdot \tilde{\mc A}\subseteq\mc A$.
}
\end{ma}

\begin{theor}\label{Kdilat}
Suppose that $K\in{\bf Ker}_\a^U(X)$. There is a Hilbert $\mc A$-module $(\mb M,\sis{\,\cdot\,}{\,\cdot\,})$, a mapping $R:X\to\mr{Lin}_{\mc A}(\mb V;\mb M)$ and a $\sigma$-multiplier representation $\tilde U:G\to\mc U_{\mc A}(\mb M)$ with the 2-cocycle $\m$ associated to $\alpha$ (i.e.\ $\tilde U(e)=I_{\mb M}$ and $\tilde U(g)\tilde U(h)=\m(g,h).\tilde U(gh)$ for all $g,\,h\in G$) such that
\begin{itemize}
\item[{\rm (i)}] $K_{x,y}(v,w)=\sis{R(x)v}{R(y)w}$ for all $x,\,y\in X$ and $v,\,w\in\mb V$,
\item[{\rm (ii)}] $\mr{lin}_{\mb C}\{R(x)v\,|\,x\in X,\ v\in\mb V\}$ is dense in $\mb M$ and
\item[{\rm (iii)}] $R(x)U(g)=\a(g,g^{-1}x).\tilde U(g)R(g^{-1}x)$ for all $x\in X$ and $g\in G$.
\end{itemize}
If $(\mb M',R',\tilde U')$ is another triple satisfying the conditions (i)--(iii) above, then there is a unitary map $W\in\mr{Lin}_{\mc A}(\mb M,\mb M')$ such that $WR(x)=R'(x)$ for all $x\in X$ and $W\tilde U(g)=\tilde U'(g)W$ for all $g\in G$. Conversely, if $(\mb M,\sis{\,\cdot\,}{\,\cdot\,})$ is a Hilbert $\mc A$-module, $R:X\to\mr{Lin}_{\mc A}(\mb V;\mb M)$ a map, and  $\tilde U:G\to\mc U_{\mc A}(\mb M)$ a multiplier representation with the 2-cocycle $\m$ (associated to $\alpha$) which satisfy (iii) then (i) defines a positive $(X,\a,U)$-covariant kernel.
\end{theor}

\begin{proof}
Let $K\in{\bf Ker}_\a^U(X)$. The existence of the Hilbert $\mc A$-module $\mb M$ and the map $R$ with the properties (i) and (ii), and the existence of a unitary $W\in\mr{Lin}_{\mc A}(\mb M,\mb M')$ with $WR(x)=R'(x)$, $x\in X$, are proved in  \cite[Proposition 3.1, Theorem 3.3]{KSGNS}. Hence, it is left to construct the representation $\tilde U$ satisfying the conditions of the proposition.

Suppose that a pair $(\mb M,R)$, satisfying items (i) and (ii),  is chosen.  Assume that $\eta\in\mb M$ belongs to the dense subspace of item (ii), i.e.,\ there are $n\in\mb N$, $x_1,\ldots,\,x_n\in X$, $v_1,\ldots,\,v_n\in\mb V$ such that $\eta=\sum_{j=1}^nR(x_j)v_j$. Define
$$
\eta^g=\sum_{j=1}^nR(gx_j)U(g)v_j\a(g,x_j)^{-1},\qquad g\in G.
$$
Using the covariance condition \eqref{covacond}, we get
\begin{eqnarray*}
\sis{\eta^g}{\eta^g}&=&\sum_{j,k=1}^n\sis{R(gx_j)U(g)v_j\a(g,x_j)^{-1}}{R(gx_k)U(g)v_k\a(g,x_k)^{-1}}\\
&=&\sum_{j,k=1}^n[\a(g,x_j)^{-1}]^*K_{gx_j,gx_k}\big(U(g)v_j,U(g)v_k\big)\a(g,x_k)^{-1}\\
&=&\sum_{j,k=1}^n[\a(g,x_j)^{-1}]^*\a(g,x_j)^*K_{x_j,x_k}(v_j,v_k)\a(g,x_k)\a(g,x_k)^{-1}\\
&=&\sum_{j,k=1}^n\sis{R(x_j)v_j}{R(x_k)v_k}=\sis{\eta}{\eta}
\end{eqnarray*}
showing that, for any $g\in G$, one may define a linear map $\tilde U(g):\mb M\to\mb M$ via
$$
\tilde U(g)R(x)v:=R(gx)U(g)v\a(g,x)^{-1},\qquad x\in X,\;v\in\mb V,
$$
which is clearly well-defined, $\mc A$-linear and bounded, and which preserves the inner products by polarization. Especially, $\tilde U(e)=I_{\mb M}$ and one gets
\begin{eqnarray*}
\tilde U(gh)R(x)v&=&R(ghx)U(gh)v\a(gh,x)^{-1}\\
&=&R\big(g(hx)\big)U(g)U(h)v\a(g,hx)^{-1}\a(h,x)^{-1}\m(g,h)^{-1}\\
&=&\tilde U(g)R(hx)U(h)v\a(h,x)^{-1}\m(g,h)^{-1}\\
&=&\tilde U(g)\tilde U(h)R(x)v\m(g,h)^*
\end{eqnarray*}
for all $x\in X$, $v\in\mb V$, and $g,\,h\in G$, so that $\tilde U(gh)=\m(g,h)^*.\tilde U(g)\tilde U(h)$ or $\tilde U(g)\tilde U(h)=\m(g,h).\tilde U(gh)$ by property (ii). For all $m,\, m'\in\mb M$ one gets
\begin{eqnarray*}
\sis{\tilde U(g)m}{ m'}&=&\sis{\tilde U(g^{-1})\tilde U(g)m}{\tilde U(g^{-1}) m'}\\
&=&\sis{m\m(g^{-1},g)}{\tilde U(g^{-1}) m'}=\sis{m}{\tilde U(g^{-1}) m' \m(g^{-1},g)^*}
\end{eqnarray*}
so that $\tilde U(g)^*=\m(g^{-1},g)^*.\tilde U(g^{-1})=\tilde U(g)^{-1}$ and $\tilde U:G\to\mc U_{\mc A}(\mb M)$ is a multiplier representation. Obviously, (iii) holds.

Assume that $(\mb M',R',\tilde U')$ satisfy conditions (i)--(iii). If $W:\mb M\to\mb M'$ is such that $WR(x)=R'(x)$ for all $x\in X$ then 
\begin{eqnarray*}
W\tilde U(g)R(x)&=&\a(g,x)^{-1}.WR(gx)U(g)=\a(g,x)^{-1}.R'(gx)U(g)\\
&=&\tilde U'(g)R'(x)=\tilde U'(g)WR(x)
\end{eqnarray*}
for all $x\in X$ and $g\in G$, and hence (ii) implies $W\tilde U(g)=\tilde U'(g)W$. The rest of the proof is straightforward.
\end{proof}

\begin{ma}
{\rm
For any $K\in{\bf Ker}_\a^U(X)$, we say that a triple $(\mb M,R,\tilde U)$ satisfying conditions (i)--(iii) of the preceding theorem is called a {\it minimal $(X,\a,U)$-covariant (Kolmogorov) decomposition} for $K$. A minimal $(X,\a,U)$-covariant decomposition for $K$ is unique up to a unitary transformation intertwining the unitary representations. If $\mb M$ is self-dual we say that $K$ is {\it regular}. Although we concentrate here on covariant kernels, the results presented in the sequel also apply to the case of positive kernels with no particular covariance properties. Namely, if one chooses the group $G$ to be  trivial, $G=\{e\}$, then the kernels $K\in{\bf Ker}_\a^U$ are `non-covariant' and, in this case, we say that a pair $(\mb M,R)$ satisfying the conditions (i) and (ii) is a {\it minimal (Kolmogorov) decomposition} for $K$.
}
\end{ma}

\begin{rem}\rm
We could extend our definition of $\alpha$ and assume that its range could also contain $0$ so that, for any $g\in G$ and $x\in X$, $\alpha(g,x)$ is either invertible element of the center or $0$. Similarly, we may drop out the condition $\a(e,x)=1_{\tilde{\mc A}}$. Now if $\alpha(h,x)=0$ for some $h\in G$ and $x\in X$ then \eqref{multi} implies that $\alpha(g,x)=0$ for all $g\in G$ and $K_{gx,gy}=K_{gy,gx}=0$ for all $g\in G$ and $y\in X$ by \eqref{covacond}; let us denote the set of these $x\in X$ as $\Theta_\alpha$. Then, by redefining $\alpha'(g,hx):=\sigma(g,h)^*$ for all $g,\,h\in G$ and $x\in\Theta_\alpha$ and $\alpha'(g,x):=\alpha(g,x)$ otherwise, one sees that \eqref{multi} and \eqref{covacond} hold when $\alpha$ is replaced by $\alpha'$. Moreover, recalling that $\sigma(g,e)=1_{\tilde{\mc A}}$, one obtains using \eqref{multi} that $\alpha(g,x)=\alpha(ge,x)=\alpha(e,x)\alpha(g,x)$ for all $g\in G$ and $x\in X$. From this one sees that (by the definition of $\Theta_\alpha$) $\alpha(e,x)=0$ when $x\in\Theta_\alpha$ and otherwise (since now $\alpha(g,x)$ is invertible) one finds that $\alpha(e,x)=1_{\tilde{\mc A}}$. Thus, for the redefined map $\alpha'$, $\alpha'(e,x)=1_{\tilde{\mc A}}$ for all $x\in X$. Finally, we note that if $K_{x,x}=0$ (i.e.\ $R(x)=0$) for some $x\in G$ then $K_{gx,y}=K_{y,gx}=0$ for all $g\in G$ and $y\in X$ so that one can restrict $K$ to the set $X_0\times X_0$ where $X_0=\{x\in X\,|\,K(x,x)\ne 0\}$.  Hence, in many applications one can assume that $K_{x,x}\ne 0$ (i.e.\ $R(x)\ne 0$) for all $x\in X$.
\end{rem}

\begin{rem}\rm\label{invariant}
In the case when the representation $U$ is trivial, $U=I$, i.e.\ $U(g)=I_{\mb V}$ for all $g\in G$, we say that $K\in{\bf Ker}_\a^I(X)$ is (projectively) invariant. Hence, Theorem \ref{Kdilat} is a generalization of Theorem 2.7 of \cite{partha} (where $\mc A=\mb C=\mb V$).
\end{rem}

\begin{rem}\rm\label{central}
In many physically relevant situations, one considers a multiplier representation $U:\,G\to\mr{GL}_{\mc A}(\mb V)$, i.e.\ $U(e)=I_{\mb V}$ and $U(g)U(h)v=U(gh)vz(g,h)$ for all $g,\,h\in G$ and $v\in\mb V$ where $z:\,G\times G\to\mc U_{\mc A}\cap\mc Z_{\mc A}$ is a 2-cocycle on $G$. Now a positive $(X,\a,U)$-covariant kernel $K$ can be defined in exactly the same way as in Definition \ref{kovarkernel} and one obtains Proposition \ref{Kdilat} (i.e.\ the triple $(\mb M,R,\tilde U)$) with the exception that $\tilde U:G\to\mc U_{\mc A}(\mb M)$ is a $\m'$-multiplier representation with the 2-cocycle $\m'(g,h)=\m(g,h)z(g,h)$ (where $\m$ is associated to $\a$). In this case, one can consider a central extension group $G^z:=G\times(\mc U_{\mc A}\cap\mc Z_{\mc A})$ where the group product is $(g,t)(h,u):=\big(gh,z(g,h)tu\big)$, $g,\,h\in G$, $t,\,u\in\mc U_{\mc A}\cap\mc Z_{\mc A}$, and equip $X$ with a $G^z$--action $(g,t)\cdot x:=gx$ for all $(g,t)\in G^z$ and $x\in X$. Moreover, we may set up $\a^z:\,G^z\times X\to\mr{Inv}_{\mc A}\cap\mc Z_{\mc A}$ through $\a^z\big((g,t),x\big):=\a(g,x)$ and define $\sigma^z:\,G^z\times G^z\to\mc U_{\mc A}\cap\mc Z_{\mc A}$ via $\sigma^z\big((g,t),\,(h,u)\big):=\sigma(g,h)$. Now $\sigma^z$ is the 2-cocycle associated to $\a^z$, $U^z:\,G^z\to\mr{GL}_{\mc A}(\mb V)$, $U^z(g,t)v:=U(g)vt$ is a representation of $G^z$ on $\mb V$ and $K$ can be viewed as a positive $(X,\a^z,U^z)$-covariant kernel. Let $(\mb M,R,\tilde U^z)$ be a minimal $(X,\a^z,U^z)$-covariant Kolmogorov decomposition for $K$ (recall that the minimal $\mb M$ and $R$ for $K$ are unique up to a unitary equivalence as stated in Theorem \ref{Kdilat}). Especially, $\tilde U^z:G^z\to\mc U_{\mc A}(\mb M)$ is a $\sigma^z$-multiplier representation, $K_{x,y}(v,w)=\sis{R(x)v}{R(y)w}$ for all $x,\,y\in X$ and $v,\,w\in\mb V$, and $R(x)U(g)v=\tilde U^z(g,1_{\tilde{\mc A}})R(g^{-1}x)v\a(g,g^{-1}x)$ for all $x\in X$, $g\in G$, and $v\in\mb V$. Finally, we note that $\tilde U^z(g,t)=t.\tilde U(g)$ for all $(g,t)\in G^z$.
\end{rem}

Suppose now that $\mc A$ is a $W^*$-algebra and $K\in{\bf Ker}_\a^U(X)$ with a minimal $(\alpha,U)$-covariant decomposition $(\mb M,R,\tilde U)$. According to Remark \ref{vonNeumann1}, we may view $\mb M^*$ as a self-dual Hilbert $\mc A$-module where $\mb M$ is a $\tau_1$- (or $\tau_2$-) dense subset. Thus, the $\mb C$-linear combinations of vectors of the form $R(x)v$, $x\in X$, $v\in\mb V$, form a $\tau_1$-dense subspace of $\mb M^*$. Moreover, $\tilde U$ extends to a unique unitary multiplier representation of $G$ on $\mb M^*$ which we also denote by $\tilde U$. We now have the following corollary of Theorem \ref{Kdilat}:

\begin{cor}\label{WKdilat}
Suppose that $\mc A$ is a $W^*$-algebra and $K\in{\bf Ker}_\a^U(X)$. There is a self-dual Hilbert $\mc A$-module $(\mb L,\sis{\,\cdot\,}{\,\cdot\,})$ equipped with the topology $\tau_1$ (resp.\  $\tau_2$) introduced in Remark \ref{vonNeumann1}, a mapping $R:X\to\mr{Lin}_{\mc A}(\mb V;\mb L)$ and a $\m$-multiplier representation $\tilde U:G\to\mc U_{\mc A}(\mb L)$ such that
\begin{itemize}
\item[{\rm (i)}] $K_{x,y}(v,w)=\sis{R(x)v}{R(y)w}$ for all $x,\,y\in X$ and $v,\,w\in\mb V$,
\item[{\rm (ii)}] $\mr{lin}_{\mb C}\{R(x)v\,|\,x\in X,\ v\in\mb V\}$ is $\tau_1$- (resp.\ $\tau_2$-) dense in $\mb L$ and
\item[{\rm (iii)}] $R(x)U(g)=\a(g,g^{-1}x).\tilde U(g)R(g^{-1}x)$ for all $x\in X$ and $g\in G$.
\end{itemize}
If $(\mb L',R',\tilde U')$ is another triple satisfying the conditions (i)--(iii) above, then there is a unitary map $W\in\mr{Lin}_{\mc A}(\mb L,\mb L')$ such that $WR(x)=R'(x)$ for all $x\in X$ and $W\tilde U(g)=\tilde U'(g)W$ for all $g\in G$.
\end{cor}

\begin{ma}\label{W*dilat}\rm
When $\mc A$ is a $W^*$-algebra and $K\in{\bf Ker}_\a^U(X)$ we call a triple $(\mb L,R,\tilde U)$ satisfying the conditions (i)--(iii) of Corollary \ref{WKdilat} as a {\it $W^*$-minimal $(X,\a,U)$-covariant (Kolmogorov) decomposition} for $K$. The pair $(\mb L,R)$ is  simply called a {\it $W^*$-minimal (Kolmogorov) decomposition} for $K$.
\end{ma}

\subsection{Extreme points}

Let $G$ be a group, $X\neq\emptyset$ a $G$-space with the $G$-action $(g,x)\mapsto gx$, $\mb V$ a module over a $C^*$-algebra $\mc A$, and $U:G\to\mr{GL}_{\mc A}(\mb V)$ a representation of $G$ on $\mb V$. We also fix the maps $\a$ and $\m$ as in the preceding section. Also assume that $Z\subset X\times X$ is nonempty and let $L\in{\bf Ker}_\a^U(X)$. Denote the set of those $K\in{\bf Ker}^U_\a(X)$ such that $K_{x,y}=L_{x,y}$ for all $(x,y)\in Z$ by ${\bf Ker}_\a^U(Z,L)$. Clearly the set ${\bf Ker}_\a^U(Z,L)$ is convex.

\begin{theor}\label{Kcovarext}
Suppose that $K\in{\bf Ker}_\a^U(Z,L)$ and $(\mb M,R,\tilde U)$ is a minimal $(X,\a,U)$-covariant decomposition for $K$. If $K$ is an extreme point of the set ${\bf Ker}_\a^U(Z,L)$
then for all self-adjoint $D\in\mc L_{\mc A}(\mb M)$, the conditions $[D,\tilde U(g)]=0$ for all $g\in G$ and
\begin{equation}\label{Kextehto}
\sis{R(x)v}{DR(y)v}=0,\qquad(x,y)\in Z,\quad v\in\mb V
\end{equation}
imply $D=0$. If $K$ is regular then also the converse holds.

Suppose that $\mc A$ is a $W^*$-algebra, and let $(\mb L,R,\tilde U)$ be a $W^*$-minimal $(X,\a,U)$-covariant decomposition for $K$. The kernel $K$ is extreme in the set ${\bf Ker}_\a^U(Z,L)$ if and only if for all selfadjoint $D\in\mc L_{\mc A}(\mb L)$ the conditions $[D,\tilde U(g)]=0$ for all $g\in G$ and (\ref{Kextehto}) yield $D=0$.
\end{theor}

\begin{proof}
Let $K\in{\bf Ker}_\a^U(Z,L)$ with a minimal covariant Kolmogorov decomposition $(\mb M,R,\tilde U)$. We note that the condition \eqref{Kextehto} equals $\sis{R(x)v}{DR(y)w}=0$ for all $(x,y)\in Z$ and $v,\,w\in\mb V$ by polarization.

Assume that $D\in\mc L_{\mc A}(\mb M)$ is self-adjoint and non-zero and satisfies the conditions of the claim. Replacing $D$ by $\|D\|^{-1}D$, we may assume that $-I_{\mb M}\leq D\leq I_{\mb M}$. Define positive kernels $K^\pm$ via
$$
K_{x,y}^\pm(v,w)=\sis{R(x)v}{D^\pm R(y)w},\qquad x,\,y\in X,\quad v,\,w\in\mb V,
$$
where $D^\pm=I_{\mb M}\pm D$. For any $(x,y)\in Z$, by using \eqref{Kextehto}, one gets
\begin{eqnarray*}
K_{x,y}^\pm(v,w)&=&\sis{R(x)v}{R(y)w}\pm\sis{R(x)v}{DR(y)w}=\sis{R(x)v}{R(y)w}\\
&=&K_{x,y}(v,w)=L_{x,y}(v,w)
\end{eqnarray*}
for all $v,\,w\in\mb V$. Let $g\in G$, $x,\,y\in X$ and $v,\,w\in\mb V$. Since $\tilde U(g)^* D^\pm\tilde U(g)=D^\pm$, we may write
\begin{eqnarray*}
K_{gx,gy}^\pm(v,w)&=&\sis{R(gx)v}{D^\pm R(gy)w}\\
&=&\sis{R(x)U(g^{-1})v\a(g,x)}{\tilde U(g)^* D^\pm\tilde U(g)R(y)U(g^{-1})w\a(g,y)} \\
&=&\a(g,x)^*\sis{R(x)U(g^{-1})v}{D^\pm R(y)U(g^{-1})w}\a(g,y)\\
&=&\a(g,x)^*K_{x,y}^\pm\big(U(g^{-1})v,U(g^{-1})w\big)\a(g,y).
\end{eqnarray*}
Hence, $K^\pm\in{\bf Ker}_\a^U(Z,L)$ and, obviously, $K=\frac12K^++\frac12K^-$. Since $D\neq0$ the minimality of the decomposition $(\mb M,R,\tilde U)$ yields $K^+\neq K^-$ so that $K$ is not extreme.

Suppose that $K$ is regular and $K=\frac12K^++\frac12K^-$ with $K^\pm\in{\bf Ker}_\a^U(Z,L)$ and $K^+\neq K^-$. Now $K^\pm\leq 2K$ and, according to Lemma 1 of \cite{pe13}, there are positive operators $D^\pm\in\mc L_{\mc A}(\mb M)$ such that $K_{x,y}^\pm(v,w)=\sis{R(x)v}{D^\pm R(y)w}$ for all $x,\,y\in X$ and $v,\,w\in\mb V$. Define a self-adjoint $D=D^+-D^-\ne 0$ since $K^+\neq K^-$. For all $(x,y)\in Z$ and $v\in\mb V$ one gets $\sis{R(x)v}{DR(y)v}=K_{x,y}^+(v,v)-K_{x,y}^-(v,v)=L_{x,y}(v,v)-L_{x,y}(v,v)=0.$ Suppose that $\eta=\sum_{j=1}^nR(x_j)v_j$ is contained in the dense subspace of (ii) of Proposition \ref{Kdilat} and $g\in G$. Now
\begin{eqnarray*}
&&\sis{\eta}{\tilde U(g)^* D^\pm\tilde U(g)\eta}\\
&=&\sum_{j,k=1}^n\sis{\tilde U(g)R(x_j)v_j}{D^\pm\tilde U(g)R(x_k)v_k} \\
&=&\sum_{j,k=1}^n\sis{R(gx_j)U(g)v_j\a(g,x_j)^{-1}}
{D^\pm R(gx_k)U(g)v_k\a(g,x_k)^{-1}}\\
&=&\sum_{j,k=1}^n[\a(g,x_j)^{-1}]^*K_{gx_j,gx_k}^\pm\big(U(g)v_j,U(g)v_k\big)\a(g,x_k)^{-1}\\
&=&\sum_{j,k=1}^nK_{x_j,x_k}^\pm(v_j,v_k)=\sis{\eta}{D^\pm\eta}.
\end{eqnarray*}
Due to the density, it follows that $D^\pm=\tilde U(g)^* D^\pm\tilde U(g)$ and $[D,\tilde U(g)]=0$ for all $g\in G$.

The last claim concerning a $W^*$-minimal dilation $(\mb L,R,\tilde U)$ when $\mc A$ is a $W^*$-algebra is proven in exactly the same way as above. Note that now we have automatically a necessary and sufficient condition for extremality because $\mb L$ is self-dual.
\end{proof}

\begin{rem}\label{symrem}
{\rm
Note that when the set $Z\subset X\times X$ is symmetric, i.e.\ when $(x,y)\in Z$ then also $(y,x)\in Z$, the extremality condition of the above theorem can be made stronger: {\it Let $K\in{\bf Ker}_\a^U(Z,L)$ with $(\mb M,R,\tilde U)$. If $K$ is extreme in ${\bf Ker}_\a^U(Z,L)$ then, for any $D\in\mc L_{\mc A}(\mb M)$, the conditions $[D,\tilde U(g)]=0$ for all $g\in G$ and (\ref{Kextehto}) yield $D=0$.} Indeed, assume that $D\in\mc L_{\mc A}(\mb M)$ satisfies these conditions. Using the symmetry of $Z$, it is immediately seen that also $D^*$ satisfies these conditions and hence also the self-adjoint operators $A=\frac12(D+D^*)$ and $A'=\frac{i}{2}(D^*-D)$. If  $K$ is extreme then, according to the preceding theorem, it follows that $A=0=A'$ and $D=A+iA'=0$. We may similarly simplify the extremality condition when $Z$ is symmetric, $\mc A$ is a $W^*$-algebra and $K\in {\bf Ker}_\a^U(Z,L)$ is equipped with its $W^*$-minimal covariant Kolmogorov decomposition.
}
\end{rem}

\begin{rem}
{\rm
Note that, in the context of Theorem \ref{Kcovarext}, the regular $K$ is an extreme point of the larger convex set consisting of all positive kernels $K'$ such that $K'_{x,y}=L_{x,y}$, $(x,y)\in Z$, if and only if for all self-adjoint $D\in\mc L_{\mc A}(\mb M)$, the condition
\begin{equation}\label{Kextehto2}
\sis{R(x)v}{DR(y)v}=0,\qquad(x,y)\in Z,\quad v\in\mb V
\end{equation}
implies $D=0$. This follows from Theorem 2 of \cite{pe13} which is actually a special case of Theorem \ref{Kcovarext} above (put $G=\{e\}$). The same result applies to the $W^*$-case.
}
\end{rem}

\section{Completely positive maps}\label{CPmaps}

Denote the group of *-automorphisms of a $C^*$-algebra $\mc B$ by $\mr{Aut}(\mc B)$. In the case of a unital $\mc B$, an automorphism $\beta\in\mr{Aut}(\mc B)$ is called {\it inner} if there is a unitary $u\in\mc B$ such that $\beta(b)=ubu^*$. Denote the subgroup of inner automorphisms on $\mc B$ by $\mr{Inn}(\mc B)$. Suppose that $G$ is a group. We call group homomorphisms $G\ni g\mapsto\beta_g\in\mr{Aut}(\mc B)$ as {\it $G$-actions on $\mc B$} and homomorphisms $G\ni g\mapsto\beta_g\in\mr{Inn}(\mc B)$ as {\it inner $G$-actions on $\mc B$}.

\begin{rem}\rm \label{inner}
Suppose that $\beta$ is an inner $G$-action in a unital $C^*$-algebra $\mc B$, i.e.,\ $\beta_g(b)=u_gbu_g^*$ for all $b\in\mc B$ and $g\in G$ where $u_g\in\mc U_{\mc B}$. Without restricting generality, we may (and will) assume that $u_e=1_{\mc B}$. Since, for all $g,\,h\in G$ and $b\in\mc B$, $u_{gh}bu_{gh}^*=\beta_{gh}(b)=\beta_g\big(\beta_h(b)\big)=u_{g}u_{h}bu_{h}^*u_{g}^*$ equals $[u_{gh}^*u_{g}u_{h},b]=0$ it follows that $u_{gh}=m(g,h)u_{g}u_{h}$ where $m(g,h)\in\mc B$ is unitary and belongs to the center of $\mc B$. Moreover, $m:\,G\times G\to\mc U_{\mc B}\cap Z_{\mc B}$ is a 2-cocycle and $g\mapsto u_g$ is a $m^*$-multiplier representation with the 2-cocycle $m^*$ (defined by $m^*(g,h):=m(g,h)^*$). 
For example, if $\mc B$ is the $W^*$-algebra $\mc L(\mc K)$ of bounded operators on a Hilbert space $\mc K$ then the center is $\mb C I_\kil$ and one can consider $m$ as a $\mb T$-valued cocycle so that $g\mapsto u_g$ is a projective unitary representation of $G$ on $\kil$. Recall that any automorphism of $\lk$ is inner, $\mr{Aut}(\lk)=\mr{Inn}(\lk)$ \cite{pedersen} and thus any $G$-action on $\lk$, regardless of the group $G$, is inner.
\end{rem}

In what follows, $\mc A$ is a $C^*$-algebra, $\mb V$ is an $\mc A$-module, $\mc B$ is a unital $C^*$-algebra, $G$ is a group, $U:G\to\mr{GL}_{\mc A}(\mb V)$ is a representation and
$g\mapsto\beta_g$ is a $G$-action on $\mc B$.

\begin{ma}\label{kovarCP}
{\rm
We say that a map $S:\mc B\to S_{\mc A}(\mb V)$ is {\it completely positive (CP)} if it is linear and the map $K^S=\big[\mc B\times\mc B\ni(b,c)\mapsto K^S_{b,c}\in S_{\mc A}(\mb V)\big]$, $K^S_{b,c}(v,w):=S_{b^*c}(v,w)$ for all $b,\,c\in\mc B$ and $v,\,w\in\mb V$, is a positive kernel (see \cite[Section 4]{KSGNS}). Moreover, $S$ is {\it $(\beta,U)$-covariant} if
$$
S_{\beta_g(b)}(v,w)=S_b\big(U(g^{-1})v,U(g^{-1})w\big)
$$
for all $g\in G$, $b\in\mc B$ and $v,\,w\in\mb V$ (implying that the kernel $K^S$ is also covariant with respect the $G$-action on the $G$-space $\mc B$ and the representation $U$). We denote the set of $(\beta,U)$-covariant CP maps by ${\bf CP}_\beta^U$.
}
\end{ma}

Note that, according to the definition above, a linear map $S:\mc B\to S_{\mc A}(\mb V)$ is CP if and only if for any $n\in\mb N$, $b_1,\ldots,\,b_n\in\mc B$, and $v_1,\ldots,\,v_n\in\mb V$ one has
$$
\sum_{j,k=1}^nS_{b_j^*b_k}(v_j,v_k)\geq0.
$$
This definition is analogous with the classical definition of complete positivity for a linear map $\Phi:\mc B\to\mc L(\hil)$ between $C^*$-algebra $\mc B$ and the algebra of bounded operators on a Hilbert space $\hil$: for any $n\in\mb N$, $b_1,\ldots,\,b_n\in\mc B$, and $\f_1,\ldots,\,\f_n\in\hil$ we require that
$$
\sum_{j,k=1}^n\sis{\f_j}{\Phi(b_j^*b_k)\f_k}\geq0,
$$
or, equivalently, that the amplifications $(b_{jk})_{j,k=1}^n\mapsto (\Phi(b_{jk}))_{j,k=1}^n$ between matrix $C^*$-algebras are positive.

In physics, we are typically interested in a special class of CP maps, {\it normal CP maps}. Recall that a linear map $f:\mc B\to\mc A$, where $\mc B$ and $\mc A$ are $W^*$-algebras, is said to be normal when it is continuous with respect to the $\sigma$-weak topologies of $\mc A$ and $\mc B$ or, equivalently, whenever $\{s_\lambda\}_{\lambda\in\mc L}\subset\mc B$ is a bounded increasing net of self-adjoint elements, then $\sup_{\lambda\in\mc L}f(s_\lambda)=f\big(\sup_{\lambda\in\mc L}s_\lambda\big)$.

\begin{ma}\label{kovarNCP}\rm
In the case of $W^*$-algebras $\mc A$ and $\mc B$, we say that a map $S\in{\bf CP}_\beta^U$ is {\it normal} if, for all $v\in\mb V$, the map $\mc B\ni b\mapsto S_b(v,v)\in\mc A$ is
normal. We denote the set of normal elements of ${\bf CP}_\beta^U$ by ${\bf NCP}_\beta^U$.
\end{ma}

\begin{theor}\label{CPdilat}
Let $S\in{\bf CP}_\beta^U$. (a) There is a Hilbert $\mc A$-module $(\mb M,\sis{\,\cdot\,}{\,\cdot\,})$, a unital *-homomorphism $\pi:\mc B\to\mc L_{\mc A}(\mb M)$, a representation $\tilde U:G\to\mc U_{\mc A}(\mb M)$ and  $J\in\mr{Lin}_{\mc A}(\mb V;\mb M)$ such that
\begin{itemize}
\item[{\rm (i)}] $S_b(v,w)=\sis{Jv}{\pi(b)Jw}$ for all $b\in\mc B$ and $v,\,w\in\mb V$,
\item[{\rm (ii)}] $\mr{lin}_{\mb C}\{\pi(b)Jv\,|\,b\in\mc B,\ v\in\mb V\}$ is dense in $\mb M$,
\item[{\rm (iii)}] $JU(g)=\tilde U(g)J$ for all $g\in G$ and
\item[{\rm (iv)}] $\tilde U(g)\pi(b)=[\pi\circ\beta_g](b)\tilde U(g)$ for all $g\in G$ and $b\in\mc B$.
\end{itemize}
If $(\mb M',\pi',J',\tilde U')$ is another quadruple satisfying the conditions (i)--(iv), then there is a unitary $W\in\mr{Lin}_{\mc A}(\mb M;\mb M')$ such that $WJ=J'$, $W\pi(b)=\pi'(b)W$ for all $b\in\mc B$, and $W\tilde U(g)=\tilde U'(g)W$ for all $g\in G$.

(b) Suppose that $\mc A$ is a $W^*$-algebra. There is a self-dual Hilbert $\mc A$-module $(\mb L,\sis{\,\cdot\,}{\,\cdot\,})$ equipped with the topology $\tau_1$ (resp.\ $\tau_2$) introduced in Remark \ref{vonNeumann1}, a unital *-homomorphism $\pi:\mc B\to\mc L_{\mc A}(\mb L)$, a unitary representation $\tilde U:G\to\mc U_{\mc A}(\mb L)$ and $J\in\mr{Lin}_{\mc A}(\mb V;\mb L)$ such that (i), (iii), and (iv) of (a) hold but (ii) is replaced by
\begin{itemize}
\item[{\rm (ii)}] $\mr{lin}_{\mb C}\{\pi(b)Jv\,|\,b\in\mc B,\ v\in\mb V\}$ is $\tau_1$- (resp.\ $\tau_2$-) dense in $\mb L$.
\end{itemize}
If $(\mb L',\pi',J',\tilde U')$ is another quadruple satisfying the conditions (i)--(iv), then there is a unitary $W\in\mr{Lin}_{\mc A}(\mb L;\mb L')$ such that $WJ=J'$, $W\pi(b)=\pi'(b)W$ for all $b\in\mc B$, and $W\tilde U(g)=\tilde U'(g)W$ for all $g\in G$. 

(c) Assume that $\mc A$ and $\mc B$ are $W^*$-algebras, and $(\mb L,\pi,J,\tilde U)$ satisfies (i)--(iv) of (b). Then $S\in{\bf NCP}_\beta^U$ if and only if $\pi$ is normal.

(d) Let $(\mb M,\sis{\,\cdot\,}{\,\cdot\,})$ be a Hilbert $\mc A$-module, $\pi:\mc B\to\mc L_{\mc A}(\mb M)$ a unital *-homomorphism, $\tilde U:G\to\mc U_{\mc A}(\mb M)$ a representation, and  $J\in\mr{Lin}_{\mc A}(\mb V;\mb M)$ such that (iii) and (iv) holds. Then $S$ defined by (i) belongs to ${\bf CP}_\beta^U$.
\end{theor}

\begin{proof}
Let $S\in{\bf CP}_\beta^U$. (a) The existence of a triplet $(\mb M,\pi,J)$ with the properties (i) and (ii) has been shown in \cite[Lemma 4.2, Theorem 4.3]{KSGNS}. Indeed, if $(\mb M,R,\tilde U)$ is a minimal $(\beta,U)$-covariant decomposition for the positive kernel $K^S$ defined in Definition \ref{kovarCP}, one can choose $J=R(1_{\mc B})$ and the *-homomorphism $\pi$ can be defined via $\pi(b)R(c)=R(bc)$ for all $b,c\in\mc B$; see the proof of Theorem 4.3 of \cite{KSGNS}. Since $R(b)U(g)=\tilde U(g)[R\circ\beta_{g^{-1}}](b)$ for all $b\in \mc B$ and $g\in G$, item (iii) follows readily:
$$
JU(g)=R(1_{\mc B})U(g)=\tilde U(g)[R\circ\beta_{g^{-1}}](1_{\mc B})=\tilde U(g)R(1_{\mc B})=\tilde U(g)J,\qquad g\in G.
$$
In addition, for all $g\in G$, $b,\,c\in\mc B$ and $v\in\mb V$,
\begin{eqnarray*}
\tilde U(g)\pi(b)\pi(c)Jv&=&\tilde U(g)R(bc)v=[R\circ\beta_g](bc)U(g)v\\
&=&[\pi\circ\beta_g](b)[R\circ\beta_g](c)U(g)v \\
&=&[\pi\circ\beta_g](b)\tilde U(g)R(c)v=[\pi\circ\beta_g](b)\tilde U(g)\pi(c)Jv
\end{eqnarray*}
so that (iv) follows from (ii). 

Let $(\mb M',\pi',J',\tilde U')$ be another quadruple satisfying the conditions (i)--(iv) and denote $R'(b):=\pi'(b)J'$ for all $b\in\mc B$. By Theorem \ref{Kdilat}, there is a  
unitary $W\in\mr{Lin}_{\mc A}(\mb M;\mb M')$ such that $W R(b)=R'(b)$, $b\in\mc B$, and $W\tilde U(g)=\tilde U'(g)W$, $g\in G$. Especially, $WJ=J'$ and, for all $b,\,c\in\mc B$ and $v\in\mb V$,
$$
W\pi(b)\pi(c)Jv=WR(bc)v=R'(bc)v=\pi'(b)\pi'(c)J'v=\pi'(b)W\pi(c)Jv
$$
showing that $W\pi(b)=\pi'(b)W$ for all $b\in\mc B$.

(b) Assume that $\mc A$ is a $W^*$-algebra and denote the self-dual Hilbert $\mc A$-module $\mb M^*$ by $\mb L$ so that $\pi$ and $\tilde U$ obtained above extend uniquely to a unital
*-homomorphism $\mc B\to\mc L_{\mc A}(\mb L)$ and, respectively, to a representation $G\to\mc U_{\mc A}(\mb L)$ which we continue to denote by $\pi$ and, respectively, by $\tilde U$. The properties (i)--(iv) are immediately seen to hold as well as the uniqueness claim. 

(c) Suppose that also $\mc B$ is a $W^*$-algebra and the quadruple $(\mb L,\pi,\tilde U,J)$ satisfies the conditions (i)--(iv) of (b). Let $\{s_\lambda\}_{\lambda\in\mc L}$ be a bounded increasing net of self-adjoint elements of $\mc B$ and let $\sup s_\lambda\in\mc B$ be its least upper bound so that  $s_\lambda\to\sup s_\lambda$ in the $\sigma$-weak topology. Since $\{\pi(s_\lambda)\}_{\lambda\in\mc L}$ is increasing and bounded by $\pi(\sup s_\lambda)$, it has the least upper bound $\sup\pi(s_\lambda)\in\mc L_{\mc A}(\mb L)$ and $\pi(s_\lambda)\to\sup\pi(s_\lambda)$ ($\sigma$-weakly). Let $\ell\in\mb L$ and $\f\in\mc A_*$. Then $\f(\sis{\ell}{\pi(s_\lambda)\ell})\to\f(\sis{\ell}{\sup\pi(s_\lambda)\ell})$ since $\mc L_{\mc A}(\mb L)\ni B\mapsto\f(\sis{\ell}{B\ell})\in\mb C$ is a positive $\sigma$-weakly continuous functional. Hence, $\sis{\ell}{\pi(s_\lambda)\ell}\to\sis{\ell}{\sup\pi(s_\lambda)\ell}$ in the $\sigma$-weak topology of $\mc A$. Especially, if $\pi:\,\mc B\to\mc L_{\mc A}(\mb L)$ is normal (i.e.\ $\sigma$-weakly continuous) then $\sup\pi(s_\lambda)=\pi\big(\sup s_\lambda\big)$ and $\mc B\ni b\mapsto S_b(v,v)=\sis{Jv}{\pi(b)Jv}\in\mc A$ is normal for all $v\in\mb V$. Suppose then that $S$ is normal. By polarization, $\sis{\ell}{\pi(s_\lambda)\ell}=\sum_{j,k=1}^n\sis{Jv_j}{\pi(b_j^*s_\lambda b_k)Jv_k}\to\sum_{j,k=1}^n\sis{Jv_j}{\pi(b_j^*\sup s_\lambda b_k)Jv_k}=\sis{\ell}{\pi(\sup s_\lambda)\ell}$ in the $\sigma$-weak topology of $\mc A$ for all $\ell=\sum_{j=1}^n\pi({b_j})Jv_j$ so that $\sis{\ell}{\sup\pi(s_\lambda)\ell}=\sis{\ell}{\pi(\sup s_\lambda)\ell}$ for all $\ell$ contained in the $\tau_1/\tau_2$-dense subset $\mr{lin}_{\mb C}\{\pi(b)Jv\,|\,b\in\mc B,\ v\in\mb V\}$ of $\mb L$. Thus, $\sup\pi(s_\lambda)=\pi\big(\sup s_\lambda\big)$ and $\pi$ is normal.

Finally, item (d) can be verified by direct calculation. 
\end{proof}

\begin{rem}\rm \label{rem:kosykliCP}
One could try to generalize the preceding notion of covariant CP maps by using a map $\a:\,G\times\mc B\to\mr{Inv}_{\mc A}\cap\mc Z_{\mc A}$ such that $\a(gh,b)=\m(g,h)\a(h,b)\a\big(g,\beta_h(b)\big)$ where $\m:\,G\times G\to\mc U_{\mc A}\cap\mc Z_{\mc A}$ is a 2-cocycle.
In this framework, one can say that a linear map $\mc B\ni b\mapsto S_b\in S_{\mc A}(\mb V)$ is a $(\a,\beta,U)$-covariant CP map if  $K^S$ is a $(\mc B,\a,U)$-covariant positive kernel. By writing $b^*c=1_{\mc B}\cdot b^*c$ one sees that \eqref{covacond} can be written as
\begin{eqnarray*}
&&\a(g,b)^*\a(g,c)S_{b^*c}(U(g^{-1})v,U(g^{-1})w)=S_{\beta_g(b^*c)}(v,w) \\
&=&\a(g,1_{\mc B})^*\a(g,b^*c)S_{b^*c}(U(g^{-1})v,U(g^{-1})w)
\end{eqnarray*}
for all $g\in G$, $b,\,c\in\mc B$, and $v,\,w\in\mb V$. Moreover, $S$ and $S\circ\beta_g$ are (positive) linear maps in equation $S_{\beta_g(c)}(v,v)=\a(g,1_{\mc B})^*\a(g,c)S_{c}(U(g^{-1})v,U(g^{-1})v)$ so that  it is reasonable to assume that $\a$ does not depend on its second argument, i.e.\ $\alpha$ can be viewed as a map from $G$ to $\mr{Inv}_{\mc A}\cap\mc Z_{\mc A}$ such that $\a(gh)=\m(g,h)\a(g)\a(h)$. But then
\begin{eqnarray*}
S_{\beta_g(b)}(v,w)&=&S_b\big(U(g^{-1})v\a(g),U(g^{-1})wa(g)\big)\\
&=&S_b\big(U'(g^{-1})v,U'(g^{-1})w\big)
\end{eqnarray*}
where $U'(g):=\a(g^{-1}).U(g)$ is a multiplier representation of $G$. If we assume that $U$ is a multiplier representation in Definition \ref{kovarCP} then the only modification in Theorem \ref{CPdilat} is that $\tilde U$ becomes a unitary multiplier representation (with the 2-cocyle of $U$) and can be extended to a unitary representation of the central extension group of $G$, see Remark \ref{central}.
\end{rem}

In the context of Theorem \ref{CPdilat} (either in the general (a)-case or in the $W^*$-case (b)), assume that the exists a $g\in G$ such that $\beta_g$ is inner, that is, there is a unitary $u_g\in\mc B$ such that $\beta_g(b)=u_gbu_g^*$. Define a unitary map
\begin{equation}\label{ovltilde}
\ovl U(g)=\pi\big(u_g^*\big)\tilde U(g)\in\mc U_{\mc A}(\mb M).
\end{equation}
Now $\ovl U(g)J=\pi\big(u_g^*\big)\tilde U(g)J=\pi\big(u_g^*\big)JU(g)$ and, for all $b\in\mc B$,
\begin{eqnarray*}
\ovl U(g)\pi(b)&=&\pi\big(u_g^*\big)\tilde U(g)\pi(b)=\pi\big(u_g^*\big)[\pi\circ\beta_g](b)\tilde U(g)\\
&=&\pi\big(u_g^*u_gbu_g^*\big)\tilde U(g)=\pi(b)\ovl U(g).
\end{eqnarray*}
Suppose that $\beta$ is an inner $G$-action, i.e.,\ $\beta_g(b)\equiv u_gbu_g^*$, and let $m:\,G\times G\to\mc U_{\mc B}\cap Z_{\mc B}$ be the 2-cocycle associated with $g\mapsto u_g$ (see, Remark \ref{inner}). Define a map
$$
M:G\times G\to\mc U_{\mc A}(\mb M)\cap Z_{\pi(\mc B)},\qquad(g,h)\mapsto M(g,h)=[\pi\circ m](g,h)
$$ 
where $Z_{\pi(\mc B)}$ is  the center of the range $\pi(\mc B)\subseteq\mc L_{\mc A}(\mb M)$ of $\pi$; especially $M$ commutes with $\ovl U$. Clearly, $M$ is a 2-cocycle. Using results derived above, one finds that
\begin{eqnarray*}
\ovl U(gh)&=&\pi(u_{gh}^*)\tilde U(gh)=\pi\big(m(g,h)^*u_h^*u_g^*\big)\tilde U(g)\tilde U(h)\\
&=&M(g,h)^*\pi(u_h^*)\pi(u_g^*)\tilde U(g)\tilde U(h)=M(g,h)^*\pi(u_h^*)\ovl U(g)\tilde U(h)\\
&=&M(g,h)^*\ovl U(g)\pi(u_h^*)\tilde U(h)=M(g,h)^*\ovl U(g)\ovl U(h)
\end{eqnarray*}
or $\ovl U(g)\ovl U(h)=M(g,h)\ovl U(gh)$ for all $g,\,h\in G$. Hence, $\ovl U$ is an $M$-multiplier representation. Immediately we get the following corollary:

\begin{cor}\label{covacoro}
Assume that $S\in{\bf CP}_\beta^U$ where $g\mapsto\beta_g$ is inner, i.e.,\ there is a map $G\ni g\mapsto u_g\in\mc U_{\mc B}$, such that $\beta_g(b)=u_gbu_g^*$ and $u_{gh}=m(g,h)u_gu_h$ where $m:\,G\times G\to\mc U_{\mc B}\cap Z_{\mc B}$ is a 2-cocycle. There exists a triple $(\mb M,\pi,J)$ of Theorem \ref{CPdilat}(a) satisfying the conditions (i) and (ii) of Theorem \ref{CPdilat}(a) and, additionally, there is an $M$-multiplier representation $\ovl U:G\to\mc U_{\mc A}(\mb M)$ with the 2-cocycle $M:\,G\times G\to\mc U_{\mc A}(\mb M)$, $(g,h)\mapsto M(g,h)=[\pi\circ m](g,h)$, such that
\begin{itemize}
\item[{\rm (iii)'}] $\pi\big(u_g^*)JU(g)=\ovl U(g)J$ for all $g\in G$,
\item[{\rm (iv)'}] $\ovl U(g)\pi(b)=\pi(b)\ovl U(g)$ for all $g\in G$ and $b\in\mc B$.
\end{itemize}
If $(\mb M',\pi',J',\ovl U')$ is another quadruple satisfying the conditions (i), (ii), (iii)' and (iv)' then there is a unitary $W\in\mr{Lin}_{\mc A}(\mb M;\mb M')$ such that $WJ=J'$, $W\pi(b)=\pi'(b)W$ for all $b\in\mc B$, and $W\ovl U(g)=\ovl U'(g)W$ for all $g\in G$.

If $\mc A$ is a $W^*$-algebra and $(\mb L,\pi,J)$ is associated to $S$ by Theorem \ref{CPdilat}(b), then the similar result as above holds when one replaces $\mb M$ by $\mb L$.
\end{cor}

\begin{ma}
{\rm
Suppose that $S\in{\bf CP}_\beta^U$. We call a quadruple $(\mb M,\pi,J,\tilde U)$ satisfying conditions (i)--(iv) of Theorem \ref{CPdilat} as a {\it minimal $(\beta,U)$-covariant dilation} or {\it KSGNS (Kasparov, Stinespring, Gel'fand, Na\u{\i}mark, Segal) construction} for $S$. In the case of an inner action in Corollary \ref{covacoro},  $(\mb M,\pi,J,\ovl U)$ is also called minimal covariant dilation. In both cases, the triple $(\mb M,\pi,J)$ is simply called a {\it minimal dilation} or {\it KSGNS construction}. These minimal dilations are unique up to unitary equivalence. We say that a CP map $S\in{\bf CP}_\beta^U$ with a minimal dilation $(\mb M,\pi,J)$ is {\it regular} if $\mb M$ is self-dual. When $\mc A$ is a $W^*$-algebra and $S\in{\bf CP}_\beta^U$, we call the quadruples $(\mb L,\pi,\tilde U/\ovl U,J)$ of Theorem \ref{CPdilat} and Corollary \ref{covacoro} as {\it $W^*$-minimal $(\beta,U)$-covariant dilations} or {\it KSGNS constructions} for $S$. Again, the triple $(\mb L,\pi,J)$ is simply called a {\it $W^*$-minimal dilation} or {\it KSGNS construction}.
}
\end{ma}

\begin{rem}\rm\label{central2}
In the context of Corollary \ref{covacoro}, one can assume that the 2-cocycle $m$ is trivial, $m(g,h)\equiv1_{\mc B}$, by replacing $G$ with  the central extension group 
$G^{m^*}:=G\times(\mc U_{\mc B}\cap\mc Z_{\mc B})$ where the group product is $(g,t)(h,u):=\big(gh,m(g,h)^*tu\big)$, $g,\,h\in G$, $t,\,u\in\mc U_{\mc B}\cap\mc Z_{\mc B}$ (compare to Remark \ref{central}). Then the $m^*$-multiplier representation $g\mapsto u_g$ is replaced by the representation $(g,t)\mapsto u_gt$ and $U$ (resp.\ $\beta$) is extended to the representation $(g,t)\mapsto U'(g,t)=U(g)$ (resp.\ inner action $(g,t)\mapsto \beta'(g,t)=u_gbu_g^*=\beta(g)$). Now
\begin{eqnarray*}
S_{\beta'_{(g,t)}(b)}(v,w)&=&S_{\beta_g(b)}(v,w)=S_b\big(U(g^{-1})v,U(g^{-1})w\big)\\
&=&S_b\big(U'((g,t)^{-1})v,U'((g,t)^{-1})w\big).
\end{eqnarray*}
If $U$ is a multiplier representation as well, then we may further extend the group $G^{m^*}$ as in Remark \ref{central} so that one can restrict to ordinary representations $g\mapsto u_g$ and $U$.
\end{rem}

\subsection{Extreme points}

Let $G$ be a group, $\mc B$ a unital $C^*$-algebra with the $G$-action $g\mapsto\beta_g$, $\mb V$ a module over a $C^*$-algebra $\mc A$, and $U:G\to\mr{GL}_{\mc A}(\mb V)$ a representation of $G$ on $\mb V$. Fix a positive sesquilinear form $s_1\in S_{\mc A}(\mb V)$ and denote the set of those $S\in{\bf CP}_\beta^U$ such that $S_{1_{\mc B}}=s_1$ by ${\bf CP}_\beta^U(s_1)$. In the context of $W^*$-algebras, we denote the set of normal elements of ${\bf CP}_\beta^U(s_1)$ by ${\bf NCP}_\beta^U(s_1)$. (We always assume that ${\bf CP}_\beta^U(s_1)$ and ${\bf NCP}_\beta^U(s_1)$ are nonempty.) Clearly the above sets are convex. Since for all $S\in{\bf CP}_\beta^U(s_1)$, $g\in G$, and $v,\,w\in\mb V$ one has $s_1\big(U(g)v,U(g)w\big)=S_{1_{\mc B}}\big(U(g)v,U(g)w\big)=S_{\beta_{g^{-1}}(1_{\mc B})}(v,w)=S_{1_{\mc B}}(v,w)=s_1(v,w)$, the sesquilinear form $s_1$ has to be {\it invariant} in the sense that
$$
s_1\big(U(g)v,U(g)w\big)=s_1(v,w)
$$
for all $g\in G$ and $v,\,w\in\mb V$. The following theorem generalizes the classical extremality result \cite[Theorem 1.4.6]{arveson}.

\begin{theor}\label{Scovarext}Let $S\in{\bf CP}_\beta^U(s_1)$ with a minimal $(\beta,U)$-covariant dilation $(\mb M,\pi,J,\tilde U)$. If $S$ is an extreme point of ${\bf CP}_\beta^U(s_1)$ then, for any $D\in\mc L_{\mc A}(\mb M)$, the conditions $[D,\pi(b)]=0$ for all $b\in\mc B$, $[D,\tilde U(g)]=0$ for all $g\in G$, and
\begin{equation}\label{Sextehto}
\sis{Jv}{DJv}=0,\qquad v\in\mb V
\end{equation}
imply $D=0$. If $S$ is regular then also the converse holds.

Suppose that $\mc A$ is a $W^*$-algebra and $(\mb L,\pi,J,\tilde U)$ is a $W^*$-minimal $(\beta,U)$-covariant dilation of $S$. The map $S$ is extreme in ${\bf CP}_\beta^U(s_1)$ if and only if for any $D\in\mc L_{\mc A}(\mb L)$ the conditions $[D,\pi(b)]=0$ for all $b\in\mc B$, $[D,\tilde U(g)]=0$ for all $g\in G$ and (\ref{Sextehto}) yield $D=0$. If also $\mc B$ is a $W^*$-algebra and $S\in{\bf NCP}_\beta^U(s_1)$ then $S$ is extreme in ${\bf NCP}_\beta^U(s_1)$ if and only if $S$ is extreme in ${\bf CP}_\beta^U(s_1)$.
\end{theor}

\begin{proof}
The proof largely parallels that given for Theorem \ref{Kcovarext}, and thus we only give here a brief outline of the proof. The existence of a non-zero operator $D$ with the properties of the claim can easily be used to show the existence of a non-trivial convex decomposition. Suppose now that $S$ is regular and not extreme in ${\bf CP}_\beta^U(s_1)$, i.e.,\ $S^\pm\in{\bf CP}_\beta^U(s_1)$ for which $S^+\neq S^-$ and $S=\frac12S^++\frac12S^-$. When we view $S$ and $S^\pm$ as positive kernels and $S$ is regular, we may apply Lemma 1 of \cite{pe13} and obtain operators $D^\pm\in\mc L_{\mc A}(\mb M)$ such that $D^+\neq D^-$ and
$$
S_{b^*c}^\pm(v,w)=\sis{Jv}{\pi(b)^* D^\pm\pi(c)Jw},\qquad b,\,c\in\mc B,\quad v,\,w\in\mb V.
$$
Using the fact that $S^\pm_{b^*cd}=S^\pm_{(c^*b)^*d}$ one finds $\sis{\pi(b)Jv}{[D^\pm,\pi(c)]\pi(d)Jw}=0$ for all $v,\,w\in\mb V$ and $b,\,c,\,d\in\mc B$. Applying the minimality of the dilation $(\mb M,\pi,J,\tilde U)$, we obtain $[D^\pm,\pi(b)]=0$ for all $b\in\mc B$. It follows directly from the proof of Theorem \ref{Kcovarext} that $[D^\pm,\tilde U(g)]=0$ for all $g\in G$. For $D:=D^+-D^-\ne0$ one gets $[D,\pi(b)]=0$, $b\in\mc B$, $[D,\tilde U(g)]=0$, $g\in G$, and $\sis{Jv}{DJv}=0$ for all $v\in\mb V$.

As in the proof of Theorem \ref{Kcovarext}, the last claims concerning the $W^*$-case follow easily. To see this, note that for all $S\in{\bf NCP}_\beta^U$, $S'\in{\bf CP}_\beta^U$, and $\lambda>0$ such that $\lambda S'\leq S$ one has $S'\in{\bf NCP}_\beta^U$.
\end{proof}

\begin{rem}\rm
In the context of the preceding theorem (either in the general or $W^*$-case), assume that $S$ is extreme. Immediately we see that, for any $E\in\mc L_{\mc A}(\mb M)$ such that 
$[E,\pi(b)]=0$, $b\in\mc B$, and $[E,\tilde U(g)]=0$, $g\in G$,
$$
\sis{EJv}{EJv}=\sis{Jv}{E^*EJv}=0,\qquad v\in\mb V
$$
implies $E=0$. But $\sis{EJv}{EJv}=0$ if and only if
$$
\|EJv\|=\sqrt{\|\sis{EJv}{EJv}\|_{\mc A}}=0
$$
if and only $EJv=0$. Hence,  $EJv=0$, $v\in \mb V$, imply $E=0$ and one sees that
the intersection of the center of $\pi(\mc B)$, the center of $\tilde U(G)$ and the annihilator of the range of $J$ is trivial.
\end{rem}

\begin{rem}\label{ovlext}
{\rm
Suppose that the action $\beta$ is inner and let $S\in{\bf CP}_\beta^U(s_1)$ be regular with $(\mb M,\pi,\ovl U,J)$ of Corollary \ref{covacoro}. Immediately one gets the following variant of Theorem \ref{Scovarext}:} The map $S$ is extreme in ${\bf CP}_\beta^U(s_1)$ if and only if, for any $D\in\mc L_{\mc A}(\mb M)$, the conditions $[D,\pi(b)]=0$, $b\in\mc B$, $[D,\ovl U(g)]=0$, $g\in G$, and $\sis{Jv}{DJv}=0$ for all $v\in\mb V$ imply $D=0$. {\rm The variant related to a $W^*$-minimal dilation is obvious. Setting $G=\{e\}$ one sees that a regular CP map $S:\mc B\to S_{\mc A}(\mb V)$ with a minimal dilation $(\mb M,\pi,J)$ is extreme in the set of CP maps $T:\mc B\to S_{\mc A}(\mb V)$ such that $T_{1_{\mc B}}(v,w)=s_1(v,w)$ for all $v,\,w\in\mb V$ with a fixed $s_1\in S_{\mc A}(\mb V)$ if and only if for any $D\in\mc L_{\mc A}(\mb M)$ the conditions $[D,\pi(b)]=0$ for all $b\in\mc B$ and $\sis{Jv}{DJv}=0$ for all $v\in\mb V$ imply $D=0$. The variant of this result in the $W^*$-case is obvious.}
\end{rem}

\section{Marginal maps and subminimal dilations}\label{marginal}

In this section, $\mc A$ is a $C^*$-algebra and $\mb V$ is an $\mc A$-module. Let $\mc B$ and $\mc C$ be $W^*$-algebras so that we may define their $\sigma$-weak tensor product $\mc B\otimes\mc C$ which is a $W^*$-algebra as well. We also fix a group $G$ and $G$-actions $G\ni g\mapsto\beta_g\in\mr{Aut}(\mc B)$ and $G\ni g\mapsto\gamma_g\in\mr{Aut}(\mc C)$; these actions may well be trivial. The tensor product $\mc B\otimes\mc C$ is equipped with the $G$-action $g\mapsto\delta_g:=\beta_g\otimes\gamma_g$. We also assume that $U:G\to\mr{GL}_{\mc A}(\mb V)$ is a representation. Fix a positive $\mc A$-sesquilinear form $s_1\in S_{\mc A}(\mb V)$ such that $s_1(U(g)v,U(g)w)=s_1(v,w)$ for all $g\in G$ and $v,\,w\in\mb V$.

\begin{ma}
{\rm
Suppose that $S\in{\bf CP}_\delta^U(s_1)$. Define the {\it marginals} $S^1\in{\bf CP}_\beta^U(s_1)$ and $S^2\in{\bf CP}_\gamma^U(s_1)$ of $S$ through $S^1_b=S_{b\otimes1_{\mc C}}$, $b\in\mc B$, and $S^2_c=S_{1_{\mc B}\otimes c}$, $c\in\mc C.$
}
\end{ma}

\noindent
It is clear that the marginals $S^1$ and $S^2$ of $S\in{\bf CP}_\delta^U(s_1)$ are covariant CP maps. Thus, $S^1$ and $S^2$ have their own covariant minimal dilations according to Theorem \ref{CPdilat}. 

For any Hilbert $\mc A$-module $(\mb M,\sis{\,\cdot\,}{\,\cdot\,})$ and a $C^*$-algebra $\mc D$, we say that a linear map $F:\mc D\to\mc L_{\mc A}(\mb M)$ is completely positive if the corresponding sesquilinear-form-valued map $\mc D\ni d\mapsto S^F_d\in S_{\mc A}(\mb M)$, where $S^F_d(v,w):=\sis{v}{F(d)w}$, $v,\,w\in\mb M$, is completely positive.

\begin{prop}\label{submin}
Let $S\in{\bf CP}_\delta^U(s_1)$ and denote its marginals by $S^1$ and $S^2$. Suppose that $S^1$ is regular and $(\mb M,\pi,K,\tilde U)$ is a minimal $(\beta,U)$-covariant dilation of $S^1$. There is a unique (unital) CP map $E:\mc C\to\mc L_{\mc A}(\mb M)$ such that $\tilde U(g)E(c)=[E\circ\gamma_g](c)\tilde U(g)$, $\pi(b)E(c)=E(c)\pi(b)$ and $S_{b\otimes c}(v,w)=\sis{Kv}{\pi(b)E(c)Kw}$ for all $g\in G$, $b\in\mc B$, $c\in\mc C$ and $v,\,w\in\mb V$. Assume that $\mc A$ is a $W^*$-algebra and $(\mb L,\pi,K,\tilde U)$ is a $W^*$-minimal $(\beta,U)$-covariant dilation of $S^1$. Again, there is a unique (unital) CP map $E:\mc C\to\mc L_{\mc A}(\mb L)$ such that $\tilde U(g)E(c)=[E\circ\gamma_g](c)\tilde U(g)$, $\pi(b)E(c)=E(c)\pi(b)$ and $S_{b\otimes c}(v,w)=\sis{Kv}{\pi(b)E(c)Kw}$ for all $g\in G$, $b\in\mc B$, $c\in\mc C$ and $v,\,w\in\mb V$. Moreover, the second marginal $S^2$ is normal if and only if $E$ is normal.
\end{prop}

\begin{proof}
Pick a positive $c\in\mc C$ and define a CP map $S^c$ by $S^c_b:=S_{b\otimes c}$ for all $b\in\mc B$. It follows that $\|c\|S^1_b-S^c_b=S_{b\otimes(\|c\|1_{\mc C}-c)}$ for all $b\in\mc B$ and thus $S^c\leq\|c\|S^1$ since $\|c\|1_{\mc C}-c\geq0$. Hence, following the proof of Theorem \ref{Scovarext}, one see that there is a unique operator $E(c)\in\mc L_{\mc A}(\mb M)$, commuting with $\pi$, such that $S_{b\otimes c}(v,w)=\sis{Kv}{\pi(b)E(c)Kw}$ for all $b\in\mc B$ and $v,\,w\in\mb V$. Since any $c\in\mc C$ can be expressed in a unique way as a combination of its positive parts and $(\mb M,\pi,K)$ is minimal, we obtain a unique linear map $E:\mc C\to\mc L_{\mc A}(\mb M)$ such that it commutes with $\pi$ and $S_{b\otimes c}(v,w)=\sis{Kv}{\pi(b)E(c)Kw}$ for all $b\in\mc B$, $c\in\mc C$ and $v,\,w\in\mb V$.

Let $n\in\mb N$, $j\in\{1,2,.\ldots,n\}$, and $\eta_j\in\mr{lin}_{\mb C}\{\pi(b)Kv\,|\,b\in\mc B,\ v\in\mb V\}$,   i.e., there are $n_j\in\mb N$, $b_{jl}\in\mc B$ and $v_{jl}\in\mb V$, $l=1,\ldots,n_j$, such that $\eta_j=\sum_{l=1}^{n_j}\pi(b_{jl})Kv_{jl}$. Since $S$ is completely positive it follows that, for all $c_j\in\mc C$, $j=1,\ldots,n$,
\begin{eqnarray*}
\sum_{i,j=1}^n\sis{\eta_i}{E(c_i^*c_j)\eta_j}&=&\sum_{i,j=1}^n\sum_{k=1}^{n_i}\sum_{l=1}^{n_j}\sis{Kv_{ik}}{\pi(b_{ik}^*b_{jl})E(c_i^*c_j)Kv_{jl}} \\
&=&\sum_{i,j=1}^n\sum_{k=1}^{n_i}\sum_{l=1}^{n_j}S_{(b_{ik}\otimes c_i)^*(b_{jl}\otimes c_j)}(v_{ik},v_{jl})\geq0
\end{eqnarray*}
and $E$ is completely positive due to the minimality of  $(\mb M,\pi,K)$. 

Finally, for all $\eta=\sum_{k=1}^{n}\pi(b_{k})Kv_{k}\in\mr{lin}_{\mb C}\{\pi(b)Kv\,|\,b\in\mc B,\ v\in\mb V\}$, $g\in G$ and $c\in\mc C$, one finds
\begin{eqnarray*}
&&\sis{\eta}{\tilde U(g)^* E(c)\tilde U(g)\eta}\\
&=&\sum_{j,k=1}^n\sis{Kv_j}{\pi(b_j^*)\tilde U(g)^* E(c)\tilde U(g)\pi(b_k)Kv_k}\\
&=&\sum_{j,k=1}^n\sis{KU(g)v_j}{[\pi\circ\beta_g](b_j^*)E(c)[\pi\circ\beta_g](b_k)KU(g)v_k}\\
&=&\sum_{j,k=1}^n\sis{KU(g)v_j}{[\pi\circ\beta_g](b_j^*b_k)E(c)KU(g)v_k}\\
&=&\sum_{j,k=1}^nS_{\beta_g(b_j^*b_k)\otimes c}\big(U(g)v_j,U(g)v_k\big)\\
&=&\sum_{j,k=1}^nS_{\delta_g\big(b_j^*b_k\otimes\gamma_{g^{-1}}(c)\big)}\big(U(g)v_j,U(g)v_k\big)\\
&=&\sum_{j,k=1}^nS_{b_j^*b_k\otimes\gamma_{g^{-1}}(c)}(v_j,v_k)=\sis{\eta}{[E\circ\gamma_{g^{-1}}](c)\eta}
\end{eqnarray*}
and, hence, $\tilde U(g)^* E(c)\tilde U(g)=[E\circ\gamma_{g^{-1}}](c)$ or equivalently
$$
\tilde U(g)E(c)=[E\circ\gamma_g](c)\tilde U(g),\qquad g\in G,\quad c\in\mc C.
$$

When $\mc A$ is a $W^*$-algebra, the remaining claims can be proven in exactly the same way. To prove the last claim, note that for any positive $b\in\mc B$ also the map ${}^bS:\mc C\to S_{\mc A}(\mb V)$, ${}^bS_c=S_{b\otimes c}$, is bounded by $\|b\|S^2$ so that, if $S^2$ is normal, then ${}^bS$ is normal for any $b\in\mc B$. Again, for all $\ell=\sum_{j=1}^n\pi(b_j)Kv_j$ the map $c\mapsto\sis{\ell}{E(c)\ell}=\sum_{j,k=1}^n\sis{Kv_j}{\pi(b_j^*)E(c)\pi(b_k)Kv_k}=\sum_{j,k=1}^nS_{b_j^*b_k\otimes c}(v_j,v_k)$ is normal and, as in the proof of Theorem \ref{CPdilat}, it follows that $E$ is normal. The converse claim follows immediately.

Finally we note that, since $E$ is unique, $E(1_{\mc C})=I_{\mb M}$.
\end{proof}

\begin{rem}\label{ovlsubmin}
\rm
In the context of the preceding proposition, assume that the action $\beta$ is inner, i.e.,\ there is a multiplier representation $u:\,G\to\mc U_{\mc B}$ such that $\beta_g(b)\equiv u_gbu_g^*$. Assume that the first (regular) marginal $S^1$ has the minimal covariant dilation $(\mb M,\pi,K,\ovl U)$ of Corollary \ref{covacoro}. Now, there is a unique CP map $E:\mc C\to\mc L_{\mc A}(\mb M)$ such that $\ovl U(g)E(c)=[E\circ\gamma_g](c)\ovl U(g)$, $\pi(b)E(c)=E(c)\pi(b)$ and $S_{b\otimes c}(v,w)=\sis{Jv}{\pi(b)E(c)Jw}$ for all $g\in G$, $b\in\mc B$, $c\in\mc C$ and $v,\,w\in\mb V$. This result has an obvious counterpart in the $W^*$-case.
\end{rem}

\begin{ma}
{\rm
We call the collection $(\mb M,\pi,E,K,\tilde U)$ of the preceding proposition (or $(\mb M,\pi,E,K,\ovl U)$ in the case of an inner action) as an {\it $\mc A$-subminimal $(\beta,U)$-covariant dilation of $S$}. We can define the $\mc B$-subminimal dilations in the same way if $S^2$ is regular. When $\mc A$ is a $W^*$-algebra, we call $(\mb L,\pi,E,K,\tilde U)$ (or $(\mb L,\pi,E,K,\ovl U)$ in the case of an inner action) as an {\it $\mc A$-subminimal $(\beta,U)$-covariant $W^*$-dilation of $S$}.
}
\end{ma}

Note that, for a subminimal dilation $(\mb M,\pi,E,K,\tilde U)$ of $S$, the mapping $\rho:\mc B\otimes\mc C\to\mc L_{\mc A}(\mb M),\;b\otimes c\mapsto \pi(b)E(c)$ is not necessarily a *-representation. Hence, $(\mb M,\rho,K,\tilde U)$ typically fails to be a true KSGNS construction for $S$.

Let $s\in S_{\mc A}(\mb V)$, $S\in{\bf CP}_\beta^U(s_1)$, and $S'\in{\bf CP}_\gamma^U(s_1)$. If there exists a map $T\in{\bf CP}_\delta^U(s_1)$ such that $S$ and $S'$ coincide with the marginals of $T$, i.e., $S=T^1$ and $S'=T^2$, we say that $S$ and $S'$ are {\it compatible}. Moreover, $T$ is called a {\it joint map} of $S$ and $S'$. Simple modifications to the proofs presented in \cite{haapasalo2} show that if $S\in{\bf CP}_\beta^U(s_1)$ and $S'\in{\bf CP}_\gamma^U(s_1)$ are compatible and either one of them is an extreme point then their joint map $T$ is unique. Moreover, if both of the compatible maps are extreme points then also their unique joint map is extreme.

\begin{rem}\label{subminrem}\rm
Suppose that $S\in{\bf NCP}_\delta^U(s_1)$ and denote its marginals by $S^1\in{\bf NCP}_\beta^U(s_1)$ and $S^2\in{\bf NCP}_\gamma^U(s_1)$. Let $(\mb M_1,\pi_1,E_1,K_1,\tilde U_1)$ be a $\mc B$-subminimal dilation for $S$ where $(\mb M_1,\pi_1,K_1,\tilde U_1)$ is a minimal $(\beta,U)$-covariant dilation of $S^1$ and $\pi_1$ and $E_1$ are normal. Let $(\mb H_1,\sigma_1,C_1,W_1)$ be a minimal $(\gamma,\tilde U_1)$-covariant dilation of $E_1$, i.e.,\ $\mb H_1$ is a Hilbert $\mc A$-module, $\sigma_1:\mc C\to\mc L_{\mc A}(\mb H_1)$ is a normal unital *-homomorphism, $C_1:\mb M_1\to\mb H_1$ is an $\mc A$-linear isometry and $W_1:G\to\mc L_{\mc A}(\mb H_1)$ is a unitary representation such that $E_1(c)=C_1^*\sigma_1(c)C_1$ for all $c\in\mc C$, the $\mb C$-linear combinations of vectors $\sigma_1(c)C_1m$, $c\in\mc C$, $m\in\mb M_1$, form a dense subset of $\mb H_1$, $C_1\tilde U_1(g)=W_1(g)C_1$ for all $g\in G$ and $W_1(g)\sigma_1(c)=[\sigma_1\circ\gamma_g](c)W_1(g)$ for all $g\in G$ and $c\in\mc C$. Using familiar techniques, one sees that we may define a normal unital *-homomorphism $\rho_1:\mc B\to\mc L_{\mc A}(\mb H_1)$ through $\rho_1(b)\sigma_1(c)C_1m=\sigma_1(c)C_1\pi_1(b)m$ for all $b\in\mc B$, $c\in\mc C$ and $m\in\mb M_1$. Clearly, $\rho_1$ and $\sigma_1$ commute and we may typically define the unique normal unital *-homomorphism $\Pi_1:\mc B\otimes\mc C\to\mc L_{\mc A}(\mb H_1)$ such that $\Pi_1(b\otimes c)=\rho_1(b)\sigma_1(c)$ for all $b\in\mc B$ and $c\in\mc C$.\footnote{The extension of the product of two commuting normal *-representations into a normal *-representation of the tensor product $\mc B\otimes\mc C$ is not always possible but for the von Neumann algebras typically encountered in quantum theory this is possible. Especially, in the situation studied in subsequent sections where one of the two algebras is a type-I factor, the product of the commuting normal *-representations can be chosen to be the tensor product of the identity representation of the type-I factor and a normal *-representation of the other algebra thus being a normal *-representation of $\mc B\otimes\mc C$ because of the special structure of normal *-representations of type-I factors \cite[Lemma 2.2, Chapter 9]{davies}.} Since the set $\mr{lin}_{\mb C}\{\pi_1(b)K_1v\,|\,b\in\mc B,\ v\in\mb V\}$ is dense in $\mb M_1$ and $\mr{lin}_{\mb C}\{\sigma_1(c)C_1m\,|\,c\in\mc C,\ m\in\mb M_1\}$ is dense in $\mb H_1$, it follows that, by defining the $\mc A$-linear map $Y_1=C_1K_1$, the set $\mr{lin}_{\mb C}\{\Pi_1(b\otimes c)Y_1v\,|\,b\in\mc B,\ c\in\mc C,\ v\in\mb V\}=\mr{lin}_{\mb C}\{\sigma_1(c)C_1\pi_1(b)K_1v\,|\,b\in\mc B,\ c\in\mc C,\ v\in\mb V\}$ is dense in $\mb H_1$. Hence, we have obtained a minimal $(\delta,U)$-covariant dilation $(\mb H_1,\Pi_1,Y_1,W_1)$ for $S$ arising from an intermediate subminimal dilation. One easily sees that $W_1(g)\Pi_1(d)=[\Pi_1\circ\delta_g](d)W_1(g)$ for all $g\in G$ and $d\in\mc B\otimes\mc C$.

The procedure described above can be slightly varied if, e.g.,\ the action $\beta$ is inner so that we may use the representation $\ovl U_1$ of Corollary \ref{covacoro} instead of $\tilde U_1$; this is indeed what we will do later in the case of quantum instruments. Again, if also $\mc A$ is a $W^*$-algebra, the above reasoning works in the context of $W^*$-dilations.

Suppose that we apply the above `dilation in stages' method also to the $\mc C$-subminimal dilation of $S$. Thus we obtain a Hilbert $\mc A$-module $\mb H_2$, unital *-representations $\rho_2:\mc B\to\mc L_{\mc A}(\mb H_2)$ and $\sigma_2:\mc C\to\mc L_{\mc A}(\mb H_2)$ so that we may define $\Pi_2:\mc B\otimes\mc C\to\mc L_{\mc A}(\mb H_2)$, $\Pi_2(b\otimes c)=\rho_2(b)\sigma_2(c)$, a unitary representation $W_2:G\to\mc L_{\mc A}(\mb H_2)$ such that $W_2(g)\Pi_2(d)=[\Pi_2\circ\delta](d)W_2(g)$ for all $g\in G$ and $d\in\mc B\otimes\mc C$ and an $\mc A$-linear map $Y_2:\mb V\to\mb H_2$ such that $(\mb H_2,\Pi_2,W_2,Y_2)$ is a minimal $(\delta,U)$-covariant dilation for $S$. Hence, there is an ($\mc A$-linear) unitary operator $U_{12}:\,\mb H_1\to\mb H_2$ which intertwines $\Pi_1$ with $\Pi_2$ and $W_1$ with $W_2$. 
\end{rem}

\section{CP maps in quantum theory}\label{quantumCP}
The standard description for a (symmetric) measurement in quantum mechanics can be given in terms of a (covariant) instrument whose definition is the following:

\begin{ma}
\label{ma:instrumentti}
{\rm
Let $\mc V$ and $\mc K$ be Hilbert spaces and $(\Om,\Sigma,\nu)$ is a measure space, where $\nu:\Sigma\to[0,\infty]$ is $\sigma$-finite. Let $G$ be a measurable group with a measurable action on $\Om$, i.e.,\ a measurable map $G\times\Om\ni(g,\om)\mapsto g\om\in\Om$ such that $e\om=\om$ and $(gh)\om=g(h\om)$ for all $g,\,h\in G$ and $\om\in\Om$. Moreover, $\nu$ is quasi-$G$-invariant, i.e.,\ $\nu(gX)=0$ for every $g\in G$ whenever $X\in\Sigma$ is $\nu$-null. This action gives rise to a $G$-action $g\mapsto\gamma_g$, $[\gamma_g(f)](\om):=f(g^{-1}\om)$ for all $f\in L^\infty(\nu)$. Let $U:G\to\mc U(\mc V)$ be a multiplier representation of $G$ and $g\mapsto\beta_g$ an inner $G$-action on $\mc L(\mc K)$ arising from a multiplier representation $g\mapsto u_g$, $\beta_g(b)=u_gbu_g^*$ for all $g\in G$ and $b\in\mc L(\mc K)$. A normal $(\beta\otimes\gamma,U)$-covariant CP map $S:\mc L(\mc K)\otimes L^\infty(\nu)\to S_{\mb C}(\mc V)$ such that $S_{I_{\mc K}\otimes1}(v,w)=\sis{v}{w}$ for all $v,\,w\in\mc V$ is called a {\it covariant instrument} whose first marginal $S^1$ is the {\it associated channel} and the second marginal $S^2$ is the {\it associated observable}.
}
\end{ma}

\begin{rem}\label{rem:kvanttiCP}
{\rm
By denoting the characteristic function of $X\in\Sigma$ by $\chi_X$ and writing
\begin{equation}\label{eq:instr}
\sis{v}{\Gamma(b,X)w}=S_{b\otimes\chi_X}(v,w)
\end{equation}
the above implies\footnote{Also the converse holds if $\mc V$ is separable, namely any instrument $\Gamma$ is absolutely continuous with respect to the (quasi-invariant) probability measure $\nu(X)=\tr{\rho\Gamma(I_\kil,X)}$, $X\in \Sigma$, where $\rho$ is a faithful positive trace-1 operator on $\mc V$ \cite{Pe13a}. Now Equation \eqref{eq:instr} defines an instrument $S$ of Definition \ref{ma:instrumentti}.} the usual definition of an instrument \cite{Pe13a}: a $(\beta\otimes\gamma,U)$-covariant instrument is a map $\Gamma:\mc L(\mc K)\times\Sigma\to\mc L(\mc V)$ such that
\begin{enumerate}
\item for all $X\in\Sigma$, the map $\mc L(\mc K)\ni b\mapsto\Gamma(b,X)\in\mc L(\mc V)$ is CP and normal,
\item $\Gamma(I_{\mc K},\Om)=I_{\mc V}$,
\item $\tr{T\Gamma(b,\cup_{j=1}^\infty X_j)}=\sum_{j=1}^\infty\tr{T\Gamma(b,X_j)}$ for all trace-class operators $T$ on $\mc V$, all $b\in\mc L(\mc K)$, and any disjoint sequence $\{X_j\}_{j=1}^\infty\subset\Sigma$, and
\item $\Gamma(u_gbu_g^*,gX)=U(g)\Gamma(b,X)U(g)^*$ for all $g\in G$, $b\in\mc L(\mc K)$, and $X\in\Sigma$.
\end{enumerate}
In Section \ref{covariantobs}, we denote the set of such maps by $\mc I_u^U(\Om,\Sigma)$ or simply by $\mc I_u^U(\Om)$. The first marginals of these instruments are covariant quantum channels, i.e.,\ normal unital CP maps $\mc E:\mc L(\mc K)\to\mc L(\mc V)$ such that $\mc E(u_gbu_g^*)=U(g)\mc E(b)U(g)^*$ for all $g\in G$ and $b\in\mc L(\mc K)$. The second marginals of the instruments are covariant positive operator measures $\ms M:\Sigma\to\mc L(\mc V)$, i.e.,\ for any unit vector $v\in\mc V$, the map $X\mapsto\sis{v}{\ms M(X)v}$ is a probability measure and $\ms M(gX)=U(g)\ms M(X)U(g)^*$ for all $g\in G$ and $X\in\Sigma$. In section \ref{covariantobs}, we denote the set of these covariant observables by $\mc O^U(\Om,\Sigma)$ or $\mc O^U(\Om)$. If the group or the group actions are trivial then one gets, as special cases, the non-covariant instruments, channels, and observables.
}
\end{rem}

In what follows, we give generalizations for the quantum CP maps introduced above by replacing the Hilbert space $\mc V$ of Definition \ref{ma:instrumentti} by a module $\mb V$ over a $C^*$-algebra. {\it For the rest of this section including the following subsections}, we assume that $\hil$ and $\kil$ are Hilbert spaces, $\mb V$ an $\lh$-module, $G$ is a group, $U:\,G\to\mr{GL}_{\lh}(\mb V)$ is a representation, and $g\mapsto\beta_g$ a $G$-action on $\lk$. It follows that, for all $g,\,h\in G$, $b\in\lk$, $\beta_g(b)=u_gbu_g^*$ where $u_g\in\lk$ is unitary, and $u_{gh}=m(g,h)u_gu_h$ where $m$ is  a $\mb T$-valued 2-cocycle, see Remark \ref{inner}. Suppose that $s_1\in S_{\mc L(\hil)}(\mb V)$ is positive (and invariant). First, let us concentrate on a generalization of channels.

\begin{prop}\label{lemmmm}
Let $S\in{\bf NCP}_\beta^U(s_1)$. There exists a Hilbert space $\kil'$ and an $\lh$-linear map $J:\,\mathbb V\to\mc L(\hil;\kil\otimes\kil')$ such that
$$
S_b(v,w)=(Jv)^*(b\otimes I_{\kil'})(Jw)
$$
for all $b\in\lk$ and $v,\,w\in\mb V$, and ${\rm lin}_{\mb C}\{(b\otimes I_{\kil'})(Jv)\,|\,b\in\lk,\;v\in\mb V\}$ is $\sigma$-strongly dense in $\mc L(\hil;\kil\otimes\kil')$. Moreover, there is an $m$-multiplier representation $G\ni g\mapsto u'_g\in\mc U(\kil')$ such that $JU(g)=(u_g\otimes u'_g)J$ for all $g\in G$. The CP map $S$ is extreme in ${\bf NCP}_\beta^U(s_1)$ if and only if, for any $D'\in\mc L(\mc K')$, the conditions $[D',u'_g]=0$ for all $g\in G$ and $(Jv)^*(I_\kil\otimes D')(Jv)=0$ for all $v\in\mb V$ imply $D'=0$. We also have a `Kraus-type' decomposition \cite{kraus}
$$
S_b(v,w)=\sum_{\lambda\in\mc L}(\mathsf A_\lambda v)^*b(\mathsf A_\lambda w),\qquad b\in\lk,\quad v,\,w\in\mb V,
$$
where $\mc L$ is a set and $\ms A_\lambda:\mb V\to\mc L(\hil;\kil)$, $\lambda\in\mc L$, are $\mc L(\hil)$-linear maps and the series converges $\sigma$-strongly.
\end{prop}

\begin{proof}
Let $(\mb L,\rho,J)$ be a $W^*$-minimal dilation for $S$. Since $\mb L$ is a self-dual Hilbert $\lh$-module we may choose $\mb L=\mc L(\hil;\hil')$, where $\hil'$ is a Hilbert space, and identify $\mc L_{\lh}\big(\mc L(\hil;\hil')\big)$ with $\mc L(\hil')$, see Remark \ref{vonNeumann}. Note that the $\tau_1$-topology of $\mb L$ is the $\sigma$-strong operator topology. Now $\rho:\lk\to\mc L(\mc H')$ is normal unital ${}^*$-homomorphism so that $\hil'$ can be expressed as a tensor-product space $\mc K\otimes\mc K'$, where $\mc K'$ is a Hilbert space, and $\rho(b)=b\otimes I_{\mc K'}$ for all $b\in\mc L(\mc K)$ \cite[Lemma 2.2, Chapter 9]{davies}. Furthermore, $\sis{m}{\rho(b)m'}=m^*(b\otimes I_{\mc K'})m'$ for all $m,\,m'\in\mc L(\hil;\mc K\otimes\mc K')$ and $b\in\lk$.

By combining Corollary \ref{covacoro} with Lemma \ref{lemmmm} one sees that there is a map 
$
\ovl U:G\to\mc U_{\lh}\mb(\mc L(\hil;\kil\otimes\kil')\big)\cong\mc U(\kil\otimes\kil')
$ 
such that $\big(u_g^*\otimes I_{\kil'}\big)JU(g)=\ovl U(g)J$, $g\in G$, and $\ovl U(g)(b\otimes I_{\kil'})=(b\otimes I_{\kil'})\ovl U(g)$ for all $g\in G$ and $b\in\mc L(\mc K)$ implying that there exist unitary operators $u'_g\in\mc U(\kil')$ such that $\ovl U(g)=I_\kil\otimes u'_g$ and $JU(g)=(u_g\otimes u'_g)J$ for all $g\in G$. Since $g\mapsto\tilde U(g)=\rho(u_g)\ovl U(g)=u_g\otimes u'_g$ is a group homomorphism, $g\mapsto u'_g$ is a $m$-multiplier representation, i.e.\ $u'_{gh}=\overline{m(g,h)}u'_gu'_h$ for all $g,\,h\in G$. Assume that $D\in\mc L_{\mc L(\hil)}\big(\mc L(\hil;\mc K\otimes\mc K')\big)\cong\mc L(\mc K\otimes\mc K')$ is an operator which commutes with the homomorphisms $\rho$ and $\ovl U$. Hence, $D=I_\kil\otimes D'$ where $D'\in\mc L(\kil')$ and $[D',u'_g]=0$ for all $g\in G$. The extremality characterization now follows from Remark \ref{ovlext}.

To prove the last claim, fix an orthonormal basis $\mc L$ of $\kil'$ to construct bounded operators $a_\lambda:\,\kil\otimes\kil'\to\kil$ via $\langle\psi|a_\lambda\eta\rangle:=\langle\psi\otimes \lambda|\eta\rangle$, $\psi\in\kil$, $\eta\in\kil\otimes\kil'$, where $\lambda\in\mc L$. For each $\lambda\in\mc L$, define an $\lh$-linear map $\mathsf A_\lambda:\,\mb V\to\mc L(\hil;\kil)$ by $(\mathsf A_\lambda v)\psi:=a_\lambda\big((Jv)\psi\big)$, $v\in\mb V$, $\psi\in\hil$. Since for any increasing net $(\mc L_j)_{j\in\mc J}$ of finite subsets of $\mc L$ such that $\cup_{j\in\mc J}\mc L_j=\mc L$ the bounded increasing net $\big(\sum_{\lambda\in\mc L_j}a_\lambda^*a_\lambda\big)_{j\in\mc J}$ converges $\sigma$-strongly to $I_{\mc K\otimes\mc K'}$ and $b a_\lambda\equiv  a_\lambda\rho(b)$, one sees that the Kraus-type decomposition of the claim converges in the $\sigma$-strong operator topology. Note that, if $\kil'$ is separable then $\mc L$ is countable and the above sum is countable.
\end{proof}

We say that $\kil'$ of Proposition \ref{lemmmm} is a {\it minimal ancillary space of S}. It is unique up to a unitary equivalence.

\begin{ex}\rm
Choose $\hil=\mb C$, and let $\mb V$ be a dense subspace of a Hilbert space $\mc V$ and $S:\,\lk\to S_{\lh}(\mb V)$ a normal CP map. Suppose that the sesquilinear form $S_{I_\kil}$ is bounded so that it defines a bounded positive operator $P$ on $\mc V$ via $\langle\f|P\psi\rangle=S_{I_\kil}(\f,\psi)$, $\f,\,\psi\in\mb V$. Hence, any $S_{b}$, $b\in\lk$, is also bounded and $S$ can be viewed as a normal CP map from $\lk$ to $\mc L(\mc V)$. In the case $P\le I_{\mc V}$, $S$ is referred as a quantum operation and, in its special case $P=I_{\mc V}$, as a quantum channel (in this case, $J^*J=I_{\mb V}$, i.e.\ $J$ extends to an isometry) \cite{kraus}.
\end{ex}

\begin{ex}\rm
Let $\hil$, $\kil$,  and $\kil'$ be Hilbert spaces and $\mc L$ an orthonormal basis of $\kil'$. Note that we may assume that $\mc L$ is an arbitrary set and let $\kil'$ be the sequence space $\ell^2(\mc L)$; then any $\lambda\in\mc L$ can be identified with the characteristic function of $\{\lambda\}$. Let $G\ni g\mapsto u_g\in\mc U(\kil)$ be a unitary repsentation of a group $G$, $\beta_g(b)\equiv u_gbu_g^*$, and for each $\lambda\in\mc L$, let $\Phi_\lambda:\,\lk\to\lk$ be a normal CP map which is covariant in the sense that $\Phi_\lambda(u_gbu_g^*)=u_g\Phi_\lambda(b)u_g^*$ for all $g\in G$ and $b\in\lk$. We denote $\Phi=\{\Phi_{\lambda}\}_{\lambda\in\mc L}$ and ${\bf1}=\{1_{\lambda}\}_{\lambda\in\mc L}$ where $1_\lambda(b)\equiv b$, the identity channel.

Define a vector space $\mb V$ as a set of maps $v:\,\mc L\to\mc L(\hil,\kil)$ such that $v(\lambda)\ne 0$ only for finitely many $\lambda\in\mc L$. The addition and scalar product are defined pointwisely. Equip $\mb V$ with the module product $(va)(\lambda):=v(\lambda)a$, $v\in\mb V$, $a\in\mc L(\hil)$, and $\lambda\in\mc L$. Define then $\lh$-linear maps $J:\,\mathbb V\to\mc L(\hil;\kil\otimes\kil')$ and $\mathsf A_\lambda:\,\mb V\to\mc L(\hil;\kil)$ via $(Jv)\psi=\sum_{\lambda\in\mc L}v(\lambda)\psi\otimes\lambda$ and $\mathsf A_\lambda v:=v(\lambda)$. Now $S^\Phi:\,\lk\to S_{\lh}(\mb V)$,
\begin{eqnarray*}
S^\Phi_b(v,w)&:=&\sum_{\lambda\in\mc L}v(\lambda)^*\Phi(b)w(\lambda)=\sum_{\lambda\in\mc L}(\mathsf A_\lambda v)^*\Phi(b)(\mathsf A_\lambda w)\\
&=&(Jv)^*\big(\Phi(b)\otimes I_{\kil'}\big)(Jw),
\end{eqnarray*}
for all $b\in\lk,$ $v,\,w\in\mb V$, is a normal CP map. By defining a representation $U:\,G\to\mr{GL}_{\lh}(\mb V)$ via $[U(g)v](\lambda):=u_gv(\lambda)$, $g\in G$, $v\in\mathbb V$, $\lambda\in\mc L$, one sees that, for all $g\in G$, $v\in\mathbb V$, $\psi\in\hil$,
\begin{eqnarray*}
[JU(g)v]\psi&=&\sum_{\lambda\in\mc L}u_gv(\lambda)\psi\otimes\lambda=(u_g\otimes I_{\kil'})\sum_{\lambda\in\mc L}v(\lambda)\psi\otimes\lambda\\
&=&(u_g\otimes I_{\kil'})(Jv)\psi,
\end{eqnarray*}
that is, $JU(g)=(u_g\otimes I_{\kil'})J$. Hence, $S^\Phi$ is covariant: $S^\Phi_{u_gbu_g^*}(v,w)\equiv S^\Phi_b\big(U(g^{-1})v,U(g^{-1})w\big)$. Especially, if $\Phi$ is a family of covariant channels (i.e.\ $\Phi_\lambda(I_{\kil})=I_{\kil}$) one sees that
$$
S^\Phi_{I_\kil}(v,w)=S^{\bf 1}_{I_\kil}(v,w)=s_1(v,w):=\sum_{\lambda\in\mc L}v(\lambda)^*w(\lambda).
$$
In this case, one can extend the module $\mb V$ to a Hilbert $\lh$-module $\overline{\mb V}$ which consists of maps $v$ such that $\sum_{\lambda\in\mc L}\|v(\lambda)\|^2<\infty$ (and the inner product is $s_1$). Since ${\rm lin}_{\mb C}\{(b\otimes I_{\kil'})(Jv)\,|\,b\in\lk,\;v\in\mb V\}$ is $\sigma$-strongly dense in $\mc L(\hil;\kil\otimes\kil')$, $S^{\bf1}$ is an extreme point of the convex set of all completely positive maps $S:\,\mc L(\kil)\to S_{\mc\lh}(\mb V)$ such that $S_{I_\kil}=s_1$ (and thus an extreme point of ${\bf NCP}_\beta^U(s_1)$). Indeed, suppose that $(Jv)^*(I_\kil\otimes D')(Jv)=0$ for all $v\in\mathbb V$. Thus, by polarization, $(Jv)^*(I_\kil\otimes D')(Jw)=0$ for all $v,\,w\in\mathbb V$. Let $\f\in\kil$ and $\psi\in\hil$ be unit vectors and $\lambda',\,\lambda''\in\mc L$. By choosing $v(\lambda')=|{\f}\rangle\langle{\psi}|$ and $v(\lambda)=0$, $\lambda\ne \lambda'$, and $w(\lambda'')=|{\f}\rangle\langle{\psi}|$ and $w(\lambda)=0$, $\lambda\ne \lambda''$, one gets $\langle\psi|(Jv)^*(I_\kil\otimes D')(Jw)\psi\rangle=\langle\lambda'|D'\lambda''\rangle=0$ implying $D'=0$.
\end{ex}

\subsection{Generalized quantum instruments}\label{genquantinstr}

{\it In addition to the definitions made before Proposition \ref{lemmmm}, let us, from now on, fix the following}: Let $\Omega$ be a set equipped with a $\sigma$-algebra $\Sigma$, and let $\nu:\,\Sigma\to[0,\infty]$ be a $\sigma$-finite measure. Let also $G$ be a group and assume that $\Om$ is equipped with a measurable $G$-action in the sense of Definition \ref{ma:instrumentti} with respect to which $\nu$ is quasi-invariant. Define a $G$-action $G\times L^\infty(\nu)\ni(g,f)\mapsto\gamma_gf$ as in Definition \ref{ma:instrumentti}. As before, let $\delta$ be the $G$-action $\delta_g\equiv\beta_g\otimes\gamma_g$. First we consider generalized instruments without covariance properties (Proposition \ref{propo2}) and then with symmetries (Proposition \ref{paapropo}).

\begin{prop}\label{propo2}
Let $S:\,\lk\otimes L^\infty(\nu)\to S_{\lh}(\mb V)$ be a normal CP map. There exists a Hilbert space $\hil_0$, a spectral measure $\ms F:\,\Sigma\to\mc L(\hil_0)$, and an $\lh$-linear map $Y:\,\mb V\to\mc L(\hil;\kil\otimes\hil_0)$ such that
$$
S_{b\otimes f}(v,w)=(Yv)^*\big[b\otimes\smallint f d\ms F\big](Yw)
$$ 
for all $b\in\lk$, $f\in L^\infty(\nu)$, and $v,\,w\in\mb V$, and
$$
{\rm lin}_{\mb C}\{(b\otimes \ms F(X))(Yv)\,|\,b\in\lk,\,X\in\Sigma,\,v\in\mb V\}
$$
is $\sigma$-strongly dense in $\mc L(\hil;\kil\otimes\hil_0)$.
\end{prop}

\begin{proof}
Note that the first marginal $S^1$ is CP and normal. It follows from Proposition \ref{lemmmm} that there exists a Hilbert space $\kil'$ and an $\lh$-linear map $J:\,\mathbb V\to\mc L(\hil;\kil\otimes\kil')$ such that $S^1_b(v,w)=(Jv)^*(b\otimes I_{\kil'})(Jw)$ for all $b\in\lk$ and $v,\,w\in\mb V$. In addition, the set ${\rm lin}_{\mb C}\{(b\otimes I_{\kil'})(Jv)\,|\,b\in\lk,\;v\in\mb V\}$ is $\sigma$-strongly dense in $\mc L(\hil;\kil\otimes\kil')$. By Proposition \ref{submin} and Remark \ref{vonNeumann1}, there is a unique normal CP map $E:\,L^\infty(\nu)\to\mc L_{\lh}\big(\mc L(\hil;\kil\otimes\kil')\big)\cong \mc L(\kil\otimes\kil')$ such that, for each $b\in\lk$ and $f\in L^\infty(\nu)$, $(b\otimes I_{\kil'})E(f)=E(f)(b\otimes I_{\kil'})$, that is, $E(f)=I_\kil\otimes E'(f)$ where $L^\infty(\nu)\ni f\mapsto E'(f)\in\mc L(\kil')$ is CP and normal. Hence, $E'$ is uniquely determined by the normalized positive operator measure $\ms E':\,\Sigma\to\mc L(\mc K')$, $X\mapsto\ms E'(X):=E'(\chi_X)$ where $\chi_X$ is the characteristic function of $X$, and  $\ms E'$ is absolutely continuous with respect to $\nu$ (i.e.,\ $\nu(X)=0$ implies $\ms E(X)=0$) and $E'(f)$ is the operator integral $\int_\Omega f(\omega)d\ms E'(\omega)$ for all $f\in L^\infty(\nu)$. Furthermore, $S_{b\otimes f}(v,w)=(Jv)^*[b\otimes E'(f)](Jw)$ for all $b\in\lk$, $f\in L^\infty(\nu)$ and $v,\,w\in\mb V$. Let $(\hil_0,F,K)$ be the minimal Na\u{\i}mark dilation of $E'$, that is, $\hil_0$ is a Hilbert space, $F:\,L^\infty(\nu)\to\mc L(\hil_0)$ is a normal unital ${}^*$-homomorphism determined by the spectral measure $\ms F:\,\Sigma\to\mc L(\hil_0)$ via $F(f)\equiv\int_\Omega f(\omega)d\ms F(\omega)$, and $K:\,\kil'\to\hil_0$ is an isometry such that $E'(f)=K^*F(f)K$ for all $f\in L^\infty(\nu)$. In addition, ${\rm lin}_{\mb C}\{F(f)K\psi\,|\,f\in L^\infty(\nu),\,\psi\in\kil'\}$ is dense in $\hil_0$. By denoting $Y:=(I_{\kil}\otimes K)J$, we have proven the claim.
\end{proof}

From now on, we denote by $\rho_g$ the Radon-Nikod\'ym derivative $d\nu_g/d\nu$ of the translated measure $\nu_g(X)=\nu(gX)$ with respect to $\nu$.

\begin{prop}\label{paapropo}
Let $S\in{\bf NCP}_\delta^U(s_1)$ be such that a minimal ancillary space of the first marginal $S^1$ is separable. There is a direct integral Hilbert space $\hil_\oplus=\int_\Om^\oplus\hil_{n(\om)}\,d\nu(\om)$ such that $n(g\omega)=n(\omega)$ for all $g\in G$ and $\omega\in\Om$, and an $\mc L(\hil)$-linear map $Y:\mb V\to\mc L(\hil;\mc K\otimes\hil_\oplus)$ such that
$$
S_{b\otimes f}(v,w)=(Yv)^*\big(b\otimes\hat f\,\big)(Yw),\qquad b\in\mc L(\mc K),\quad f\in L^\infty(\nu),\quad v,\,w\in\mb V,
$$
and the $\mb C$-linear combinations of vectors of the form $(b\otimes\hat f)(Yv)$, $b\in\mc L(\mc K)$, $f\in L^\infty(\nu)$, $v\in\mb V$, form a $\sigma$-strongly dense subspace of $\mc L(\hil;\mc K\otimes\hil_\oplus)$. For any $g\in G$, there is a weakly $\nu$-measurable field $\om\mapsto y(g,\om)$ of unitary operators $y(g,\om):\hil_{n(\om)}\to\hil_{n(\om)}$ such that, when one defines $y_g\in\mc U(\hil_\oplus)$ through
\begin{equation}\label{eq:ygpsi}
(y_g\psi)(\om)=y(g,\om)\psi(g^{-1} \om)\sqrt{\rho_{g^{-1}}(\om)}
\end{equation}
for all $\psi\in\hil_\oplus$ and $\nu$-a.a.\ $\om\in\Om$, one has
$$
YU(g)=(u_g\otimes y_g)Y,\qquad g\in G.
$$
Furthermore, for all $g,\,h\in G$ and $\nu$-a.a.\ $\om\in\Om$
$$
y(gh,\om)=\ovl{m(g,h)}y(g,\om)y(h,g^{-1} \om),
$$
where $m:G\times G\to\mb T$ is the 2-cocycle associated with the multiplier representation $g\mapsto u_g$, and $g\mapsto u_g\otimes y_g$ is a unitary representation. Finally, $S$ is an extreme point of ${\bf NCP}_\delta^U(s_1)$ if and only if, for any decomposable $D=\int_\Omega^\oplus D(\omega)d\nu\in\mc L(\hil_\oplus)$, the conditions $(Yv)^*\big[I_{\kil}\otimes D\big](Yv)\equiv 0$ and $[D(\omega),y(g,\omega)]=0$ for all $g\in G$ and $\omega\in\Omega\setminus{N_g}$, where $N_g\subset\Omega$ is $\nu$-zero measurable, implies $D=0$.
\end{prop}

\begin{proof}
By Proposition \ref{lemmmm}, since $S^1\in{\bf NCP}_\beta^U(s_1)$, there are a Hilbert space $\kil'$, an $\mc L(\hil)$-linear map $J:\mb V\to\mc L(\hil;\kil\otimes\kil')$, and a multiplier representation $G\ni g\mapsto u'_g\in\mc U(\kil')$ with the 2-cocycle $m$ such that we can write $S^1_b(v,w)=(Jv)^*(b\otimes I_{\kil'})(Jw)$ and $JU(g)=(u_g\otimes u'_g)J$. Furthermore,
$$
S_{b\otimes f}(v,w)=(Jv)^*\big[b\otimes\smallint f d\ms E'\big](Jw)%=(Yv)^*\big[b\otimes\smallint f d\ms F\big](Yw)
$$ 
for all $b\in\lk$, $f\in L^\infty(\nu)$, and $v,\,w\in\mb V$, where the operator measure $\ms E'$ is as in the proof of Proposition \ref{propo2} %and the spectral measure $\ms F$ is as in the claim of Proposition \ref{propo2}.
Denote $E'(f)=\int f\,d\ms E'$ for all $f\in L^\infty(\nu)$. As in Remark \ref{ovlsubmin}, it is easy to see that $u'_g E'(f)=[E'\circ\gamma_g](f)u'_g$ or, equivalently, $u'_g\ms E'(X){u'_g}^*=\ms E'(gX)$ for all $g\in G$, $f\in L^\infty(\nu)$, and $X\in\Sigma$, %From Remark \ref{ovlsubmin}, denoting $E(f)=I_\kil\otimes E'(f)$ and $\overline U(g)=I_\kil\otimes u'_g$, one sees that $\ovl U(g)E(f)\equiv[E\circ\gamma_g](f)\ovl U(g)$ implying
%$$
%u'_g E'(f)=[E'\circ\gamma_g](f)u'_g,\qquad g\in G,\; f\in L^\infty(\mu),
%$$
%or, equivalently,
%$$
%u'_g\ms E'(X){u'_g}^*=\ms E'(gX),\qquad g\in G,\; X\in\Sigma,
%$$
i.e.,\ the normalized positive operator measure $\ms E'$ is covariant.

By assumption, $\kil'$ is separable and we have a minimal Na\u{\i}mark dilation $(\hil_\oplus,F,K)$ of $E'$ where $\hil_\oplus=\int_\Om^\oplus\hil_{n(\om)}\,d\nu(\om)$, $F(f)\equiv\hat f$, and $K:\,\mc K'\to\hil_\oplus$ is an isometry \cite{HyPeYl}. Following \cite{cattaneo}, one can define an $m$-multiplier representation $y$, $G\ni g\mapsto y_g\in\mc L(\hil_\oplus)$, such that $K u'_g\equiv y_g K$ for all $g\in G$ and $y_g\hat f=\widehat{\gamma_g(f)}y_g$ for all $g\in G$ and $f\in L^\infty(\nu)$.\footnote{Recall that $y_g$ can be constructed by extending $y_g\hat f K\f:=\widehat{\gamma_g(f)}K u'_g\f$, $g\in G$, $f\in L^\infty(\mu)$, $\f\in\kil'$.} Hence, we have obtained a (projective) imprimitivity system $(\hil_\oplus,\ms F,y)$ whose Hilbert space $\hil_\oplus$ is a direct integral but not necessarily separable. Moreover, we have only assumed that the $G$-action is measurable without any further assumptions on the group $G$.

For all $g,\,h\in G$, define a $\sigma$-homomorphism $T_g:\Om\to\Om$, $\om\mapsto T_g(\om)=g\om$. By using Lemma \ref{covariantoperators} in the appendix, one finds, for all $g\in G$, a weakly measurable field $\om\mapsto y(g,\om)$ of unitary operators $y(g,\om):\hil_{n(g^{-1}\om)}\to\hil_{n(\om)}$ defined as in Equation \ref{eq:ygpsi}. Furthermore,
\begin{eqnarray*}
(y_{gh}\psi)(\om)&=&y(gh,\om)\psi\big((gh)^{-1} \om\big)\sqrt{\rho_{(gh)^{-1}}(\om)}=\ovl{m(g,h)}(y_gy_h\psi)(\om)\\
&=&\ovl{m(g,h)}y(g,\om)y(h,g^{-1} \om)\psi\big((gh)^{-1} \om\big)\sqrt{\rho_{g^{-1}}(\om)\rho_{h^{-1}}(g^{-1} \om)}\\
&=&\ovl{m(g,h)}y(g,\om)y(h,g^{-1} \om)\psi\big((gh)^{-1} \om\big)\sqrt{\rho_{(gh)^{-1}}(\om)}
\end{eqnarray*}
for all $g,\,h\in G$, $\psi\in\hil_\oplus$ and $\nu$-a.a.\ $\om\in\Om$.

Recall that $\lk\otimes L^\infty(\nu)$ acts on $\kil\otimes\hil_\oplus$ and the commutator of $\lk\otimes L^\infty(\nu)$ consists of bounded operators $I_\kil\otimes D=I_\kil\otimes\int_\Omega^\oplus D(\omega) d\nu(\om)$ where $D\in\mc L(\hil_\oplus)$ is decomposable (with operators $D(\omega)\in\mc L(\hil_{n(\omega)})$ for $\nu$-a.a.\ $\omega\in\Omega$). Hence, $S$ is an extreme point if and only if, for any decomposable $D\in\mc L(\hil_\oplus)$ the conditions $[D,y_g]=0$, $g\in G$, and $(Yv)^*\big[I_{\kil}\otimes D\big](Yv)\equiv 0$ implies $D=0$, see Theorem \ref{Scovarext}.
\end{proof}

\begin{rem}\label{rem:orbitit}\rm
In the context of Proposition \ref{paapropo}, we see that, by using the identification
$$
\hil_\oplus=\left[\bigoplus_{n\in\mb N}L^2(\nu_n)\otimes\hil_n\right]\oplus\left[L^2(\nu_\infty)\otimes\hil_\infty\right]
$$
given in \eqref{valuedirect}, one must have, for all $\omega\in\Omega$,
$$
G\omega=\{g\omega\,|\,g\in G\}\subseteq\Omega_{n^\omega}
$$
for some $n^\omega\in\{0\}\cup\mb N\cup\{\infty\}$, that is, any orbit belongs completely to some set $\Omega_n$. For any $n\in\{0\}\cup\mb N\cup\{\infty\}$, the projection $P_n:=\widehat{\chi_{\Omega_n}}$ maps the vectors of $\hil_\oplus$ onto $L^2(\nu_n)\otimes\hil_n$ and we may write, for all $b\in\mc L(\mc K),$ $f\in L^\infty(\nu),$ $v,\,w\in\mb V$,
$$
S_{b\otimes f}(v,w)=\sum_{n\in\mb N}(Yv)^*\big(b\otimes P_n\hat fP_n\big)(Yw)+(Yv)^*\big(b\otimes P_\infty\hat fP_\infty\,\big)(Yw)
$$
(weakly). Hence, $S$ can be viewed as a series of normal $(\beta\otimes\gamma^n,U)$-covariant CP maps $S^n:\,\lk\otimes L^\infty(\nu_n)\to S_{\lh}(\mb V)$,
$$
b\otimes f\mapsto\big[(v,w)\mapsto S^n_{b\otimes f}(v,w):=(Yv)^*\big(b\otimes P_n\hat fP_n\big)(Yw)\big],
$$
where the $G$-action $\gamma^n:\,G\times L^\infty(\nu_n)\to L^\infty(\nu_n)$ is defined similarly as $\gamma$. If an orbit $O:=G\omega$ belongs to $\Sigma$ (or to its $\nu$-completion) and $\nu(O)>0$ then $S$ can be further reduced by defining
$$
S^{O}_{b\otimes f}(v,w):=(Yv)^*\big(b\otimes \widehat{\chi_{O}}\hat f\widehat{\chi_{O}}\big)(Yw)
$$
(so that $S=S^{O}+S^{\Omega\setminus O}$). If $S$ is extreme then any $S^n$ is extreme in the convex set of normal $(\beta\otimes\gamma^n,U)$-covariant CP maps $S'$ (such that $S'_{I_\kil\otimes 1}(v,w)\equiv (Yv)^*(I_\kil\otimes P_n)(Yw)$), and similar result holds for any $S^{O}$. Eventually, we have seen that, by reducing $S$ to its components $S^n$, we may assume that $\hil_\oplus$ (associated to any $S^n$) is of the form $L^2(\nu_n)\otimes\hil_n$ and, if an orbit is measurable, then one can assume the $G$-action to be transitive in the orbital component. 
\end{rem}

Let $S\in{\bf NCP}_\delta^U(s_1)$ with the Na\u{\i}mark dilation space $\hil_0$ of Proposition \ref{propo2}. In the next proposition, {\it we assume that $\hil_0$ and $\kil$ are separable}. Hence, the minimal ancillary space of $S^1$ is separable and we may assume that $\hil_0$ is the direct integral $\hil_\oplus=\int_\Om^\oplus\hil_{n(\om)}\,d\nu(\om)$ of Proposition \ref{paapropo}.
%and consider the second marginal $S^2$ which can be viewed as a normal covariant CP map $\mc L(\mb C)\otimes L^\infty(\nu)\to S_{\lh}(\mb V)$ whose minimal ancillary space is separable (since $S^2_1(v,w)=S_{I_\kil\otimes 1}(v,w)=(Yv)^*(Yw)$). It follows from Proposition \ref{paapropo} (by setting $\mc K=\mb C$) that there are a direct integral Hilbert space $\hil_\oplus^0=\int_\Om^\oplus\hil_{n^0(\om)}\,d\nu(\om)$ with $n^0(g\omega)\equiv n^0(\omega)$ and an $\mc L(\hil)$-linear map $K:\mb V\to\mc L(\hil;\hil_\oplus^0)$ such that $S_f^2(v,w)=(Kv)^*\hat f^0(Kw)$ for all $v,\,w\in\mb V$ and $f\in L^\infty(\nu)$; here we have denoted by $\hat f^0$ the diagonal (multiplicative) operator of $f$ acting in $\hil_\oplus^0$.

\begin{prop}\label{prop:Ckiet}
There is a direct integral Hilbert space
$$
\hil_\oplus^0=\int_\Om^\oplus\hil_{n^0(\om)}\,d\nu(\om)
$$
with $n^0(g\omega)\equiv n^0(\omega)$ and an $\mc L(\hil)$-linear map $K:\mb V\to\mc L(\hil;\hil_\oplus^0)$ such that $S_f^2(v,w)=(Kv)^*\hat f^0(Kw)$ for all $v,\,w\in\mb V$ and $f\in L^\infty(\nu)$, where, for all $f\in L^\infty(\nu)$, $\hat f^0$ is the diagonal multiplicative operator, $(\hat f^0\f)(\om)=f(\om)\f(\om)$, $\f\in\hil_\oplus^0$, defined by $f$. Moreover, the $\mb C$-linear span of the operators $\hat f^0(Kv)$, $f\in L^\infty(\nu)$, $v\in\mb V$, is a $\sigma$-strongly dense subset of $\mc L(\hil;\hil_\oplus^0)$ and there exists a representation $G\ni g\mapsto y^0_g\in\mc U(\hil_\oplus^0)$,
$$
(y_g^0\f)(\om):=y^0(g,\om)\f(g^{-1}\om)\sqrt{\rho_{g^{-1}}(\om)}
$$
for all $g\in G$, $\f\in\hil_\oplus^0$, and $\nu$-a.a.\ $\om\in\Om$, with $KU(g)= y_g^0K$ defined by a map $y^0:\,G\times\Om\to\mc U(\hil_{n^0(\omega)})$ such that, for all $g,\,h\in G$ and $\nu$-a.a.\ $\om\in\Om$, $y^0(gh,\om)=y^0(g,\om)y^0(h,g^{-1} \om)$.

Let the Hilbert space $\hil_\oplus$, the map $Y$, and $g\mapsto y_g$ be as in Proposition \ref{paapropo}. There is an isometry $C:\,\hil_\oplus^0\to\mc K\otimes\hil_\oplus$ such that $(C\f)(\om)=C(\om)\f(\om)$ for all $\f\in\hil_\oplus^0$ and $\nu$-a.a.\ $\om\in\Om$ with a weakly $\nu$-measurable field $\om\mapsto C(\om)$ of isometries $C(\om):\,\hil_{n^0(\om)}\to\mc K\otimes\hil_{n(\om)}$ so that $Y=CK$. Moreover, for all $g\in G$,
\begin{equation}\label{Ckiet}
C(g\om)=\big(u_g\otimes y(g,g\om)\big)C(\om)y^0(g,g\om)^*
\end{equation}
for $\nu$-a.a.\ $\om\in\Om$.
\end{prop}

\begin{proof}
We consider the second marginal $S^2$ which can be viewed as a normal covariant CP map $\mc L(\mb C)\otimes L^\infty(\nu)\to S_{\lh}(\mb V)$ whose minimal ancillary space is separable since, using the notations of Proposition \ref{propo2}, $S^2_f(v,w)=S_{I_\kil\otimes f}(v,w)=(Yv)^*(I_\kil\otimes\smallint f\,d\ms F)(Yw)$ and the non-minimal dilation space $\kil\otimes\hil_0$ is separable. It now follows from Proposition \ref{propo2} (by setting $\mc K=\mb C$) that there is a separable Hilbert space $\hil^0$, a spectral measure $\ms F^0:\Sigma\to\mc L(\hil^0)$ and an $\mc L(\hil)$-linear map $K:\mb V\to\mc L(\hil;\hil^0)$ such that
$$
S_f^2(v,w)=(Kv)^*\smallint f\,d\ms F^0\,(Kw),\qquad f\in L^\infty(\nu),\quad v,\,w\in\mb V.
$$
Since $\hil^0$ is separable, we may identify it with the direct integral $\hil_\oplus^0$ of the claim, and $\smallint f\,d\ms F^0=\hat f^0$.
%It follows from Proposition \ref{paapropo} (by setting $\mc K=\mb C$) that there are a direct integral Hilbert space $\hil_\oplus^0=\int_\Om^\oplus\hil_{n^0(\om)}\,d\nu(\om)$ and an $\mc L(\hil)$-linear map $K:\mb V\to\mc L(\hil;\hil_\oplus^0)$ with the properties stated in the claim. %Moreover, the $\mb C$-linear span of the operators $\hat f^0(Kv)$, $f\in L^\infty(\nu)$, $v\in\mb V$, is a $\sigma$-strongly dense subset of $\mc L(\hil;\hil_\oplus^0)$
Similarly, by Proposition \ref{paapropo} (again setting $\kil=\mb C$), there exists a map $G\times\Om\ni (g,\om)\mapsto y^0(g,\om)\in\mc U(\hil_{n^0(\omega)})$ such that, for all $g,\,h\in G$ and $\nu$-a.a.\ $\om\in\Om$, $y^0(gh,\om)=y^0(g,\om)y^0(h,g^{-1} \om)$ and, defining $G\ni g\mapsto y^0_g\in \mc U(\hil_\oplus^0)$, 
$$
(y_g^0\f)(\om):=y^0(g,\om)\f(g^{-1}\om)\sqrt{\rho_{g^{-1}}(\om)}
$$
for all $g\in G$, $\f\in\hil_\oplus^0$, and $\nu$-a.a.\ $\om\in\Om$, we have $KU(g)= y_g^0K$ and (by definition) $y_g^0\hat f^0=\widehat{\gamma_g(f)}{}^0y_g^0$ for all $g\in G$ and $f\in L^\infty(\nu)$.

According to Remark \ref{subminrem}, there is a (unique) isometry $C:\,\hil_\oplus^0\to\mc K\otimes\hil_\oplus$ such that $C\hat f^0=(I_\kil\otimes\hat f)C$, $Cy_g^0=(u_g\otimes y_g)C$, and $Y=CK$ for all $f\in L^\infty(\nu)$ and $g\in G$. The existence of the field of isometries of the claim follows by using $C\hat f^0=(I_\kil\otimes\hat f)C$ and standard methods (embedding the spaces $\hil_\oplus$ and $\hil_\oplus^0$ in the same space $L^2(\nu)\otimes\hil_\infty$ with decomposable isometries) and by identifying $\mc K\otimes\hil_\oplus$ with $\int_\Om^\oplus\mc K\otimes\hil_{n(\om)}\,d\nu(\om)$. The second condition $Cy_g^0\equiv(u_g\otimes y_g)C$ is equivalent with Equation \eqref{Ckiet}.
\end{proof}

\begin{rem}
{\rm
Let us elaborate the result above. The isometry $C$ defines a quantum channel $\Phi:\,\lk\to\mc L(\hil_\oplus^0)$ via $\Phi(b):=C^*(b\otimes I_{\hil_\oplus})C$, $b\in\lk$, and we may write
\begin{equation}\label{aklsjdgfhlaksjdf}
S_{b\otimes f}(v,w)=(Kv)^*\Phi(b)\hat f^0(Kw),\qquad b\in\mc L(\mc K),\quad f\in L^\infty(\nu),\quad v,\,w\in\mb V.
\end{equation}
Furthermore, $[\Phi(b),\hat f^0]\equiv 0$ and $(y_g^0)^*\Phi(b)y_g^0\equiv\Phi(u_g^*bu_g)$ (i.e.\ the channel $\Phi$ is covariant) or, written in the pointwise form,
$$
y^0(g,g\omega)^*\Phi_{g\omega}(b)y^0(g,g\omega)=\Phi_\omega(u_g^*bu_g)
$$
for all $g\in G$ and $\nu$-a.a.\ $\om\in\Om$ where $\Phi_\omega:\,\lk\to\mc L(\hil_{n(\omega)})$ is a quantum channel defined by $\Phi_\omega(b):=C(\omega)^*(b\otimes I_{\hil_{n(\omega)}})C(\omega)$, $b\in\lk$. Now \eqref{aklsjdgfhlaksjdf} transforms to
\begin{eqnarray}
&&\langle\f| S_{b\otimes f}(v,w)\psi\rangle\nonumber\\
&=&\int_\Omega f(\omega)\big\langle[(Kv)\f](\omega)\big|\Phi_\omega(b)[(Kw)\psi](\omega)\big\rangle d\nu(\omega) \nonumber \\
&=&\int_\Omega f(\omega)\big\langle C(\omega)[(Kv)\f](\omega)\big|(b\otimes I_{\hil_{n(\omega)}})C(\omega)[(Kw)\psi](\omega)\big\rangle d\nu(\omega) \label{CK} \\
&=&\int_\Omega f(\omega)\big\langle [(Yv)\f](\omega)\big|(b\otimes I_{\hil_{n(\omega)}})[(Yw)\psi](\omega)\big\rangle d\nu(\omega) \nonumber
\end{eqnarray}
for all $b\in\mc L(\mc K),$ $f\in L^\infty(\nu),$ $v,\,w\in\mb V$, and $\f,\,\psi\in\hil$, where, e.g.,\ $[(KU(g)w)\psi](\omega)=[y_g^0(Kw)\psi](\omega)=y^0(g,\om)[(Kw)\psi](g^{-1}\om)\sqrt{\rho_{g^{-1}}(\om)}$ and \eqref{Ckiet} holds. Hence, {\it $S$ is completely determined by the space $\hil_\oplus^0$, the map $K$, the representation $g\mapsto y_g^0$, and the covariant decomposable channel $\Phi$,} so that we have obtained a generalization of Theorem 1 of \cite{Pe13b}. Finally, we note that the direct integral space $\hil_\oplus\cong\hil_0$  was assumed to be separable so that the restricted measure space $(\Omega',\Sigma',\nu')$, $\Omega'=\{\omega\in\Omega\,|\,n(\omega)>0\}$, must have a countable basis, see Remark \ref{vonNeumann2}. Clearly, $\hil_\oplus^0\cong C\hil_\oplus^0\subseteq\kil\otimes\hil_\oplus$ is also separable.
}
\end{rem}

From now on, we assume that $\Omega$ (resp.\ $G$) is a locally compact second countable Hausdorff (lcsc) space (resp.\ lcsc group) equipped with the Borel $\sigma$-algebra $\mc B(\Om)$ (resp. $\mc B(G)$) and the $G$-action $G\times\Omega\ni(g,\omega)\mapsto g\omega\in\Omega$ is continuous (and thus jointly Borel-measurable). In addition, the 2-cocycle $m$ is assumed to be trivial by Remark \ref{central2} (note that $m$ is now $\mathbb T$-valued so that $G^{m^*}$ is also a lcsc group and the $G^{m^*}$-action $(g,t)x:=gx$ is also continuous). The lcsc-assumptions imply that the stability subgroups are closed and all orbits are Borel sets, and thus we restrict our attention to a single orbit (of positive measure) and assume that $\Omega$ is a transitive $G$-space. Transitivity of $\Om$ implies that, by picking a reference point $\om_0\in\Om$ and defining the (closed) stability subgroup $H\leq G$ of $\om_0$, $\Om$ is homeomorphic to the space $G/H$, the left $H$-cosets in $G$. Hence, we simply set $\Om=G/H$. We also assume that the map $G\ni g\mapsto u_g\in\mc U(\kil)$ is weakly measurable, which means that, for each $\f,\,\psi\in\kil$, $G\ni g\mapsto\langle\f|u_g\psi\rangle\in\mb C$ is $(\mc B(G),\mc B(\mb C))$-measurable.

There is a Borel section $s:\,\Om\to G$, $s(H)=e$, for the canonical projection $G\to \Om$, $g\mapsto \overline g:=gH$, which from now on shall be fixed \cite[Theorem 5.11]{varadarajan}. For any Hilbert space $\mc M$, we say that a map $y:\,G\times\Om\to\mc U(\mc M)$ is a {\it strict $(G,\Om,\mc M)$-cocycle} if it is (jointly) measurable and $y(e,\om)=I_{\mc M}$ and $y(gh,\om)=y(g,\om)y(h,g^{-1}\om)$ hold for all $g,\,h\in G$ and all $\om\in\Om$. For any Hilbert space $\mc M$ and any weakly measurable representation $\pi:H\to\mc U(\mc M)$, define then the {\it Wigner rotation} $y^\pi:\,G\times\Om\to\mc U(\mc M)$ through $y^\pi(g,\om)=\pi\big(s(\om)^{-1}gs(g^{-1}\om)\big)$ for all $g\in G$ and $\om\in\Om$. One may easily check that $y^\pi$ is a strict $(G,\Om,\mc M)$-cocycle, $\pi(h)=y^\pi(h,H)$, $h\in H$. On the other hand, for any strict cocycle $y$, defining the representations $\pi:H\to\mc U(\mc M)$ and the weakly measurable map $\xi:\Om\to\mc U(\mc M)$ through $\pi(h):=y(h,H)$ and $\xi(\om):=y\big(s(\om),\om\big)$ one finds that $y$ and $y^\pi$ are cohomologous in the sense that $y(g,\om)=\xi(\om)y^\pi(g,\om)\xi(g^{-1}\om)^*$ for all $g\in G$ and $\om\in\Om$.

In the proposition below, we make the same assumptions on the CP map $S\in{\bf NCP}_\delta^U(s_1)$ as in Proposition \ref{prop:Ckiet}: the dilation space $\hil_0$ of Proposition \ref{propo2} as well as $\mc K$ are required to be separable. Again, we may identify $\hil_0$ with the direct integral $\hil_\oplus$ of Proposition \ref{paapropo}.

\begin{prop}\label{prop:lcsc}
Suppose that the maps $G\ni g\mapsto S_{b\otimes f}\big(v,U(g)v\big)\in\lh$, $b\in\lk$, $f\in L^\infty(\nu)$, $v\in\mb V$, are weakly measurable. Let $\hil_\oplus$, $\hil_\oplus^0$, the representations $g\mapsto y_g$ and $g\mapsto y^0_g$, $K$, $Y$, and $C$ be as in Proposition \ref{prop:Ckiet}. There are separable Hilbert spaces $\mc M$ and $\mc M^0$ such that $\hil_\oplus=L^2(\nu)\otimes\mc M$ and $\hil_\oplus^0=L^2(\nu)\otimes\mc M^0$ and weakly measurable (or equivalently strongly continuous) representations $\pi:H\to\mc U(\mc M)$ and $\rho:H\to\mc U(\mc M^0)$ such that $y_g=y^\pi_g$ and $y^0_g=y^\rho_g$ for all $g\in G$, where
\begin{eqnarray}
(y_g^\pi\psi)(\om)&=&y^\pi(g,\om)\psi(g^{-1}\om)\sqrt{\rho_{g^{-1}}(\om)}\label{imprimitiiviesitys1}\\
(y_g^\rho\f)(\om)&=&y^\rho(g,\om)\f(g^{-1}\om)\sqrt{\rho_{g^{-1}}(\om)}\label{imprimitiiviesitys1b}
\end{eqnarray}
for all $g\in G$, $\psi\in L^2(\nu)\otimes\mc M$, $\f\in L^2(\nu)\otimes\mc M^0$, and $\nu$-a.a.\ $\om\in\Om$. Moreover, there is an isometry $C_0:\mc M^0\to\mc K\otimes\mc M$ such that $C_0\rho(h)=\big(u_h\otimes\pi(h)\big)C_0$ for all $h\in H$ and
\begin{equation}\label{Ckietratk}
C(\ovl g)=\big(u_g\otimes y^\pi(g,\ovl g)\big)C_0y^\rho(g,\ovl g)^*
\end{equation}
for a.a.\ $g\in G$, where the isometries $C(\ovl g)$, $g\in G$, are the ones introduced in Proposition \ref{prop:Ckiet}.
\end{prop}

\begin{proof}
It follows that the representation $G\ni g\mapsto y_g\in\mc U(\hil_\oplus)$ of Proposition \ref{paapropo} is weakly measurable (or, equivalently, strongly continuous) and hence the associated map $G\times\Om\ni(g,\om)\mapsto y(g,\om)\in\mc U(\hil_{n(\om)})$ is (jointly) weakly measurable \cite[Lemma 6.5]{varadarajan}. Indeed, by polarization, the maps $g\mapsto S_{b^*\otimes \overline f}\big(v,U(g)w\big)=[(b\otimes\hat f)Yv]^*(YU(g)w)$ are weakly measurable. Since the linear span of vectors $(b\otimes \hat f)Yv$ is $\sigma$-strongly (and thus weakly) dense in $\mc L(\hil;\mc K\otimes\hil_\oplus)$ there exists a dense subspace $\mc S\subseteq\mc K\otimes\hil_\oplus$ such that the mappings $g\mapsto\langle\f'|(YU(g)w)\psi\rangle$ are measurable for all $\f'\in\mc S$ and $\psi\in\hil$. Thus, for all $\f\in\mc K\otimes\hil_\oplus$ and $\psi\in\hil$, there exists a sequence $\f_n\in\mc K\otimes\hil_\oplus$, $n\in\mb N$, such that $\lim_{n\to\infty}\f_n=\f$ and the maps $g\mapsto\langle\f_n|(YU(g)w)\psi\rangle$ are measurable for all $n\in\mb N$. Now $g\mapsto\langle\f|(YU(g)w)\psi\rangle$ is measurable as a pointwise limit of a sequence of measurable functions and $g\mapsto YU(g)w$ is weakly measurable implying that $g\mapsto (b\otimes\widehat{\gamma_g(f)})(u_g^*\otimes I_{\hil_\oplus})YU(g)w=(I_\kil\otimes y_g)[(b\otimes \hat f)Yw]$ is weakly measurable (since the space $\kil\otimes\hil_\oplus$ between the products is separable). Similarly as above, by using the density argument, one sees that $g\mapsto y_g$ is weakly measurable, and, in the same way, $g\mapsto y^0_g$ and the associated map $G\times\Om\ni(g,\om)\mapsto y^0(g,\om)\in\mc U(\hil_{n^0(\om)})$ are weakly measurable.

The fact that the direct integral spaces can be expressed with the separable Hilbert spaces $\mc M$ and $\mc M^0$ follows from transitivity and Remark \ref{rem:orbitit}. Thus the representations $g\mapsto y_g$ and $g\mapsto y^0_g$ arise from maps $y:\,G\times\Om\to\mc U(\mc M)$ and $y^0:\,G\times\Om\to\mc U(\mc M^0)$ such that, e.g.,\ for $g\mapsto y_g$,
$$
(y_g\psi)(\om)=y(g,\om)\psi(g^{-1}\om)\sqrt{\rho_{g^{-1}}(\om)}
$$
for all $g\in G$, $\psi\in L^2(\nu)\otimes\mc M$, and $\nu$-a.a.\ $\om\in\Om$, and $y(e,\om)=I_{\mc M}$ and $y(gh,\om)=y(g,\om)y(h,g^{-1}\om)$ for all $g,\,h\in G$ and for a.a. $\om\in\Om$, and similarly for $y^0$. According to \cite[Lemma 5.26]{varadarajan}, we may consider both $y$ and $y^0$ as strict cocycles, and thus cohomologous with Wigner rotations. This means that $y$ is (unitarily) equivalent with the Wigner rotation $y^\pi$ with $\pi(h)=y(h,H)$ and $y^0$ is (unitarily) equivalent with $y^\rho$ with $\rho(h)=y^0(h,H)$.

For the last claim, define a weakly Haar measurable isometry-valued map $F$ on $G$ by
$$
F(g):=\big(u_g\otimes y^\pi(g,\ovl g)\big)^*C(\ovl g)y^\rho(g,\ovl g)\in\mc L(\mc M^0;\mc K\otimes\mc M),\qquad g\in G.
$$ 
Since the function $(g,g')\mapsto F(gg')-F(g')$ is jointly (weakly) Haar measurable, the set of those $(g,g')\in G\times G$ such that $F(gg')-F(g')=0$ is a product measurable set \footnote{That is, it belongs to the completion of the product $\sigma$-algebra $\mc B(G)\otimes\mc B(G)$ with respect to the product of a Haar measure with itself.} as a preimage of the measurable set $\{0\}\subset\mc L(\mc M^0;\mc K\otimes\mc M)$ (equipped with the weak operator topology).  Using the Equation (\ref{Ckiet}), one finds that, for all $g\in G$, one has $F(gg')=F(g')$ for a.a.\ $g'\in G$. Applying the Fubini theorem, one finds that $F(gg')=F(g')$ for a.a.\ $(g,g')\in G\times G$, and another application of the Fubini theorem implies that there is $g_0\in G$ such that $F(gg_0)=F(g_0)$ for a.a.\ $g\in G$. Thus $F$ coincides with a fixed operator $C_0$ almost everywhere, i.e.,\ (\ref{Ckietratk}) holds for a.a.\ $g\in G$. It immediately follows that $C_0$ is an isometry. Fix an $h\in H$. Since e.g.\ $ y^\rho(gh,\ovl g)\equiv y^\rho(g,\ovl g)\rho(h)$,
\begin{eqnarray*}
C_0&=&\big(u_g\otimes y^\pi(g,\ovl g)\big)^*C(\ovl g)y^\rho(g,\ovl g)=\big(u_{gh}\otimes y^\pi(gh,\ovl g)\big)^*C(\ovl g)y^\rho(gh,\ovl g) \\
&=&\big(u_{h}\otimes \pi(h)\big)^*\big(u_{g}\otimes y^\pi(g,\ovl g)\big)^*C(\ovl g)y^\rho(g,\ovl g)\rho(h) 
=\big(u_{h}\otimes \pi(h)\big)^*C_0\rho(h) 
\end{eqnarray*}
for a.a.\ $g\in G$.
\end{proof}

To conclude, by combining the above results with \eqref{CK}, we have seen that the structure of $S$ is determined by a separable Hilbert space $\mc M^0$, an $\mc L(\hil)$-linear map $K:\,\mb V\to\mc L\big(\hil;L^2(\nu)\otimes\mc M^0\big)$, a strict $(G,\Om,\mc M^0)$-cocycle $y^\rho$ (or its induced representation $g\mapsto y^\rho_g$) such that $KU(g)\equiv y^\rho_gK$ and a (covariant) quantum channel $\Phi_0:\,\kil\to\mc L(\mc M^0)$ satisfying $y^\rho(h,H)^*\Phi_0(b)y^\rho(h,H)\equiv\Phi_0(u_h^*bu_h)$ (where $y^\rho(h,H)=\rho(h)$ for all $h\in H$):
\begin{eqnarray}\label{fiinollasisalla}
&&\sis{\f}{S_{b\otimes f}(v,w)\psi}\nonumber\\
&=&\int_{\Omega}f(\ovl g)\sis{[(Kv)\f](\ovl g)}{y^\rho(g,\ovl g)\Phi_0(u_g^*bu_g)y^\rho(g,\ovl g)^*[(Kw)\psi](\ovl g)}d\nu(\ovl g)
\end{eqnarray}
 for all $\f,\,\psi\in\hil$, $b\in\mc L(\mc K),$ $f\in L^\infty(\nu),$ and $v,\,w\in\mb V.$ 
The covariant dilation of $\Phi_0$ is of the form $\Phi_0(b)=C_0^*(b\otimes I_{\mc M})C_0$, $b\in\lk$, where the isometry $C_0:\,\mc M^0\to\mc K\otimes\mc M$ is such that $C_0y^\rho(h,H)=\big(u_h\otimes\pi(h)\big)C_0$ for all $h\in H$ (and $\pi$ is a unitary representation of $H$ in a Hilbert space $\mc M$). Note that $h\mapsto \rho(h)=C_0^*\big(u_h\otimes\pi(h)\big)C_0$ is (unitarily equivalent with) a subrepresentation of $h\mapsto u_h\otimes\pi(h)$. Moreover, giving $\Phi_0$ a Kraus decomposition 
\begin{equation}\label{fiinollakraus}
\Phi_0(b)=\sum_{j=1}^r\ms A_j^*b\ms A_j,\qquad r\in\mb N\cup\{\infty\},
\end{equation}
($\sigma$-strongly) where $\ms A_j\in\mc L(\mc M^0;\kil)$, one can write the covariance condition for $\Phi_0$ in the form 
\begin{equation}\label{covarcondi}
\sum_{j=1}^r\ms A_j^*u_h^*bu_h\ms A_j=\sum_{j=1}^r\rho(h)^*\ms A_j^*b\ms A_j\rho(h)
\end{equation}
for all $h\in H$ and $b\in\mc L(\mc K)$. In the next section, we study the case $\hil=\mb C$ and elaborate equation \eqref{fiinollasisalla} for a quantum instrument by adding assumptions on the groups $G$ and $H$; see Theorem \ref{covinstr}.

\section{Covariant instruments in the case of a type-I group}\label{covariantobs}

In this section, we fix a unimodular lcsc group $G$ of type I \cite{Dix1,folland} and a strongly continuous unitary representation $U:\,G\to\mc U(\mc V)$ where $\mc V$ is a separable Hilbert space. Denote the unitary dual space of $G$ by $\hat G$, i.e.,\ $\tau\in\hat G$ is a unitary equivalence class of irreducible unitary representations of $G$. We identify each $\tau\in\hat G$ with a fixed representative, i.e.,\ a single irreducible unitary representation $g\mapsto\tau_g$ that operates on a separable Hilbert space $\mc K(\tau)$ whose dimension is $n(\tau)\in\mb N\cup\{\infty\}$. It follows that there are separable Hilbert spaces $\mc L(\tau)$, $\tau\in\hat G$, such that we may set
\begin{equation}\label{Usuora}
\mc V=\int_{\hat G}^\oplus\mc K(\tau)\otimes\mc L(\tau)\,d\mu(\tau)
\end{equation}
with some measure $\mu:\mc B(\hat G)\to[0,\infty]$ and
\begin{equation}\label{Uop}
\big(U(g)\f)(\tau)=\big(\tau_g\otimes I_{\mc L(\tau)}\big)\f(\tau)
\end{equation}
for all $\f\in\hil$ and $\mu$-a.a.\ $\tau\in\hat G$ \cite{Dix1}. We assume that $H$ is a {\it compact} subgroup of $G$ and denote the space $G/H$ of left cosets by $\Om$ and the canonical projection by $q:\,g\mapsto \ovl g=gH$.

We say that a normalized positive operator measure (observable) $\ms M:\,\mc B(\Omega)\to\mc L(\mc V)$ is $(U,\Om)$-covariant if $\ms M(gX)= U(g)\ms M(X)U(g)^*$ for all $g\in G$ and $X\in\mc B(\Omega)$, and denote the convex set of such measures by $\mc O^U(\Om)$. In this section, we give an exhaustive characterization for the $(U,\Om)$-covariant observables largely following \cite{holevotypeI, carmeli, holevopell}. In \cite{carmeli} the characterization was given in the case where $G$ is compact whereas in \cite{holevopell} a similar result was obtained in the case of the trivial subgroup $H=\{e\}$. We show that combining these results one can characterize the set $\mc O^U(\Om)$ in the case where the group $G$ needs not be compact but the (non-trivial) subgroup $H$ is compact. In the end of this section we apply the results concerning covariant observables to covariant quantum instruments. 

Before we go on to characterize covariant observables, let us recall some basics of harmonic analysis of topological groups. Fix the Haar measures\footnote{Both left and right since the groups are unimodular.} $\nu_G$ and $\nu_H$, $\nu_H(H)=1$, of the groups $G$ and $H$, respectively. Pick an invariant measure $\nu_{\Om}:\,\mc B(\Om)\to[0,\infty]$ such that
\begin{equation}\label{Hmitta}
\int_Gf(g)\,d\nu_G(g)=\int_\Om\int_Hf(gh)\,d\nu_H(h)\,d\nu_\Om(\ovl g)
\end{equation}
for all $f\in L^1(\nu_G)$. Especially, $\nu_\Om(X)=\nu_G\big(q^{-1}(X)\big)$, $X\in\mc B(\Om)$. Any two quasi-invariant measures on $\Om$ are mutually absolutely continuous \cite[Theorem 5.19]{varadarajan}. Especially, any $\ms M\in\mc O^U(\Om)$ is absolutely continuous with respect to $\nu_\Om$. 

Suppose that $\mc M^0$ is a separable Hilbert space and $\rho:H\to\mc U(\mc M^0)$ is a strongly continuous unitary representation. The representation $g\mapsto y^\rho_g$ induced from $\rho$ defined as in (\ref{imprimitiiviesitys1}) can be expressed in an equivalent way\footnote{In general, the definition of the representation (\ref{imprimitiiviesitys1}) includes multiplication with the square root of the density function $\rho_g(\om)$ but our simplifying assumption on unimodularity allows us to use the $G$-invariant measure $\nu_\Om$ implying $\rho_g(\om)=1$ for $\nu_\Om$-a.a.\ $\om\in \Om$.} (for details, see e.g.\ Section 6.1 of \cite{folland}): Define the linear space $\mf K^0$ of continuous functions $f:\,G\to\mc M^0$ such that the (projected) support of $f$ is compact and $f(gh)=\rho(h)^*f(g)$ for all $g\in G$ and $h\in H$. Let $\mf K^\rho$ denote the completion of $\mf K^0$ with respect to the inner product $\sis{f}{f'}=\int_\Om\sis{f(g)}{f'(g)}\,d\nu_\Om(\ovl g)$, $f,\,f'\in\mf K^0$. Define a strongly continuous representation $G\ni g\mapsto\tj^\rho_g\in\mc U(\mf K^\rho)$ and  a unitary operator $U':\,\mf K^\rho\to L^2(\nu_\Om;\mc M^0)\cong L^2(\nu_\Om)\otimes\mc M^0$ through equations $(\tj^\rho_gf)(g'):=f(g^{-1}g')$ and 
\begin{equation}\label{Un kaava}
(U'f)(\ovl g):=y^\rho(g,\ovl g)f(g)
\end{equation} 
for all $f\in\mf K^0$ and $g,\,g'\in G$. It follows that the representations $g\mapsto\tj^\rho_g$ and $g\mapsto y^\rho_g$ are unitarily equivalent: $y^\rho_gU'\equiv U'\tj^\rho_g$.

Since $H$ is compact we may view $\mf K^\rho$ as a closed subspace of  $L^2(\nu_G;\mc M^0)$. Indeed, when $f\in\mf K^\rho$, we may calculate its  $L^2(\nu_G;\mc M^0)$-norm
\begin{eqnarray*}
\|f\|^2=\int_\Om\int_H\|f(gh)\|^2\,d\nu_H(h)\,d\nu_\Om(\ovl g)&=&\int_\Om\int_H\|f(g)\|^2\,d\nu_H(h)\,d\nu_\Om(\ovl g)\\
&=&\int_\Om\|f(g)\|^2\,d\nu_\Om(\ovl g)
\end{eqnarray*}
which is the same as its $\mf K^\rho$-norm. The representation $G\ni g\mapsto\tj^\rho_g\in\mc U(\mf K^\rho)$ is nothing but the projection of the left regular representation $G\ni g\mapsto\lambda^{\mc M^0}_g\in\mc U\big(L^2(\nu_G;\mc M^0)\big)$ defined by $(\lambda^{\mc M^0}_g\psi)(g'):=\psi(g^{-1}g')$ for all $g\in G$, $\psi\in L^2(\nu_G;\mc M^0)$, and a.a.\ $g'\in G$:
$$
{\mc I}\tj^\rho_g=\lambda^{\mc M^0}_g{\mc I}, \qquad g\in G
$$
where ${\mc I}:\,\mf K^\rho\to L^2(\nu_G;\mc M^0)$ is the isometry $f\mapsto f$.

Define the spectral measures $\ms Q:\,\mc B(G)\to\mc L\big(L^2(\nu_G;\mc M^0)\big)$,
$\big(\ms Q(Z)\psi)(g):=\chi_Z(g)\psi(g)$ for all $Z\in\mc B(G)$, $\psi\in L^2(\nu_G;\mc M^0)$, and a.a.\ $g\in G$, and $\ms M^\rho:\,\mc B(\Om)\to\mc L(\mf K^\rho)$, $\big(\ms M^\rho(X)f\big)(g):=\chi_X(\ovl g)f(g)$ for all $f\in\mf K^\rho$, $X\in\mc B(\Om)$, and a.a.\ $g\in G$. Hence, 
$$
{\ms M}^\rho(X)={\mc I}^*{\ms Q}\big(q^{-1}(X)\big){\mc I},\qquad X\in\mc B(\Omega).
$$
From now on, we call the triple $(\mf K^\rho,\ms M^\rho,\tj^\rho)$ as the {\it canonical system of imprimitivity} associated with $\rho$. Note that $\hat\chi_X^0U'\equiv U'{\ms M}^\rho(X)$.

\begin{rem}\label{imprimitivity}\rm
Assume that $\mc O^U(\Om)\ne\emptyset$ and let $\ms M\in\mc O^U(\Om)$. According to Proposition \ref{prop:lcsc}, there are a separable Hilbert space $\mc M^0$, a strongly continuous representation $\rho:\,H\to\mc U(\mc M^0)$ and an isometry $K:\,\mc V\to\mf K^\rho\subseteq L^2(\nu_G;\mc M^0)$ such that $K^*{\ms M}^\rho(X)K\equiv{\ms M}(X)$, $KU(g)\equiv \tj^\rho_gK$, and the linear span of the vectors ${\ms M}^\rho(X)K\f$, $X\in\Sigma$, $\f\in\mc V$, is dense in $\mf K^\rho$. Note that the conditions of Proposition \ref{prop:lcsc} are met because $\mc V$ is separable and $\Om$ is lcsc guaranteeing that the minimal Na\u{\i}mark dilation space of any $\ms M\in\mc O^U(\Om)$ is separable. Thus, the quadruple $(\mf K^\rho,\ms M^\rho,K,\tj^\rho)$ constitutes a minimal covariant Na\u{\i}mark dilation for $\ms M$ where $(\mf K^\rho,\ms M^\rho,\tj^\rho)$ is the canonical system of imprimitivity associated with $\rho$. This result has first been proven in \cite{cattaneo}. Note that ${\ms M}(X)\equiv\ovl K^*{\ms Q}\big(q^{-1}(X)\big)\ovl K$ and $\ovl K U(g)=\lambda^{\mc M^0}_g\ovl K$ where $\ovl K:={\mc I}K$ is an isometry from ${\mc V}$ to  $L^2(\nu_G;\mc M^0)$. Especially, one can define a normalized positive operator measure $\ovl{\sf M}:\,{\mc B}(G)\to\mc L(\mc V)$ by $\ovl{\sf M}(Z):=\ovl K^*{\ms Q}(Z)\ovl K$, $Z\in{\mc B}(G)$. Now ${\ms M}(X)\equiv\ovl{\ms M}\big(q^{-1}(X)\big)$ and $\ovl{\sf M}$ is covariant: $\ovl{\ms M}(gZ)= U(g)\ovl{\ms M}(Z)U(g)^*$ for all $g\in G$ and $Z\in\mc B(G)$. Hence, we can apply the results of \cite{holevotypeI} to $\ovl{\ms M}$. Especially, we find that the measure $\mu$ of \eqref{Usuora} must be absolutely continuous with respect to the Plancherel measure $\mu_{\hat G}:\,\mc B(\hat G)\to[0,\infty]$ associated with the Haar measure $\nu_G$ by the Parseval--Plancherel formula. We clearly may and, from now on, will assume that $\mu=\mu_{\hat G}$.
\end{rem}

Pick $\mu_{\hat G}$-measurable fields $\hat G\ni\tau\mapsto\zeta(\tau)\in\mc K(\tau)$ and $\hat G\ni\tau\mapsto\xi(\tau)\in\mc L(\tau)$ and define the field $\zeta\star\xi$ via $(\zeta\star\xi)(\tau)=\zeta(\tau)\otimes\xi(\tau)$. Denote by $\mb V$ the linear hull of vectors $\zeta\star\xi\in\mc V$ for which the norm $\|\zeta\star\xi\|_1:=\int_{\hat G}\|\zeta(\tau)\|\|\xi(\tau)\|\,d\mu_{\hat G}(\tau)<\infty$. Note also that the Hilbert space norm %$\|\zeta\star\xi\|_{\mc V}=\left[\int_{\hat G}\|\zeta(\tau)\|^2\|\xi(\tau)\|^2\,d\mu_{\hat G}(\tau)\right]^{1/2}$ is finite.
$\|\zeta\star\xi\|_{\mc V}<\infty$.\footnote{In addition, we consider $\tau\mapsto\zeta(\tau)\otimes\xi(\tau)$ as a fixed representative from the class $\zeta\star\xi\in\mc V$.} Clearly, $\mb V$ is dense in $\mc V$ and invariant under $U$ (i.e.,\ $U(g)\mb V\subseteq\mb V$ for all $g\in G$). Let $\hat G\ni\tau\mapsto \{e_j(\tau)\}_{j=1}^{n(\tau)}\subset\mc K(\tau)$ be a measurable field of orthonormal bases and $\{e_j(\tau)^*\}_{j=1}^{n(\tau)}\subset\mc K(\tau)^*$ the dual basis of $ \{e_j(\tau)\}_{j=1}^{n(\tau)}$ where $\mc K(\tau)^*$ is the topological dual space of $\mc K(\tau)$.

 The following theorem characterizes the convex set $\mc O^U(\Om)$ and its extreme points.

\begin{theor}\label{extcovobsker}
For each $\ms M\in\mc O^U(\Om)$ there are a separable Hilbert space ${\mc M}^0$, a linear map $\Lambda:\mb V\to{\mc M}^0$ whose range $\Lambda(\mb V)$ is total in ${\mc M}^0$, and a strongly continuous unitary representation $\rho:\,H\to\mc U({\mc M}^0)$ such that $\Lambda U(h)=\rho(h)\Lambda$ for all $h\in H$ and
\begin{equation}\label{kolmogorov}
\sis{\f}{\ms M(X)\psi}=\int_X\sis{\Lambda U(g)^*\f}{\Lambda U(g)^*\psi}\,d\nu_\Om(\ovl g)
\end{equation}
for all $\f,\,\psi\in\mb V$ and $X\in\mc B(\Om)$. The map $\Lambda$ has the following structure: there are weakly $\mu_{\hat G}$-measurable fields $\tau\mapsto\Lambda_j(\tau)$, $j\in\mb N$, of bounded operators $\Lambda_j(\tau):\,\mc L(\tau)\to {\mc M}^0$ such that $\sum_{j=1}^{n(\tau)}\Lambda_j(\tau)^*\Lambda_j(\tau)=I_{\mc L(\tau)}$ ($\sigma$-strongly) and \begin{equation}\label{Phiint}
\Lambda\f=\int_{\hat G}\Lambda(\tau)\f(\tau)\,d\mu_{\hat G}(\tau),\qquad\f\in\mb V,
\end{equation}
where $\Lambda(\tau):=\sum_{j=1}^{n(\tau)} e_j(\tau)^*\otimes\Lambda_j(\tau)$ for all $\tau\in\hat G$. Conversely, a map $\ms M:\,\mc B(\Om)\to\mc L(\hil)$ defined as above in (\ref{kolmogorov}) belongs to $\mc O^U(\Om)$. Moreover, $\ms M$ is extreme in $\mc O^U(\Om)$ if and only if, for any $D\in\mc L({\mc M}^0)$, the conditions $D\rho(h)=\rho(h)D$ for all $h\in H$ and
\begin{equation}\label{extcovobskerchar}
\sum_{j=1}^{n(\tau)}\Lambda_j(\tau)^*D\Lambda_j(\tau)=0
\end{equation}
for $\mu_{\hat G}$-a.a.\ $\tau\in\hat G$ imply $D=0$.
\end{theor}

\begin{proof}
This proof largely follows the proofs of \cite[Proposition 2, Theorem 1]{holevopell} and \cite[Theorem 2]{holevotypeI}. Suppose that $\mc O^U(\Om)$ is non-empty and recall that $\mc V=\int_{\hat G}^\oplus\mc K(\tau)\otimes\mc L(\tau)\,d\mu_{\hat G}(\tau)$. Let $\ms M\in\mc O^U(\Om)$, and let $(\mf K^\rho,\ms M^\rho,K,\tj^\rho)$ be a minimal covariant Na\u{\i}mark dilation for $\ms M$ and $\ovl K={\mc I}K$ as in Remark \ref{imprimitivity}. Especially, $\rho:\,H\to\mc U(\mc M^0)$ is a strongly continuous representation where $\mc M^0$ is a Hilbert space. Next we recall the construction of a minimal covariant dilation for $\ovl{\sf M}$, see \cite[Section 4.2]{holevopell}.

For each $\tau\in\hat G$ define the conjugate linear bijection $\mc K(\tau)\ni\beta\mapsto\beta^*\in\mc K(\tau)^*$ through $\beta^*(\zeta):=\sis{\beta}{\zeta}$ whenever $\beta,\,\zeta\in\mc K(\tau)$. The space $\mc K(\tau)\otimes\mc K(\tau)^*$ is identified with the Hilbert space of Hilbert-Schmidt operators on $\mc K(\tau)$ so that $\tau_g$ acts on the first component of $|{\zeta}\rangle\langle\beta|:=\zeta\otimes\beta^*\in\mc K(\tau)\otimes\mc K(\tau)^*$. Let $\mc K_\oplus:=\int_{\hat G}^\oplus\mc K(\tau)\otimes\mc K(\tau)^*\,d\mu_{\hat G}(\tau)$. For any $\f\in L^1(\nu_G)\cap L^2(\nu_G)$ define $\mc F\f\in \mc K_\oplus$ by $(\mc F\f)(\tau)=\int_G\f(g)\tau_g\,d\nu_G(g)$ for $\mu_{\hat G}$-a.a.\ $\tau\in\hat G$. The map $\f\mapsto\mc F\f$ can be extended to a unitary map $\mc F:\,L^2(\nu_G)\to\mc K_\oplus$, the {Fourier-Plancherel operator}. Define the representation $G\ni g\mapsto\hat\lambda_g^{\mc M^0}:=(\mc F\otimes I_{\mc M^0})\lambda_g^{\mc M^0}(\mc F\otimes I_{\mc M^0})^*\in\mc U(\mc K_\oplus\otimes\mc M^0)$ for which $(\hat\lambda_g^{\mc M^0}\psi)(\tau)=(\tau_g\otimes I_{\mc K(\tau)^*}\otimes I_{\mc M^0})\psi(\tau)$ for all $g\in G$, $\psi\in\mc K_\oplus\otimes\mc M^0$ and $\mu_{\hat G}$-a.a.\ $\tau\in\hat G$. Since $\ovl K$ intertwines $U$ with $\lambda^{\mc M^0}$ and the unitary operator $\mc F\otimes I_{\mc M^0}$ intertwines $\lambda^{\mc M^0}$ with $\hat\lambda^{\mc M^0}$ it follows that the isometry $W=(\mc F\otimes I_{\mc M^0})\ovl K$ intertwines $U$ with $\hat\lambda^{\mc M^0}$. Let $\mc M$ be an infinite dimensional Hilbert space such that $\mc M^0$ is a closed subspace of $\mc M$. Then $\mc M=\mc M^0\oplus({\mc M^0})^\perp$ and we may consider $W$ as an isometry from $\mc V$ to $\kil_\oplus\otimes\mc M$. As shown in detail in \cite[Lemma 1]{holevopell}, $W$ is decomposable in the sense that $(W\f)(\tau)=\big(I_{\mc K(\tau)}\otimes W(\tau)\big)\f(\tau)$ for all $\f\in\mc V$ and $\mu_{\hat G}$-a.a.\ $\tau\in\hat G$ with some weakly $\mu_{\hat G}$-measurable field $\tau\mapsto W(\tau)$ of isometries $W(\tau):\,\mc L(\tau)\to\mc K(\tau)^*\otimes\mc M$. But $\sis{\zeta\otimes\alpha}{W\f}=0$ for all $\zeta\in\kil_\oplus$, $\alpha\in(\mc M^0)^\perp$, and $\f\in\mc V$ implying that we may assume that the range of any $W(\tau)$ is contained in $\mc K(\tau)^*\otimes\mc M^0$ (see the proof of Lemma 2 in \cite{holevopell}).

Any vector $\eta\in\mc K(\tau)^*\otimes\mc M^0$ can be expressed as a unique sequence of vectors $\{\eta_j\}_{j=1}^{n(\tau)}\subset\mc M^0$ such that $\eta=\sum_{j=1}^{n(\tau)} e_j(\tau)^*\otimes \eta_j$. Let $\Psi_j(\tau):\,\mc K(\tau)^*\otimes\mc M^0\to\mc M^0$ be the bounded operator $\eta\mapsto \eta_j$, i.e., $\eta=\sum_{j=1}^{n(\tau)} e_j(\tau)^*\otimes\big(\Psi_j(\tau)\eta\big)$ for any $\eta\in\mc K(\tau)^*\otimes\mc M^0$. Let $\Lambda_j(\tau):=\Psi_j(\tau)W(\tau)$ and $\mb V(\tau):={\rm lin}\{\zeta\otimes\xi\,|\,\zeta\in\mc K(\tau),\; \xi\in\mc L(\tau)\}$. Define a linear map $\Lambda(\tau):\,\mb V(\tau)\to\mc M^0$ by $\Lambda(\tau)(\zeta\otimes\xi):=\sum_{j=1}^{n(\tau)}\sis{e_j(\tau)}{\zeta}\Lambda_j(\tau)\xi$ so that $\Lambda(\tau)(e_j(\tau)\otimes\xi)=\Lambda_j(\tau)\xi$ and
 $\sum_{j=1}^{n(\tau)}\Lambda_j(\tau)^*\Lambda_j(\tau)=I_{\mc L(\tau)}$ (weakly and thus $\sigma$-strongly since the sequence is increasing and bounded). Fix $\zeta\star\xi\in\mb V$. Since $\sum_{j=1}^{n(\tau)}|\sis{e_j(\tau)}{\zeta(\tau)}|^2=\|\zeta(\tau)\|^2$ and $\sum_{j=1}^{n(\tau)}\|\Psi_j(\tau)W(\tau)\xi(\tau)\|^2=\|\xi(\tau)\|^2$ we may define the integral operator $\Lambda:\mb V\to\mc M^0$ by
\begin{eqnarray*}
\Lambda(\zeta\star\xi)&=&\int_{\hat G}\Lambda(\tau)(\zeta\star\xi)(\tau)\,d\mu_{\hat G}(\tau)\\
&=&\int_{\hat G}\sum_{j=1}^{n(\tau)}\sis{e_j(\tau)}{\zeta(\tau)}\underbrace{[\Psi_j(\tau)W(\tau)\xi(\tau)]}_{\in\mc M^0}\,d\mu_{\hat G}(\tau).
\end{eqnarray*}
Denote the trace over $\mc K(\tau)$ by $\mr{tr}_\tau$. Let $\alpha\in\mc M^0$ and $T(\tau):=\zeta(\tau)\otimes\beta(\tau)^*=|\zeta(\tau)\ra\la\beta(\tau)|$ where $\beta(\tau):=\sum_{j=1}^{n(\tau)}\ovl{\sis{\alpha}{\Psi_j(\tau)W(\tau)\xi(\tau)}}e_j(\tau)$. Since $\mr{tr}_\tau[T(\tau)^*T(\tau)]=\|\zeta(\tau)\|^2\|\beta(\tau)\|^2$ and $\|\beta(\tau)\|\le\|\alpha\|\|\xi(\tau)\|$ it follows that $T=\int_{\hat G}^\oplus T(\tau)d\mu_{\hat G}(\tau)\in\mc K_\oplus$ and $\mc F^*T\in L^2(\nu_G)$. For any $\f\in L^1(\nu_G)\cap L^2(\nu_G)$, the jointly measurable function $(g,\tau)\mapsto \f(g)\mr{tr}_\tau[\tau_g^*T(\tau)]$ is $\nu_G\times\mu_{\hat G}$--integrable:
\begin{eqnarray*}
&&\int_G\int_{\hat G}\big|\f(g)\mr{tr}_\tau[\tau_g^*T(\tau)]\big|\,d\mu_{\hat G}(\tau)\,d\nu_G(g)\\
&\le&\|\alpha\|\int_G|\f(g)|\,d\nu_G(g)\int_{\hat G}\|\zeta(\tau)\|\|\xi(\tau)\|d\mu_{\hat G}(\tau)<\infty.
\end{eqnarray*}
By using the Fubini theorem, the formula $(\mc F\f)(\tau)=\int_G\f(g)\tau_g\,d\nu_G(g)$ (defined as a weak integral), and the unitarity of $\mc F$, one gets
\begin{eqnarray*}
\int_G\int_{\hat G}\ovl{\f(g)}\mr{tr}_\tau[\tau_g^*T(\tau)]\,d\mu_{\hat G}(\tau)\,d\nu_G(g)&=&\int_{\hat G}\mr{tr}_\tau[(\mc F\f) (\tau)^*T(\tau)]\,d\mu_{\hat G}(\tau)\\
&=&\int_G\ovl{\f(g)}(\mc F^*T)(g)\,d\nu_G(g)
\end{eqnarray*}
so that $\int_{\hat G}\mr{tr}_\tau[\tau_g^*T(\tau)]\,d\mu_{\hat G}(\tau)=(\mc F^*T)(g)$ for all $g\in G\setminus N_\alpha$ where $N_\alpha\subset G$ is of zero Haar measure. Then, for all $g\in G\setminus N_\alpha$,
\begin{eqnarray*}
\sis{\alpha}{\Lambda U(g)^*(\zeta\star\xi)}
&=& \int_{\hat G}\sum_{j=1}^{n(\tau)}\sis{e_j(\tau)}{\tau_g^*\zeta(\tau)}\sis{\alpha}{\Psi_j(\tau)W(\tau)\xi(\tau)}\,d\mu_{\hat G}(\tau)\nonumber\\
&=&\int_{\hat G}\sis{\beta(\tau)}{\tau_g^*\zeta(\tau)}\,d\mu_{\hat G}(\tau)=\int_{\hat G}\mr{tr}_\tau\big[T(\tau)\tau_g^*\big]\,d\mu_{\hat G}(\tau)\\
&=&(\mc F^*T)(g)=\sis{\alpha}{[(\mc F^*\otimes I_{\mc M^0})W(\zeta\star\xi)](g)}\\
&=&\sis{\alpha}{[\ovl K(\zeta\star\xi)](g)}
\end{eqnarray*}
where we have used the identification $L^2(\nu_\Om)\otimes\mc M^0 \cong L^2(\nu_\Om;\mc M^0)$. Note that $G\ni g\mapsto\sis{\alpha}{\Lambda U(g)^*(\zeta\star\xi)}=\int_{\hat G}\sis{\beta(\tau)}{\tau_g^*\zeta(\tau)}\,d\mu_{\hat G}(\tau)\in\mb C$ is continuous.\footnote{Since $g\mapsto \int^\oplus_{\hat G}\tau_g\,d\mu_{\hat G}(\tau)$ is a weakly continuous representation of $G$ in $\int^\oplus_{\hat G}\kil(\tau)\,d\mu_{\hat G}(\tau)$ and there exist vectors $\beta',\zeta'\in\int^\oplus_{\hat G}\kil(\tau)\,d\mu_{\hat G}(\tau)$ such that $\sis{\beta'(\tau)}{\tau_g^*\zeta'(\tau)}=\sis{\beta(\tau)}{\tau_g^*\zeta(\tau)}$ for all $g\in G$ and $\tau\in\hat G$.} Let $\{\alpha_n\}_{n\in\mb N}\subset\mc M^0$ be a countable dense set in $\mc M^0$ and $N:=\cup_{n=1}^\infty N_{\alpha_n}$ a zero measurable set. By continuity of the inner product, we get $[\ovl K(\zeta\star\xi)](g)=\Lambda U(g)^*(\zeta\star\xi)$ for all $g\in G\setminus N$ and 
$$
\sis{\f}{\ovl{\sf M}(Z)\psi}=\sis{\f}{\ovl K^*{\ms Q}(Z)\ovl K\psi}=\int_Z\sis{\Lambda U(g)^*\f}{\Lambda U(g)^*\psi}\,d\nu_G(g)
$$
for all $\f,\,\psi\in\mb V$ and $Z\in\mc B(G)$, see equations (1) and (9) and Proposition 2 of \cite{holevopell}. Since, for all $\f\in\mb V$, $g\mapsto\Lambda U(g)^*\f$ is weakly continuous, $\Lambda U(g)^*\f=[\ovl K\f](g)$ for a.a.\ $g\in G$, and $\ovl K\f\in\mf K^\rho$ (i.e.\  $(\ovl K\f)(gh)=\rho(h)^*(\ovl K\f)(g)$ for a.a.\ $g\in G$ and all $h\in H$) it follows that $\Lambda U(h)=\rho(h)\Lambda$ for all $h\in H$. From $\ms M(X)\equiv\ovl{\ms M}\big(q^{-1}(X)\big)$ one arrives at (\ref{kolmogorov}):
\begin{eqnarray*}
\sis{\f}{{\ms M}(X)\psi}&=&\int_X\int_H\sis{\Lambda U(gh)^*\f}{\Lambda U(gh)^*\psi}\,d\nu_H(h)\,d\nu_\Om(\ovl g) \\
&=&\int_X\int_H\,d\nu_H(h)\sis{\Lambda U(g)^*\f}{\Lambda U(g)^*\psi}\,d\nu_\Om(\ovl g)\\
&=&\int_X\sis{\Lambda U(g)^*\f}{\Lambda U(g)^*\psi}\,d\nu_\Om(\ovl g).
\end{eqnarray*}
The converse assertion of the first part of the claim is easily proven.

To see that $\Lambda(\mb V)$ is total in $\mc M^0$, assume that, on the contrary, there is an $\alpha\in\mc M^0$, $\alpha\ne 0$, such that $\sis{\alpha}{\Lambda\f}=0$ for all $\f\in\mb V$. Fix a measurable section $s:\,\Omega\to G$ for $q$ and an $X'\in\mc B(\Om)$ for which $0<\nu_\Om(X')=\nu_G(q^{-1}(X'))<\infty$. Define a nonzero $\psi\in\mf K^\rho$ by $\psi(g):=\chi_{q^{-1}(X')}(g)\rho(g^{-1}s\big(\ovl g)\big)\alpha$ for all $g\in G$. Now $\sis{\psi}{{\ms M}^\rho(X)K\f}=0$ for all $X\in\Sigma$ and $\f\in\mb V$ but this is not possible since $\mb V$ is dense in $\mc V$ and the linear span of vectors ${\ms M}^\rho(X)K\f$, $X\in\Sigma$, $\f\in\mc V$, is total in $\mf K^\rho$. Next we determine whether $\ms M$ is extreme in $\mc O^U(\Om)$ using the dilation $(\mf K^\rho,\ms M^\rho,\tj^\rho,K)$ (see, e.g.\ Theorem \ref{Scovarext}). 

Suppose that $\tilde D\in\mc L(\mf K^\rho)$ is such that $\tilde D\ms M^\rho(X)=\ms M^\rho(X)\tilde D$ for all $X\in\mc B(\Om)$ and $\tilde D\tj^\rho_g=\tj^\rho_g\tilde D$ for all $g\in G$. These conditions are equivalent with the following: $(\tilde Df)(g)=Df(g)$ for all $f\in\mf K^\rho$ and a.a.\ $g\in G$ with some operator $D\in\mc L(\mc M^0)$ such that $D\rho(h)=\rho(h)D$ for all $h\in H$ \cite[Theorem 1]{Dix2}. Pick $\f=\zeta_1\star\xi_1\in\mb V$ and $\psi=\zeta_2\star\xi_2\in\mb V$. Using
$$
W(\tau)\xi_r(\tau)=\sum_{j=1}^{n(\tau)} e_j(\tau)^*\otimes\big(\Psi_j(\tau)W(\tau)\xi_r(\tau)\big)=\sum_{j=1}^{n(\tau)} e_j(\tau)^*\otimes\big(\Lambda_j(\tau)\xi_r(\tau)\big)
$$
for $r=1,\,2$, one can calculate
\begin{eqnarray*}
&&\sis{K\f}{\tilde DK\psi}=\sis{\ovl K\f}{(I_{L^2(\nu_G)}\otimes D)\ovl K\psi}\\
&=&\sis{(\mc F^*\otimes I_{\mc M^0})W\f}{(I_{L^2(\nu_G)}\otimes D)(\mc F^*\otimes I_{\mc M^0})W\psi}=\sis{W\f}{(I_{\mc K_\oplus}\otimes D)W\psi}\\
&=&\int_{\hat G}\Sis{\zeta_1(\tau)\otimes W(\tau)\xi_1(\tau)}{\zeta_2(\tau)\otimes\big(I_{\mc K(\tau)^*}\otimes D\big)W(\tau)\xi_2(\tau)}\,d\mu_{\hat G}(\tau)\\
&=&\int_{\hat G}\sis{\zeta_1(\tau)}{\zeta_2(\tau)}\sum_{j=1}^{n(\tau)}\sis{\Lambda_j(\tau)\xi_1(\tau)}{D\Lambda_j(\tau)\xi_2(\tau)}\,d\mu_{\hat G}(\tau).
\end{eqnarray*}
It follows that the condition $K^*\tilde DK=0$ equals with (\ref{extcovobskerchar}). These results combined with Theorem \ref{Scovarext} prove the latter part of the claim.
\end{proof}

In the remainder of this section, we fix a separable Hilbert space $\mc K$ and a strongly continuous representation $G\ni g\mapsto u_g\in\mc U(\mc K)$. We let $\mc I_u^U(\Om)$ denote the convex set of covariant instruments
%, i.e.,\ $\Gamma\in\mc I_u^U(\Om)$ is a quantum instrument $\Gamma:\,\mc L(\mc K)\times\mc B(\Om)\to\mc L(\mc V)$ (see Example \ref{kvanttiinstr}) such that $\Gamma(u_gbu_g^*,gX)= U(g)\Gamma(b,X)U(g)^*$ for all $g\in G$, $b\in\mc L(\mc K)$, and $X\in\mc B(\Om)$
as described in Remark \ref{rem:kvanttiCP}. Especially, the associated observable $\sf M$ of $\Gamma$ belongs to $\mc O^U(\Om)$ so that we may apply Theorem \ref{extcovobsker} and equations \eqref{fiinollasisalla} and \eqref{fiinollakraus}. Note that parts of the following remark have been proven in \cite{carmeli2}.

\begin{rem}\label{instrimprimitivity}{\rm
According to Proposition \ref{prop:lcsc} (whose conditions are again satisfied), for any $\Gamma\in\mc I_u^U(\Om)$ there exist a Hilbert space $\mc M$, a strongly continuous representation $\pi:H\to\mc L(\mc M)$ giving rise to the canonical system of imprimitivity $(\mf K^\pi,\ms M^\pi,\tj^\pi)$ similarly as in the beginning of this section, and an isometry $Y:\,\mc V\to\mc K\otimes\mf K^\pi$ such that $YU(g)=\big(u_g\otimes\tj^\pi_g\big)Y$ for all $g\in G$ and
\begin{equation}\label{instrimprimitivity1}
\Gamma(b,X)=Y^*\big(b\otimes\ms M^\pi(X)\big)Y,\qquad X\in\mc B(\Om),\quad b\in\mc L(\mc K).
\end{equation}
Furthermore, the dilation $(\mf K^\pi,\ms M^\pi,Y)$ for $\Gamma$ is minimal. Note that $\mf K^\pi$ can be viewed as a closed subspace of $L^2(\nu_G;\mc M)$ such that it consists of functions $\psi\in L^2(\nu_G;\mc M)$ satisfying the condition $\psi(gh)=\pi(h)^*\psi(g)$ for a.a.\ $g\in G$ and all $h\in H$. The instrument $\Gamma$ of (\ref{instrimprimitivity1}) is easily seen to be extreme in $\mc I_u^U(\Om)$ if and only if, for an operator $D\in\mc L(\mc M)$ such that $D\pi(h)=\pi(h)D$ for all $h\in H$, the condition $Y^*(I_{\mc K}\otimes \tilde D)Y=0$, where $\tilde D\in\mc L(\mf K_\pi)$ is defined through $(\tilde Df)(g)=Df(g)$ for all $f\in\mf K^\pi$ and a.a.\ $g\in G$, implies $D=0$.
}
\end{rem}

Using Theorem \ref{extcovobsker}, we may give the following characterization to all instruments in the set $\mc I_u^U(\Om)$ conjectured in \cite[p.\ 1377]{holevo98} (where $H$ was assumed to be trivial $\{e\}$ and $U=u$). See also the result \cite[Theorem 5.2, Chapter 4]{davies} in the case of a compact symmetry group. Let the dense subspace $\mb V\subset\mc V$ be defined as earlier in this section.

\begin{theor}\label{covinstr}
Suppose that $\Gamma\in\mc I_u^U(\Om)$. There are linear operators $\ms B_j:\mb V\to\mc K$, $j<r+1$, $r\in\mb N\cup\{\infty\}$, such that, for all $\f\in\mb V$, $b\in\mc L(\mc K)$, and $h\in H$,
\begin{eqnarray}\label{Hinv}
&&\sum_{j=1}^r\sis{\ms B_j\f}{u_h^*bu_h\ms B_j\f}=\sum_{j=1}^r\sis{\ms B_jU(h)\f}{b\ms B_jU(h)\f},
\\
\label{ehto}
&&\int_\Om\sum_{j=1}^r\sis{\ms B_jU(g)^*\f}{\ms B_jU(g)^*\f}\,d\nu_\Om(\ovl g)=\|\f\|^2,
\end{eqnarray}
and $\Gamma$ has the decomposition
\begin{equation}\label{Icovchar}
\sis{\f}{\Gamma(b,X)\psi}=\int_X\sum_{j=1}^r\sis{u_g\ms B_jU(g)^*\f}{bu_g\ms B_jU(g)^*\psi}\,d\nu_\Om(\ovl g)
\end{equation}
for all $X\in\mc B(G)$, $\f,\,\psi\in\mb V$, and $b\in\mc L(\mc K)$. Conversely, given a map $\Gamma:\,\mc L(\mc K)\times\mc B(\Om)\to\mc L(\mc V)$ defined as in (\ref{Icovchar}) where the operators $\ms B_j:\mb V\to\mc K$ satisfy conditions (\ref{Hinv}) and (\ref{ehto}), then $\Gamma\in\mc I_u^U(\Om)$.
\end{theor}

\begin{proof}
Let $\ms M\in\mc O_U(\Om)$ be the associate observable of $\Gamma$, i.e.,\ $\ms M(X)\equiv\Gamma(I_{\mc K},X)$. Let the operator $\Lambda:\mb V\to\mc M^0$ and the representation $\rho:\,H\to\mc U(\mc M^0)$ for $\ms M$ be as in Theorem \ref{extcovobsker} so that $\ms M$ has the covariant minimal dilation $(\mf K^\rho,\ms M^\rho,\tj^\rho,K)$ where $(K\f)(g)=\Lambda U(g)^*\f$ for all $\f\in\mb V$ and (almost) all $g\in G$. Fix a measurable section $s:\Om\to G$ and let $y^\rho:\,G\times\Om\to\mc U(\mc M)$, $(g,\om)\mapsto y^\rho(g,\om)=\rho\big(s(\om)^{-1}gs(g^{-1}\om)\big)$ be the corresponding Wigner rotation. By using the unitary map $U'$ of \eqref{Un kaava} we can define the isometry $K':\mc V\to L^2(\nu_\Om;\mc M^0)$, $(K'\f)(\ovl g):=y^\rho(g,\ovl g)\Lambda U(g)^*\f$ for all $g\in G$ and $\f\in\mb V$, and bring the imprimitivity system $(\mf K^\rho,\ms M^\rho,\tj^\rho)$ back to the form $\big(L^2(\nu_\Om;\mc M^0),\ms P,y^\rho)$ where $\ms P(X)\psi=\chi_X\psi$ and $(y^\rho_g\psi)(\om)=y^\rho(g,\om)\psi(g^{-1}\om)$ for all  $X\in\mc B(\Om)$, $\psi\in L^2(\nu_\Om;\mc M^0)$, $g\in G$, and $\nu_\Om$-a.a.\ $\om\in\Om$.

Assume that $\pi:\,H\to\mc U(\mc M)$ is a representation such that the entire instrument $\Gamma$ has the dilation described in Remark \ref{instrimprimitivity} and do the same trick as above for $(\mf K^\pi,\ms M^\pi,\tj^\pi)$. Hence, we get the isometry $Y':\mc V\to\mc K\otimes L^2(\nu_\Om;\mc M)\cong L^2(\nu_\Om;\mc K\otimes\mc M)$ by defining $(Y'\f)(\ovl g):=\big(I_\kil\otimes y^\pi(g,\ovl g)\big)(Y\f)(g)$ for all $\f\in\mc V$ and a.a.\ $g\in G$. Let $C:L^2(\nu_\Om;\mc M^0)\to L^2(\nu_\Om;\mc K\otimes\mc M)$ be the decomposable isometry of Propositions \ref{prop:Ckiet} and\ref{prop:lcsc} such that $Y'=CK'$ and $\ms A_j$, $j<r+1$, be the Kraus operators of \eqref{fiinollakraus} for the channel $\Phi_0$, $\Phi_0(b)=C_0^*(b\otimes I_{\mc M})C_0$, $b\in\mc L(\mc K)$. Define $\ms B_j:=\ms A_j\Lambda$. It follows from \eqref{fiinollasisalla} that, for all $\f,\,\psi\in\mb V$, $b\in\mc L(\mc K)$, and $X\in\mc B(\Om)$,
\begin{eqnarray*}
\sis{\f}{\Gamma(b,X)\psi}&=&\int_X\sis{(K'\f)(\ovl g)}{y^\rho(g,\ovl g)\Phi_0(u_g^*bu_g)y^\rho(g,\ovl g)^*(K'\psi)(\ovl g)}\,d\nu_\Om(\ovl g)\\
&=&\int_X\sis{\Lambda U(g)^*\f}{\Phi_0(u_g^*bu_g)\Lambda U(g)^*\psi}\,d\nu_\Om(\ovl g)\\
&=&\int_X\sum_{j=1}^r\sis{u_g\ms B_jU(g)^*\f}{bu_g\ms B_jU(g)^*\psi}\,d\nu_\Om(\ovl g).
\end{eqnarray*}
The conditions \eqref{Hinv} and \eqref{ehto} follow from the covariance condition \eqref{covarcondi} for $\Phi_0$ and the normalization of $\Gamma$, respectively, and the converse claim follows immediately.
\end{proof}

\subsection{Instruments covariant with respect to a square-integrable representation}

In this subsection, we investigate a particular application of the theory developed in the preceding section. We will notice that, in some physically motivated cases, the ``Kraus operators'' $\ms B_j$ appearing in the general form (\ref{Icovchar}) of a covariant instrument can be chosen to be bounded.

Throughout this section, let $\mc V$ and $\mc K$ be separable Hilbert spaces, $G$ a unimodular lcsc group, and $H$ its compact subgroup equipped with the Haar measure $\nu_H$ such that $\nu_H(H)=1$. Fix a Haar measure $\nu_G$ for $G$ and the  $G$-invariant Borel measure $\nu_\Om$ of $\Om=G/H$ such that (\ref{Hmitta}) holds. Assume that $W:G\to\mc U(\mc V)$ is a strongly continuous unitary representation with the property that there is a constant $d>0$ such that
$$
\int_G|\sis{\f}{W(g)\psi}|^2\,d\nu_G(g)=d
$$
for all unit vectors $\f,\,\psi\in\mc V$. Such a representation is called {\it square integrable}. Observables covariant with respect to a square-integrable representation have been studied, e.g.,\ in \cite{kiukas_etal2006}.

Let us fix another strongly continuous (projective) representation $G\ni g\mapsto u_g\in\mc U(\mc K)$ and concentrate on instruments $\Gamma\in\mc I_u^W(\Om)$. We have the following theorem as a direct consequence of Theorem \ref{covinstr} and results of \cite{kiukas_etal2006}.

\begin{theor}\label{squareint}
For any $\Gamma\in\mc I_u^W(\Om)$, there are Hilbert-Schmidt operators $\ms B_j:\mc V\to\mc K$, $j<r+1$, $r\in\mb N\cup\{\infty\}$, satisfying the condition (\ref{Hinv}) (with $U$ replaced by $W$) such that
\begin{equation}\label{summaehto}
\sum_{j=1}^r\tr{\ms B_j^*\ms B_j}=1/d
\end{equation}
and
\begin{equation}\label{squareintchar}
\sis{\f}{\Gamma(b,X)\f}=\int_X\sum_{j=1}^r\sis{u_g\ms B_jW(g)^*\f}{bu_g\ms B_jW(g)^*\f}\,d\nu_\Om(\ovl g)
\end{equation}
for all $\f\in\mc V$, $b\in\mc L(\mc K)$, and $X\in\mc B(\Om)$.
\end{theor}

\begin{proof}
Let $\Gamma\in\mc I_u^W(\Om)$. The existence of the operators $\ms B_j$ defined on a dense $W$-invariant subspace $\mb V\subset\mc V$ such that (\ref{Hinv}) and (\ref{squareintchar}) hold  is already guaranteed by Theorem \ref{covinstr}. Hence, it suffices to show that condition (\ref{summaehto}) is met, from which it follows immediately that the operators $\ms B_j$ extend to Hilbert-Schmidt operators.

Denote the observable marginal $\Gamma(I_{\mc K},\cdot)$ of $\Gamma$ by $\ms M$. Since $\ms M\in\mc O^W(\Om)$, we may associate to it, as discussed in Remark \ref{imprimitivity}, an observable $\ovl{\ms M}\in\mc O^W(G)$. According to \cite[Theorem 3]{kiukas_etal2006}, there is a positive operator $\ovl S$ of trace 1 on $\mc V$ such that $\ovl{\ms M}(Z)=d^{-1}\int_ZW(g)\ovl SW(g)^*\,d\nu_G(g)$ (weakly) for all $Z\in\mc B(G)$. Defining a new trace-1 positive operator $S:=\int_HW(h)\ovl SW(h)^*\,d\nu_H(h)$ (weakly), we may write for any $\f\in\mb V$ and $X\in\mc B(\Om)$
\begin{eqnarray*}
\int_X\sum_{j=1}^r\sis{\ms B_jW(g)^*\f}{\ms B_jW(g)^*\f}\,d\nu_\Om(\ovl g)&=&\sis{\f}{\ms M(X)\f}\\
&=&\frac{1}{d}\int_X\sis{W(g)^*\f}{SW(g)^*\f}\,d\nu_\Om(\ovl g).
\end{eqnarray*}
This implies that, for each $\f\in\mb V$, there exists a null set $N_\f\subset G$ such that 
\begin{equation}\label{qjvjsivhjeioe2}
\sum_{j=1}^r\sis{\ms B_jW(g)^*\f}{\ms B_jW(g)^*\f}=\frac{1}{d}\sis{W(g)^*\f}{SW(g)^*\f}
\end{equation}
for all $g\in G\setminus N_\f$. Let $\mb D\subset\mb V$ be a countable set dense in $\mc V$ and define a set $N:=\cup_{\f\in\mb D}N_\f$ of Haar measure zero. Thus \eqref{qjvjsivhjeioe2} holds for all $g\in G\setminus N$ and $\f\in\mb D$. Since the both sides of \eqref{qjvjsivhjeioe2} can be interpreted as positive quadratic (and thus sesquilinear) forms on $\mb V$ and the form on the right hand side is bounded, it follows that $\sum_{j=1}^r\ms B_j^*\ms B_j=d^{-1}S$ and, for any $j<r+1$, the operator $\ms B_j$ is Hilbert-Schmidt. This proves the claim. 
\end{proof}

Suppose that $\Gamma\in\mc I_u^W(G)$ and $\ms B_j$'s are as in Theorem \ref{covinstr}. As we saw in the proof of Theorem \ref{squareint}, there is a trace-1 positive operator $S$ on $\mc V$ commuting with the subrepresentation $H\ni h\mapsto W(h)\in\mc U(\mc V)$ such that
$$
\sum_{j=1}^r\ms B_j^*\ms B_j=\frac1{d}S,
$$
and the observable $\ms M\in\mc O^W(\Om)$ associated with the instrument $\Gamma$ is given by
\begin{equation}\label{sqsuure}
\sis{\f}{\ms M(X)\f}=\frac1{d}\int_X\sis{W(g)^*\f}{SW(g)^*\f}\,d\nu_H(\ovl g),
\qquad \f\in\mc V,\quad X\in\mc B(\Om).
\end{equation}

Theorem \ref{squareint} tells us that, whenever we have some $\ms M\in\mc O^W(G)$ with the trace-class operator $S$ such that (\ref{sqsuure}) holds, we can determine any $\ms M$-compatible instrument $\Gamma\in\mc I_u^W(G)$ (associated to a  representation $G\ni g\mapsto u_g\in\mc U(\mc K)$) by using a decomposition $S=d\sum_{j=1}^r\ms B_j^*\ms B_j$ with Hilbert-Schmidt operators $\ms B_j:\,\mc V\to\mc K$ satisfying \eqref{Hinv} and defining $\Gamma$ according to (\ref{squareintchar}).

\begin{ex}\rm
In this example, we consider the covariant phase-space instruments associated with covariant phase-space observables. The phase space of an $n$-dimensional physical system is modelled by $\mb R^{2n}$. We denote the elements of $\mb R^{2n}$ by pairs $(\vek q,\vek p)$ of $\mb R^n$-vectors. The Hilbert space $\mc V$ associated with the system is $L^2(\mb R^n)$ (equipped with the $n$-fold Lebesgue measure). As a symmetry group, one has the Heisenberg group $G=\mb R^{2n}\times\mb T$ with the group law
$$
(\vek q,\vek p,t)(\vek q',\vek p',t')=\big(\vek q+\vek q',\,\vek p+\vek p',\,tt'e^{\frac{i}{2}(\vek p\cdot\vek q'-\vek q\cdot\vek p')}\big).
$$
The subgroup $H$ consists of elements $(\bs 0,\bs 0,t)$, $t\in\mb T$, so that $\Om=\mb R^{2n}$.

We define the {\it Weyl representation} $W:\,G\to\mc U\big(L^2(\mb R^n)\big)$ inducing the phase-space translations through
$$
\big(W(\vek q,\vek p,t)\f\big)(\vek x)=\ovl te^{\frac{i}{2}\vek q\cdot\vek p}e^{i\vek p\cdot\vek x}\f(\vek x+\vek q)
$$
for all $\f\in L^2(\mb R^n)$ and a.a.\ $\vek x\in\mb R^n$. The Weyl representation is square integrable with $d=(2\pi)^n$. In fact, this representation is also irreducible. Note that all operators on $L^2(\mb R^n)$ commute with the subrepresentation $\mb T\ni t\mapsto W(\bs 0,\bs 0,t)=\ovl tI_{L^2(\mb R^n)}$.

The elements of $\mc O^W(\Om)$ (resp.\ $\mc I_W^W(\Om)$) are called as {\it covariant phase-space observables} (resp.\ {\it instruments}). Any such observable $\ms M_S$ is defined by a positive trace-1 operator $S$ on $L^2(\mb R^n)$ and the formula
$$
\ms M_S(X)=(2\pi)^{-n}\int_XW_0(\vek q,\vek p)SW_0(\vek q,\vek p)^*\,d\vek q\,d\vek p,\quad X\in\mc B(\mb R^{2n}),
$$
(as in (\ref{sqsuure})) where $W_0(\vek q,\vek p):=W(\vek q,\vek p,1)$ for all $\vek q,\,\vek p\in\mb R^n$ and the integral is defined $\sigma$-weakly. Now the structure of any instrument $\Gamma\in\mc I_W^W(\Om)$ whose associate observable is $\ms M_S$ is determined by a decomposition $S=(2\pi)^n\sum_{j=1}^r\ms B_j^*\ms B_j$, $r\in\mb N\cup\{\infty\}$ (where $\ms B_j\in\mc L\big(L^2(\mb R^n)\big)$ are Hilbert-Schmidt operators) and
$$
\Gamma(b,X)=(2\pi)^{-n}\int_X\sum_{j=1}^rW_0(\vek q,\vek p)\ms B_j^*W_0(\vek q,\vek p)^*bW_0(\vek q,\vek p)\ms B_jW_0(\vek q,\vek p)^*\,d\vek q\,d\vek p
$$
($\sigma$-weakly) for all $X\in\mc B(\mb R^{2n})$ and $b\in\mc L\big(L^2(\mb R^n)\big)$.
\end{ex}

\section{Conclusions and discussion}

In this paper, we have studied covariant positive (sesquilinear form valued) kernels in the context of modules and characterized (essentially all of) the extreme points of convex sets of such maps. Particularly, we have concentrated on completely positive maps  in the case where the algebras involved are $W^*$-algebras. A physically relevant corollary of these results is the characterization of the extremal (generalized) covariant quantum instruments. We have seen that covariant instruments can be dilated into canonical systems of imprimitivity when the value space of the instruments is a transitive space of a lcsc symmetry group. Finally, we have discussed in more detail the standard covariant quantum observables and instruments whose value space is a transitive $G$-space $G/H$ where the stability subgroup $H$ is compact and the symmetry group $G$ is unimodular and of type I. Theorems \ref{extcovobsker} and \ref{covinstr} give the general structure for such covariant observables and instruments.

It should be pointed out that we can generalize Theorem \ref{covinstr} for covariant instruments $\Gamma\in\mc I_u^U(\Om)$ (with separable spaces $\mc V$ and $\mc K$) whenever $\Om$ is a transitive space of a lcsc group $G$ and the associated observables $\ms M\in\mc O^U(\Om)$ have the structure as in \eqref{kolmogorov} with a dense $U$-invariant domain $\mb V\subset\mc V$ and a linear map $\Lambda:\mb V\to\mc M^0$ intertwining $U|_H$ with some strongly continuous representation $\rho:H\to\mc U(\mc M^0)$. Indeed, as in the proof of Theorem \ref{covinstr}, using Proposition \ref{prop:lcsc} and \eqref{fiinollasisalla}, one obtains the structure \eqref{Icovchar} for any such $\Gamma\in\mc I_u^U(\Om)$. A particular example where these requirements are met is the case of a lcsc Abelian symmetry group and a general transitive value space, see \cite{CaDeTo2004, haapasalo}. Thus, especially, when $U:\,G\to\mc U(\mc V)$ and $G\ni g\mapsto u_g\in\mc U(\mc K)$ are strongly continuous unitary representations of a lcsc Abelian group $G$ and $\Om$ is any transitive $G$-space, each $\Gamma\in\mc I_u^U(\Om)$ has the structure given in \eqref{Icovchar}. The conditions on $\mc O^U(G)$ (note the trivial stability subgroup) are also satisfied in the case of a non-unimodular lcsc symmetry group and an irreducible (projective) unitary representation $U:G\to\mc U(\mc V)$ \cite{kiukas2006}, meaning that, for any strongly continuous unitary representation $G\ni g\mapsto u_g\in\mc U(\mc K)$, we have a similar structure theorem as Theorem \ref{covinstr} for covariant instruments $\Gamma\in\mc I_u^U(G)$.

To extend the results of Section \ref{covariantobs} to the general case of a lcsc (possibly non-unimodular) symmetry group $G$ (of type I) and a value space $G/H$ with a non-compact stability subgroup $H$, some issues have to be taken into account. For example, the Fourier-Plancherel transformation played the central role in our characterization. However, partial solutions exist, as shown in \cite{CaDeTo2004, haapasalo} in the case of Abelian symmetries, and in \cite{kiukas2006} in the case of a non-unimodular group and an irreducible representation. In our study, the stability subgroup was compact just for the sake of convenience and for the direct link to observables and instruments whose value space is the symmetry group itself (as in \cite{holevotypeI,holevo98}) but one could expect that a generalization to the case of a non-compact subgroup exists. This is, indeed, a natural and physically well-motivated direction for future research on covariant quantum devices. Moreover, there is a possibility of studying observables and instruments with value spaces that are not transitive spaces of the symmetry group. The problem in these settings is how to (measurably) `stitch up' such a value space from the transitive orbits of the group, since the structure of covariant CP maps on individual orbits of positive measure can essentially be solved, as showed in Section \ref{quantumCP}. Unfortunately, the orbits are usually zero measurable and the direct reduction to orbits fails \cite{varadarajan}.

It still remains to be studied to what extent the results obtained in Section \ref{covariantobs} can be generalized to the case of a covariant CP map defined in the context of a general $\mc A$-module $\mb V$ instead of a simple Hilbert space $\mc V$ (a $\mb C$-module). An obvious problem that arises, when trying to generalize Theorem \ref{covinstr} into this framework, is the proper definition of a dense submodule $\mb V_0\subset\mb V$ such that, for any normal covariant CP-map $S:\mc L(\mc K)\otimes L^\infty(\nu)\to S_{\mc L(\hil)}(\mb V)$, one can write
$$
S_{b\otimes f}(v,w)=\int_\Om f(\ovl g)\sum_{j=1}^r\big(u_g\ms B_jU(g^{-1})v\big)^*bu_g\ms B_jU(g^{-1})w\,d\nu(\ovl g)
$$
($\sigma$-strongly) for all $b\in\mc L(\mc K)$, $f\in L^\infty(\nu)$, and $v,\,w\in\mb V_0$. Obtaining such a structure result for covariant CP-maps seems to depend a lot on the structure of the module $\mb V$ which is, in this paper, not fixed. One might think that in particular examples of the module $\mb V$ this result should be valid.

\section*{Appendix: An auxiliary lemma}

We say that two measures $\nu$ and $\nu'$ on a measurable space $(\Om,\Sigma)$ are equivalent and write $\nu\sim\nu'$ if $\nu$ and $\nu'$ are absolutely continuous with respect to each other; we let $d\nu'/d\nu$ denote the Radon-Nikod\'ym derivative of $\nu'$ with respect to $\nu$. For any $\sigma$-isomorphism $T:\Om\to\Om$, i.e.,\ invertible $\Sigma$-measurable map $T:\Om\to\Om$ such that $T^{-1}$ is also $\Sigma$-measurable, we denote the measure $X\mapsto\nu\big(T(X)\big)$ by $\nu_T$.

\begin{lemma}\label{covariantoperators}
Suppose that $T:\,\Om\to\Om$ is a $\sigma$-isomorphism and $\nu_T\sim\nu$. For an operator $B\in\mc L(\hil_\oplus)$, $\hil_\oplus=\int_\Om^\oplus\hil_{n(\om)}\,d\nu(\om)$, the condition
$$
B\widehat{\chi_X}=\widehat{\chi_{T(X)}}B
$$
holds for all $X\in\Sigma$ if and only if there is a weakly $\nu$-measurable\footnote{That is, for all $\f,\,\psi\in\hil_\oplus$, the maps $\Om\ni\omega\mapsto\langle\f(\omega)|B_T(\omega)\psi(T^{-1}(\omega))\rangle\in\mb C$ are $(\Sigma_\nu,\mc B(\mb C))$-measurable;
here $\Sigma_\nu$ is the $\nu$-completion of $\Sigma$ and $\mc B(\Gamma)$ denotes the Borel $\sigma$-algebra of any topological space $\Gamma$.} field $\om\mapsto B_T(\om)$ of operators $B_T(\om):\hil_{n(T^{-1}(\om))}\to\hil_{n(\om)}$ such that
$$
\nu\mr{-ess}\,\sup_{\om\in\Om}\|B_T(\om)\|<\infty
$$
and
\begin{eqnarray}
(B\psi)(\om)&=&B_T(\om)\psi\big(T^{-1}(\om)\big)\sqrt{\frac{d\nu_{T^{-1}}(\om)}{d\nu(\om)}},\label{covariantoperators1}\\
(B^*\psi)(\om)&=&B_T\big(T(\om)\big)\psi\big(T(\om)\big)\sqrt{\frac{d\nu_T(\om)}{d\nu(\om)}}\label{covariantoperators2}
\end{eqnarray}
for all $\psi\in\hil_\oplus$ and $\nu$-a.a.\ $\om\in\Om$. Furthermore, $B$ is unitary if and only if $n\big(T(\om)\big)=n(\om)$ and $B_T(\om):\hil_{n(\om)}\to\hil_{n(\om)}$ is unitary for $\nu$-a.a.\ $\om\in\Om$.
\end{lemma}

\begin{proof}
Define operators $K_T,\,K_{T^{-1}}\in\mc L\big(L^2(\nu)\otimes\hil_\infty)$ through
\begin{eqnarray*}
(K_T\psi)(\om)&=&\psi\big(T(\om)\big)\sqrt{\frac{d\nu_T(\om)}{d\nu(\om)}},\\
(K_{T^{-1}}\psi)(\om)&=&\psi\big(T^{-1}(\om)\big)\sqrt{\frac{d\nu_{T^{-1}}(\om)}{d\nu(\om)}}
\end{eqnarray*}
for all $\psi\in L^2(\nu)\otimes\hil_\infty$ and $\nu$-a.a.\ $\om\in\Om$. These operators are unitary; a direct calculation shows that $\|K_T\psi\|=\|\psi\|$ for all $\psi\in L^2(\nu)\otimes\hil_\infty$ and $K_T^*=K_{T^{-1}}=(K_T)^{-1}$.

The direct integral space $\hil_\oplus$ can be viewed as a closed subspace of $L^2(\nu)\otimes\hil_\infty$ (see, Remark \ref{vonNeumann2}). Indeed, for all $n\in\{0\}\cup\mb N\cup\{\infty\}$, let $Q_n:\,\hil_n\to\hil_\infty$ be an isometry and define a decomposable isometry $Q:\,\hil_\oplus\to L^2(\nu)\otimes\hil_\infty$ via $(Q\psi)(\omega):=Q(\omega)\psi(\omega)$, $\psi\in\hil_\oplus$, where $Q(\omega):=Q_{n(\omega)}$. Let $\tilde B=QBQ^*\in\mc L\big(L^2(\nu)\otimes\hil_\infty\big)$. For each $f\in L^\infty(\nu)$, we consider $\hat f$ as a multiplicative operator of both $\hil_\oplus$ and $L^2(\nu)\otimes\hil_\infty$, i.e.,\ we identify $Q\hat fQ^*$ with $\hat f$. Direct calculation shows that $(\tilde BK_T)\widehat{\chi_X}=\widehat{\chi_X}(\tilde BK_T)$ for all $X\in\Sigma$. This implies that the operator $\tilde B_T=\tilde BK_T$ is decomposable, $\tilde B_T=\int_\Om^\oplus\tilde B_T(\om)\,d\nu(\om)$ \cite[Theorem 1, p.\ 187]{Dix2}. Since $Q^*Q=I_{\hil_\oplus}$ and thus $B=Q^*\tilde BQ$ we have
\begin{eqnarray*}
(B\psi)(\om)&=&(Q^*\tilde B_TK_T^*Q\psi)(\om)=Q^*(\om)(\tilde B_TK_T^*Q\psi)(\om)\\
&=&Q^*(\om)\tilde B_T(\om)(K_T^*Q\psi)(\om)\\
&=&Q^*(\om)\tilde B_T(\om)(Q\psi)\big(T^{-1}(\om)\big)\sqrt{\frac{d\nu_{T^{-1}}(\om)}{d\nu(\om)}}\\
&=&Q^*(\om)\tilde B_T(\om)Q\big(T^{-1}(\om)\big)\psi\big(T^{-1}(\om)\big)
\sqrt{\frac{d\nu_{T^{-1}}(\om)}{d\nu(\om)}}
\end{eqnarray*}
for all $\psi\in\hil_\oplus$ and $\nu$-a.a.\ $\om\in\Om$, so that, when we denote
$$
Q^*(\om)\tilde B_T(\om)Q\big(T^{-1}(\om)\big)=B_T(\om)
$$
for $\nu$-a.a.\ $\om\in\Om$, we obtain
(\ref{covariantoperators1}).

One easily sees that $B^*\widehat{\chi_X}=\widehat{\chi_{T^{-1}(X)}}B^*$ for all $X\in\Sigma$ which, according to the first part of this proof, implies that there is a weakly measurable field $\om\mapsto B'(\om)$ of bounded operators, $B'(\om):\hil_{n(T(\om))}\to\hil_{n(\om)}$ with $\nu\text{-ess}\,\sup_{\om\in\Om}\|B'(\om)\|<\infty$ such that
$$
(B^*\psi)(\om)=B'(\om)\psi\big(T(\om)\big)\sqrt{\frac{d\nu_T(\om)}{d\nu(\om)}}
$$
for all $\psi\in\hil_\oplus$ and $\nu$-a.a.\ $\om\in\Om$. Comparing expressions
\begin{eqnarray*}
\sis{\f}{B\psi}&=&\int_\Om\sis{\f(\om)}{B_T(\om)\psi\big(T^{-1}(\om)\big)}\sqrt{\frac{d\nu_{T^{-1}}(\om)}{d\nu(\om)}}\,d\nu(\om)\\
&=&\int_\Om\sis{\f\big(T(\om)\big)}{B_T\big(T(\om)\big)\psi(\om)}\sqrt{\frac{d\nu_T(\om)}{d\nu(\om)}}\,d\nu(\om)
\end{eqnarray*}
and
$$
\sis{B^*\f}{\psi}=\int_\Om\sis{B'(\om)\f\big(T(\om)\big)}{\psi(\om)}\sqrt{\frac{d\nu_T(\om)}{d\nu(\om)}}\,d\nu(\om)
$$
for all $\f,\,\psi\in\hil_\oplus$ yields $B'(\om)=B_T\big(T(\om)\big)^*$ for $\nu$-a.a.\ $\om\in\Om$. This proves (\ref{covariantoperators2}). The proof of the converse claim is straightforward.

It is easy to check that, for an operator $B\in\mc L(\hil_\oplus)$ of (\ref{covariantoperators1}), $B^*B=I_{\hil_\oplus}$ if and only if $B_T\big(T(\om)\big)^*B_T\big(T(\om)\big)=I_{\hil_{n(\om)}}$ for $\nu$-a.a.\ $\om\in\Om$. Moreover, $BB^*=I_{\hil_\oplus}$ if and only if $B_T(\om)B_T(\om)^*=I_{\hil_{n(\om)}}$ for $\nu$-a.a.\ $\om\in\Om$. Hence, $B\in\mc L(\hil_\oplus)$ is unitary if and only if $B_T(\om)$ is unitary almost everywhere and hence the spaces $\hil_{n(\om)}$ and $\hil_{n(T(\om))}$ are isomorphic.
\end{proof}

\end{document}